\documentclass[
	final%
	,12pt%
	,pagesize%
	,headings=small%
	,paper=a4%
	,parskip=false%
	,abstract=true
	,DIV=12
	,toc=bibliography
]{scrartcl}

\addtokomafont{sectioning}{\normalfont\bfseries}
\setkomafont{title}{\normalfont}
\setkomafont{subtitle}{\normalfont}
\setkomafont{author}{\normalfont}
\setkomafont{date}{\normalfont}
\addtokomafont{pageheadfoot}{\scshape\small}
\setkomafont{caption}{\footnotesize}
\setkomafont{captionlabel}{\usekomafont{caption}\bfseries}
\setcapindent{0pt}

\usepackage{etoolbox}
\usepackage{xspace}
\usepackage{xpatch}
\usepackage{ifthen}
\usepackage[normalem]{ulem}

\usepackage[utf8]{inputenc}
\usepackage[T1]{fontenc}
\usepackage{csquotes}
\usepackage[english]{babel}
\usepackage{graphicx}
\graphicspath{{Pictures/}{/}}
\usepackage[export]{adjustbox}

\usepackage[dvipsnames]{xcolor}
\usepackage{tikz-cd}

\usepackage[final]{fixme}
\usepackage[nointlimits]{amsmath}
\usepackage{amssymb,mathrsfs,mathtools}
\usepackage{newtxtext}
\usepackage[varvw]{newtxmath}

\usepackage{upgreek}
\usepackage{aliascnt} 
\usepackage[amsmath,thmmarks,hyperref]{ntheorem}
\usepackage[expansion=true,protrusion=true]{microtype}
\usepackage{xkeyval}
\usepackage{authblk}
\usepackage{enumitem}

\usepackage[final,%
	pdftex,%
	bookmarks,%
	bookmarksdepth=3,%
	breaklinks=true,%
	colorlinks=true,%
	urlcolor=NavyBlue,%
	linkcolor=NavyBlue,%
	citecolor=ForestGreen,%
]{hyperref}%
\usepackage[all]{hypcap}
\usepackage[capitalise,nameinlink]{cleveref}

\crefname{equation}{}{}

\newlist{claims}{enumerate}{10}
\setlist[claims]{label*=\arabic*.,ref=\arabic*}
\crefname{claimsi}{Claim}{Claim}
\Crefname{claimsi}{Claim}{Claims}

\newlist{conditions}{enumerate}{10}
\setlist[conditions]{label*=\arabic*.,ref=\arabic*}
\crefname{conditionsi}{condition}{conditions}
\Crefname{conditionsi}{Condition}{Conditions}

\AtBeginEnvironment{abstract}{\footnotesize}
\makeatletter
\patchcmd{\@maketitle}{\huge}{\Large}{}{}
\makeatother

\input{Macros}
\DeclareDocumentCommand{\Graph}{ O{} O{} o o}{
	\IfValueTF{#3}{
	  \IfValueTF{#4}{
	  	\mathrm{Graph}^{#1}_{#2}(#3;#4)
	  }{
		\mathrm{Graph}^{#1}_{#2}(#3)
	  }
	}{
		\mathrm{Graph}^{#1}_{#2}
	}
}

\DeclareDocumentCommand{\Emb}{ O{} O{} o o}{
	\IfValueTF{#3}{
	  \IfValueTF{#4}{
	  	\mathrm{Emb}^{#1}_{#2}(#3;#4)
	  }{
		\mathrm{Emb}^{#1}_{#2}(#3)
	  }
	}{
		\mathrm{Emb}^{#1}_{#2}
	}
}
\newcommand{\Embsp}{\Emb[s,p][][M][\AmbSpace]}

\DeclareDocumentCommand{\Sobo}{ O{} O{} o o}{
	\IfValueTF{#3}{
	  \IfValueTF{#4}{
	  	W^{#1}_{#2}(#3;#4)
	  }{
	  	W^{#1}_{#2}(#3)
	  }
	}{
	  W^{#1}_{#2}
	}
}

\DeclareDocumentCommand{\Bessel}{ O{} O{} o o}{
	\IfValueTF{#3}{
	  \IfValueTF{#4}{
	  	H^{#1}_{#2}(#3;#4)
	  }{
	  	H^{#1}_{#2}(#3)
	  }
	}{
	  H^{#1}_{#2}
	}
}

\DeclareDocumentCommand{\Besov}{ O{} O{} o o}{
	\IfValueTF{#3}{
	  \IfValueTF{#4}{
	  	B^{#1}_{#2}(#3;#4)
	  }{
	  	B^{#1}_{#2}(#3)
	  }
	}{
	  B^{#1}_{#2}
	}
}

\DeclareDocumentCommand{\Holder}{ O{} O{} o o}{
	\IfValueTF{#3}{
	  \IfValueTF{#4}{
	  	C^{#1}_{#2}(#3;#4)
	  }{
	  	C^{#1}_{#2}(#3)
	  }
	}{
	  C^{#1}_{#2}
	}
}

\DeclareDocumentCommand{\Lebesgue}{ O{} O{} o o}{
	\IfValueTF{#3}{
	  \IfValueTF{#4}{
	  	L^{#1}_{#2}(#3;#4)
	  }{
	  	L^{#1}_{#2}(#3)
	  }
	}{
	  L^{#1}_{#2}
	}
}

	\newcommand{\mynewtheorem}[4] 
{
\newaliascnt{#1}{#2}
\newtheorem{#1}[#1]{#3}
\aliascntresetthe{#1}
\expandafter\def\csname #1autorefname\endcsname{%
#4%
}%
}

\newtheorem{theorem}{Theorem}[section]
\mynewtheorem{lemma}{theorem}{Lemma}{Lemma} 
\mynewtheorem{proposition}{theorem}{Proposition}{Proposition} 
\mynewtheorem{corollary}{theorem}{Corollary}{Corollary}
\mynewtheorem{quest}{theorem}{Question}{Question}
\mynewtheorem{conjecture}{theorem}{Conjecture}{Conjecture}

\mynewtheorem{problem}{theorem}{Problem}{Problem}

\theoremstyle{break}
\mynewtheorem{btheorem}{theorem}{Theorem}{Theorem} 
\mynewtheorem{blemma}{theorem}{Lemma}{Lemma}
\mynewtheorem{bproposition}{theorem}{Proposition}{Proposition}
\mynewtheorem{bcorollary}{theorem}{Corollary}{Corollary}
\mynewtheorem{bproblem}{theorem}{Problem}{Problem}

\theoremstyle{plain}
\theorembodyfont{\normalfont}
\mynewtheorem{definition}{theorem}{Definition}{Definition}
\mynewtheorem{example}{theorem}{Example}{Example}
\mynewtheorem{remark}{theorem}{Remark}{Remark}

\theoremstyle{break}
\mynewtheorem{bdefinition}{theorem}{Definition}{Definition}
\mynewtheorem{bex}{theorem}{Example}{Example}
\mynewtheorem{brem}{theorem}{Remark}{Remark}
\theoremheaderfont{\itshape}
\theorembodyfont{\upshape}
\theoremstyle{nonumberplain}
\theoremseparator{.}
\theoremsymbol{\ensuremath{\Box}}
\newtheorem{proof}{Proof}

\newtheorem{hproof}{Hidden proof}

{%
	\ifthenelse{\isundefined{\showhiddenproofs}}%
		{\WouldNotSayItLikeThis{Proof hidden.}\expandafter\comment}%
		{\begin{hproof}}%
}%
{%
	\ifthenelse{\isundefined{\showhiddenproofs}}%
		{\expandafter\endcomment}%
		{\end{hproof}}%
}

\title{Inverse obstacle scattering regularized by the tangent-point energy}

\author[2]{Henrik Schumacher}
\author[1]{Jannik R\"onsch}
\author[1]{Thorsten Hohage}
\author[1]{Max Wardetzky}

\affil[1]{Institute for Numerical and Applied Mathematics, Georg-August-University Göttingen}
\affil[2]{Institute for Mathematics, RWTH Aachen University}

\date{\vspace{-3ex}\today}

\begin{document}
\maketitle
\vspace{-6ex}
\begin{abstract}

We employ the so-called tangent-point energy as Tikhonov regularizer for ill-conditioned inverse scattering problems in $3D$.
The tangent-point energy is a self-avoiding functional on the space of embedded surfaces that also penalizes surface roughness. 
Moreover, it features nice compactness and continuity properties. 
These allow us to show the well-posedness of the regularized problems and the convergence of the regularized solutions to the true solution in the limit of vanishing noise level.
We also provide a reconstruction algorithm of iteratively regularized Gauss--Newton type.
Our numerical experiments demonstrate that our method is numerically feasible and effective in producing reconstructions of unprecedented quality.
\end{abstract}

\section{Introduction}\label{sec:Introduction}

The idea of gathering information about an unknown physical object by investigating its wave-scattering behavior---namely, ``inverse scattering problems''---finds application in many physical real-world applications such as radar, sonar, medical imaging, and non-destructive testing.
In this paper we focus on the classic ``three-dimensional acoustic inverse obstacle scattering problem'', but regularization by the tangent-point energy may be applied to any inverse obstacle 
problem. 
The goal is to determine the \emph{shape} of a three-dimensional scatterer from noisy far field measurements of scattered waves induced by (known) incident plane waves.
It is a well-known fact that this problem is ill-posed in the sense of Hadamard, i.e., the solution does not depend continuously on the far field data (see \cite[Theorem 2.17]{zbMATH06061716}).
In particular, this means that small perturbations of the data by noise might induce a huge error in the reconstruction.
Thus, one needs to employ a suitable regularization method in order to control the error induced by the data noise and to meaningfully reconstruct the scatterer from the measurements.

Let us briefly discuss some existing methods for solving inverse obstacle scattering problems. One broad class of methods are sampling methods, which do 
not provide a parameterization of the unknown obstacle but give us some data-based criterion 
whether a given point in space belongs to the unknown obstacle. First representatives 
of such methods have been proposed in \cite{David_Colton_1996,sampling1,sampling2}. 
While easy to implement, such methods usually yield rather inaccurate reconstructions, 
they require a lot of data, and they impose strong restrictions on the type of data. 
Two classes of methods that yield an explicit representation (e.g., a parameterization) 
of the unknown surface are decomposition methods 
\cite{coltonmonk,coltonmonk2,KIRSCH1987279,kress_zinn} and 
iterative regularization methods \cite{Thorsten_Hohage_1997,KR:94}.  
For such methods the representation of surfaces---namely the choice of a  suitable ``space of shapes''--- becomes a difficulty. Another issue is the choice of a ``natural'' smoothness 
penalty that is invariant under reparameterizations and rigid body motions.  
One common simplifying assumption consists in considering star-shaped obstacles by making use of their representation as normal-offsets of a sphere and Sobolev norms of the normal-offset as 
smoothness penalties. In a previous work of the last three authors with Eckhard and 
Hiptmair (\cite{Eck_article}, see also \cite{Eck}) it has been demonstrated that the 
bending energy of closed curves in the plane is a natural smoothness penalty allowing 
the reconstruction of complicated non-star-shaped obstacles and improving accuracy for 
star-shaped obstacles. However, this approach has no straightforward generalization to 
surfaces embedded in $\mathbb{R}^3$, and it does not prevent self intersections. 
Approaches to reconstruct more general (topologically spherical but non-star shaped) surfaces  
were suggested by Farhat, Tezaur and Djellouli \cite{farhat} as well as Hohage and Harbrecht \cite{zbMATH05251341}. The method in 
\cite{farhat} is restricted to shape classes that can be represented by a small number 
of parameters and uses the squared $\ell^2$-norm of the parameter vector as penalty functional. 
In contrast, \cite{zbMATH05251341} employs a squared Sobolev-norm of a diffeomorphism between the unit sphere and the unknown shape as penalty functional supplemented by an ad-hoc term designed to prevent local degeneration. Neither approach is invariant under reparameterization, and neither 
guarantees the avoidance of self-intersections. 
More recently, Chen, Jin, and Liu \cite{chen} proposed a reconstruction method based on machine learning.
However, the reconstruction of detailed, non-star-shaped obstacles, such as the ``Stanford bunny'' presented in \cref{fig:Bunny} and obstacles of higher genus, demand further work.


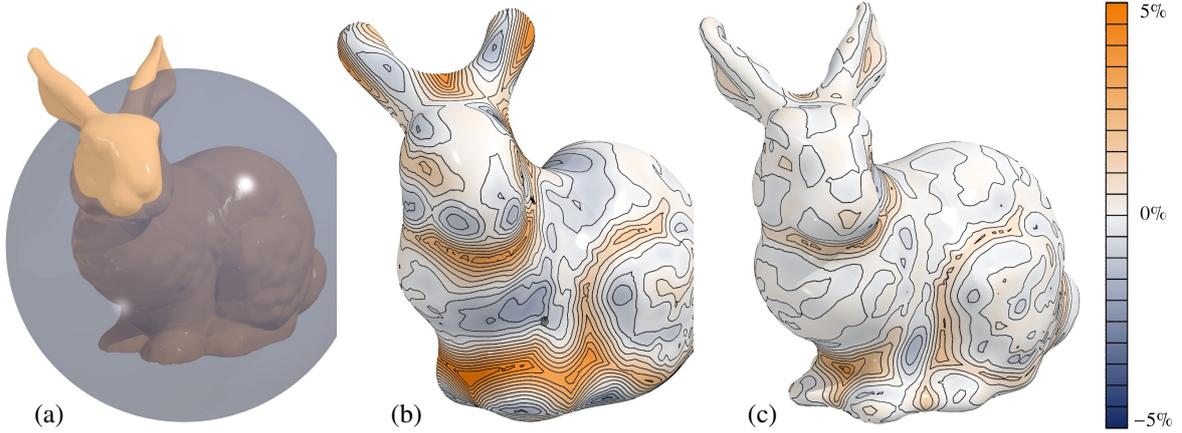
\begin{figure}[t]
    \centering
    \begin{tikzpicture}%
        \node[inner sep=0pt] (fig) at (0,0) {\includegraphics[	
            trim = 0 0 0 0, 
            clip = true,  
            angle = 0,
            width = 0.34\textwidth
        ]{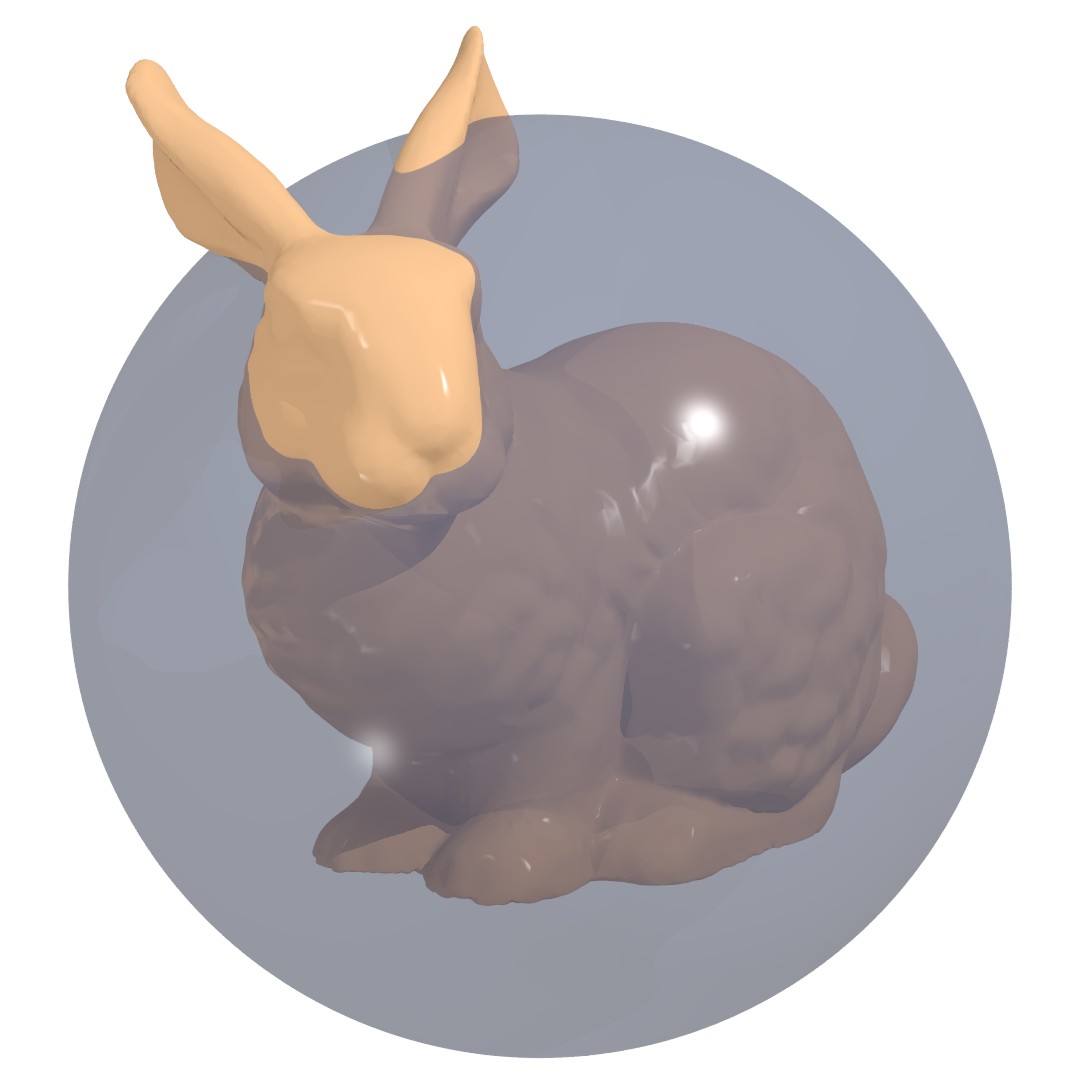}};%
        \node[above = 1ex, right= 3ex] at (fig.south west) {\begin{footnotesize}(a)\end{footnotesize}};%
    \end{tikzpicture}%
    \hspace{-0.042\textwidth}%
    \begin{tikzpicture}%
        \node[inner sep=0pt] (fig) at (0,0) {\includegraphics[	
            trim = 80 20 80 80, 
            clip = true,  
            angle = 0,
            width = 0.34\textwidth
        ]{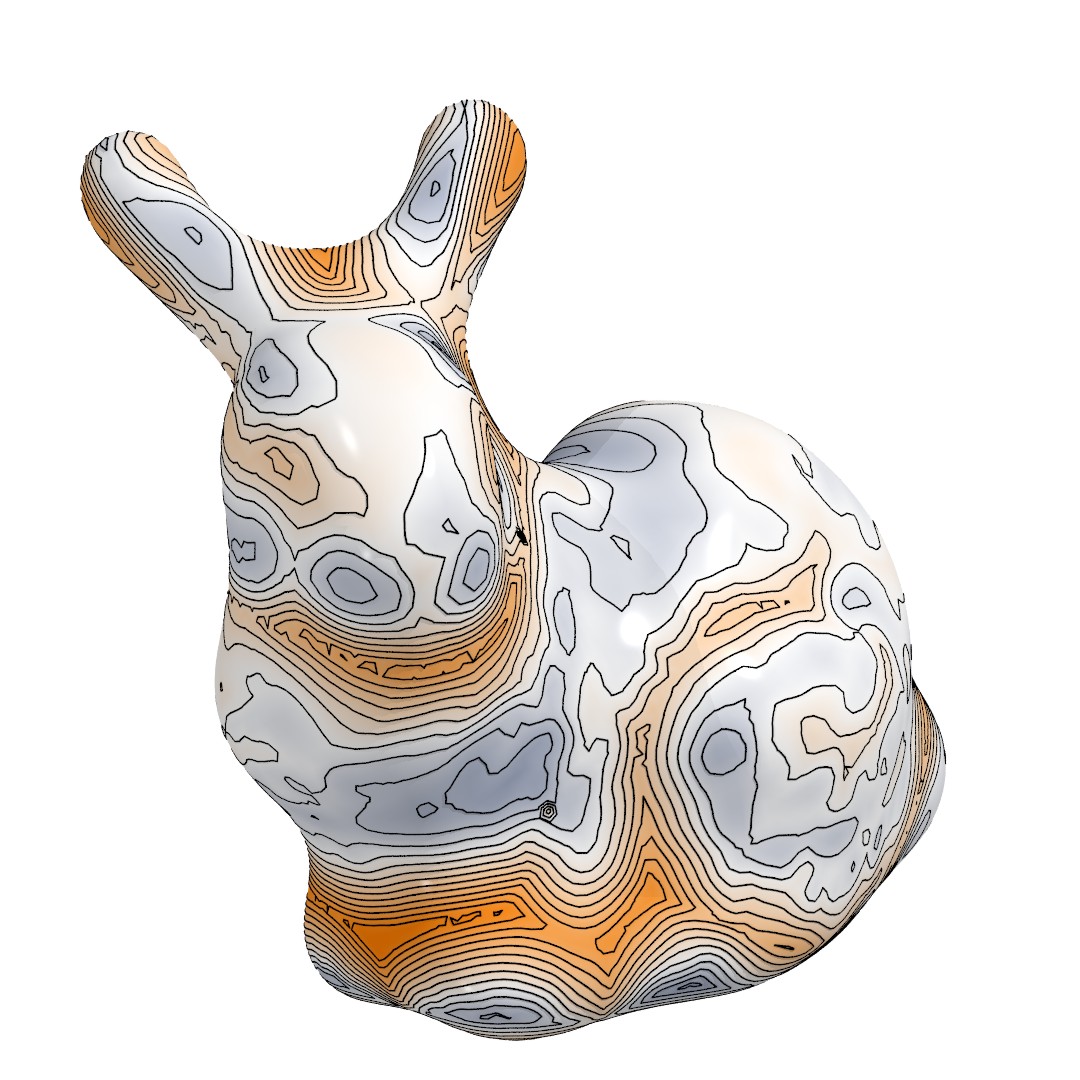}};%
        \node[above = 1ex, right= 3ex] at (fig.south west) {\begin{footnotesize}(b)\end{footnotesize}};%
    \end{tikzpicture}%
    \hspace{-0.042\textwidth}%
    \begin{tikzpicture}%
        \node[inner sep=0pt] (fig) at (0,0) {\includegraphics[	
            trim = 80 20 80 80, 
            clip = true,  
            angle = 0,
            width = 0.34\textwidth
        ]{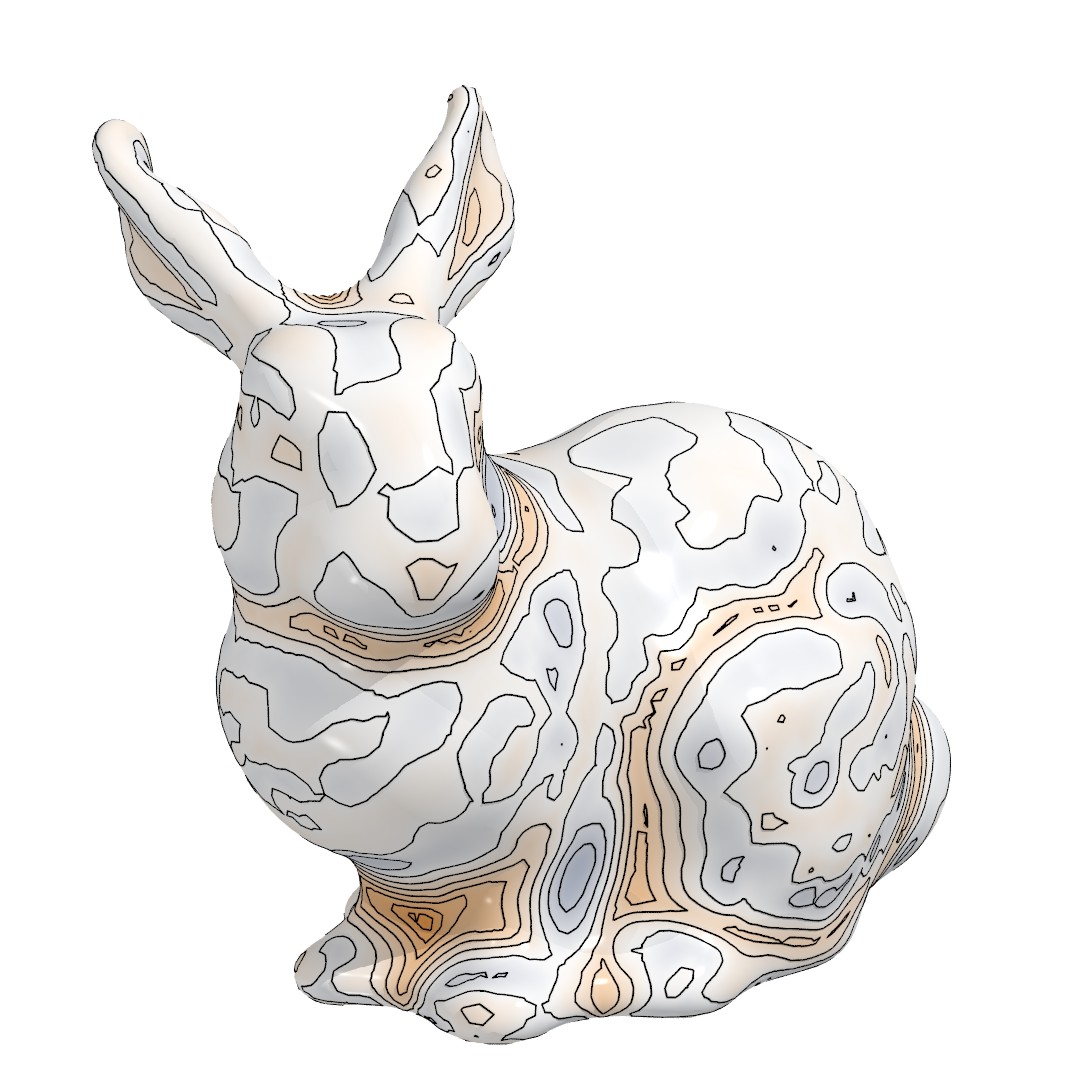}};%
        \node[above = 1ex, right= 3ex] at (fig.south west) {\begin{footnotesize}(c)\end{footnotesize}};%
    \end{tikzpicture}%
    \includegraphics[	
        trim = 0 0 0 0, 
        clip = true,  
        angle = 0,
        width = 0.0625\textwidth
    ]{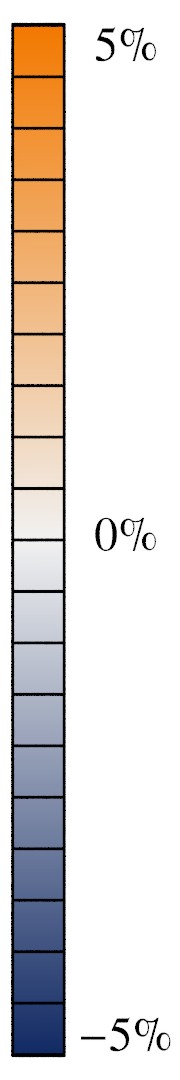}%
    \caption{%
        (a)~The original Stanford bunny together with the sphere we used as initial guess.
        The synthetic far field data was perturbed by $1\%$ Gaussian white noise.
        (b)~A rough reconstruction, shaded by the signed distance field of the true obstacle.
        We used wavelengths $\lambda \in \braces{1, 2}$, each in $16$ wave directions, roughly evenly distributed over the unit sphere.
        (c)~The final reconstruction (16 waves with $\lambda = 1/2$), also shaded by the signed distance field of the true obstacle.
        This state was reached from (b) in $32$ Gauss--Newton steps with discrepancy parameter $\tau=2$ (see \cref{sec:IRGNM} for the meaning of~$\tau$).
    }%
    \label{fig:Bunny}
\end{figure} 

\newpage

Our approach to the reconstruction problem is to consider the well-known \emph{(generalized) Tikhonov regularization}, i.e., to stabilize the minimization of 
\[
    \cF(f)\ceq \norm{F(f)-g^\delta}^2_{\Lebesgue[2]}
\]
over some Banach space $X$ of parameterizations, where $F$ denotes the boundary-to-far field map.
The main novelty of our approach is to use the so-called \emph{tangent-point energy} $\Energy$ as a stabilizing functional. This geometric surface energy has the following properties:
\begin{enumerate}[noitemsep]
    \item It enforces $C^{1,\beta}$-regularity of the reconstruction, $0 < \beta < 1$.
    \item It is repulsive, i.e., it avoids self intersections.
    \item It is relatively well-investigated  both in the analytical and in the numerical sense (cf. \cite{zbMATH06214305,MR3331696,zbMATH07046194,zbMATH06017037,zbMATH06203565,RepulsiveS,RepulsiveC}).
\end{enumerate}
For a more detailed discussion of the tangent-point energy, we refer to~\cref{sec:TPE}.
The choice of using a geometric energy as a regularizer is motivated by the approach of \cite{Eck} and \cite{Eck_article} for the two-dimensional problem.
In \cref{sec:ExRegProp} we show that the tangent-point energy is indeed a suitable regularizer in the sense of generalized Tikhonov regularization.
On the computational side we additionally require an efficient way of minimizing functions of the form $\cJ(f) = \cF(f) + \alpha \, \Energy(f) $, where $\alpha > 0$. 
It turns out that $\Energy$ dominates $\cF$ in the sense that the domain of $\Energy$ is contained in the domain of $\cF$ and $\cF$ is completely continuous with respect to the topology of the energy space of $\cE$.
Informally speaking, this means that any sufficiently good optimization method for $\Energy$ needs only small adaptation to work also for $\cJ(f)$.

Let us consider the problem of optimizing $\Energy$ for a moment. To this end it might appear natural to employ Newton's method.
However, there are several issues that complicate this approach. In particular, one has to work in an infinite-dimensional Banach space $X$, which entails several challenges:
\begin{enumerate}[noitemsep]
    \item The second derivatives $DD \Energy(f) \colon X \to X\dual$ might be difficult to compute, such that numerically solving $D D\Energy(f) \, v = - D\Energy(f)$ might become computationally infeasible.
    \item For nonconvex problems such as ours, the Hessian $D^2 \Energy(f) \colon X \times X \to \R$ might not be positive-definite, i.e., $v$ is not guaranteed to be a descend direction with respect to $\Energy$.
\end{enumerate}
Fortunately, the experiments from \cite{RepulsiveC} and \cite{RepulsiveS} indicate that these problems can be circumvented by replacing the Hessian $D^2 \Energy(f)$ by a suitable Sobolev metric as preconditioner; see \cref{sec:Preconditioner}.
This allows us to set up a suitably preconditioned gradient descent algorithm.
More precisely, we propose a Gauss--Newton-type optimization algorithm of iterative type (i.e., with successively decreasing regularization parameter) in \cref{sec:Optimization}.
The implementation of this algorithm is sketched in \cref{sec:Implementation}, and we close with various numerical examples showing the versatility of our algorithm with regard to reconstructing various shapes of different genera and its robustness against perturbation by noise (up to $40\%$ noise-level) in \cref{sec:num}.


\section{Inverse obstacle scattering}
\label{sec:Scat}

\subsection{The forward scattering problem}


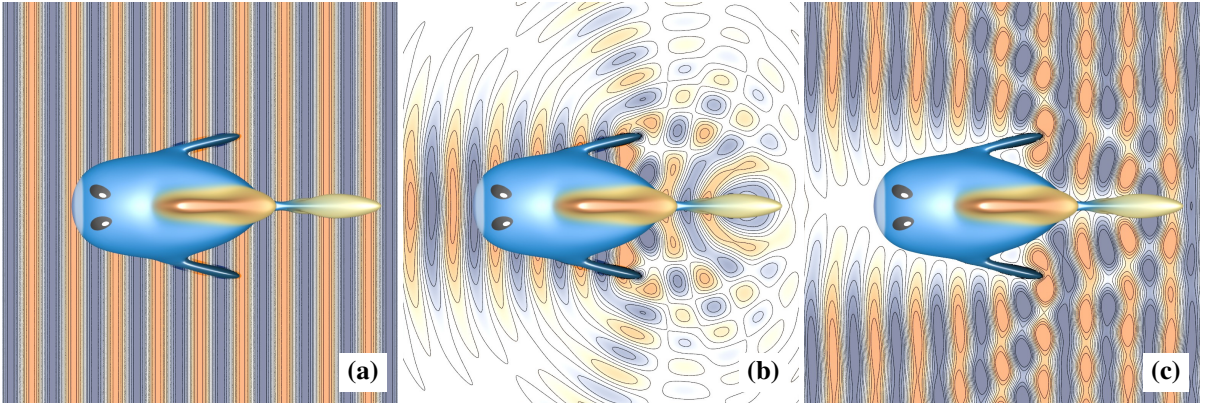
\begin{figure}[t]
    \begin{center}
    \newcommand{\inca}[2]{%
        \begin{tikzpicture}%
            \node[inner sep=0pt] (fig) at (0,0) {\includegraphics[	
            trim = 20 20 20 0, 
            clip = true,  
            angle = 0,
            width = 0.33\textwidth]{#1}};%
            \node[above left = 0ex] at (fig.south east) {%
                \begin{footnotesize} \colorbox{White}{\textbf{#2}}\end{footnotesize}%
                };%
        \end{tikzpicture}%
    }
    \inca{Figure_Nearfield/Blub_Nearfield_5Pi_Back_Incoming_Re}{(a)}%
    \hfill
    \inca{Figure_Nearfield/Blub_Nearfield_5Pi_Back_Scattered_Re}{(b)}%
    \hfill
    \inca{Figure_Nearfield/Blub_Nearfield_5Pi_Back_Outgoing_Re}{(c)}%
    \caption{%
        (a) The surface \emph{Blub} as obstacle for the plane wave $u_\inc$, incoming from the right-hand side. 
        (b) The scattered wave $u_\scat$ of the solution to the boundary value problem \cref{eq:bvp}.
        (c) The total solution $u$ of the problem. Notice that it vanishes on the boundary of the obstacle, correctly modelling the sound-absorbing behavior.
    }
    \label{fig:Nearfield}
    \vspace{-1em}
    \end{center}
\end{figure}

The forward problem we consider is the time-harmonic acoustic inverse scattering problem with a sound-soft scatterer. 
Let $\varOmega\subset\R^3$ be a bounded domain with $C^{1,\alpha}$-boundary $\varSigma\ceq\partial\varOmega$ (where $0<\alpha<1$) and connected complement.
Denote by $\kappa>0$ a fixed wave number and by $d\in\S^2$ the incident direction for the \emph{incoming} plane wave $u_\inc(x)\ceq \ee^{\ii \, \kappa \inner{x,d}}$.
The forward problem consists in finding the resulting wave 
$u\in\Holder[2][][\R^3\!\setminus\!\overline{\varOmega}][\C]\cap \Holder[0][][\R^3\!\setminus\!\varOmega][\C]$.
This wave is a superposition $u = u_\inc + u_\scat$ of the incoming wave $u_\inc$ and the \emph{scattered wave} $u_\scat$ (see \cref{fig:Nearfield}) and solves the following Dirichlet boundary value problem:
\begin{subequations}\label{eq:bvp}
\begin{alignat}{2}
\Delta u + \kappa^2u &= 0  
&&
\text{on}\ \R^3\!\setminus\! \overline{\varOmega},\label{eq:bvp_a}
\\
u &= u_\inc + u_\scat \ ,\label{eq:bvp_b}
\\
u &= 0 \qquad
&&
\text{on}\ \varSigma,\label{eq:bvp_c}
\\
\lim_{r\to\infty} r \pars*{\fdfrac{\partial u_\scat}{\partial r}- \ii \, \kappa \, u_\scat}&=0
&&
\text{where $r=\abs{x}$}.\label{eq:bvp_d}
\end{alignat}
\end{subequations}
The homogeneous Dirichlet boundary condition \cref{eq:bvp_c} models the sound-soft behavior of the obstacle's boundary $\varSigma$.
Condition \cref{eq:bvp_d} on the scattered wave is named \emph{Sommerfeld radiation condition} and ensures the scattered wave to be outgoing. By including this condition, the boundary value problem \cref{eq:bvp} has a unique solution \cite[Theorem~2.12]{Chandler}.
Without it, the solution to the boundary value problem \cref{eq:bvp_a}--\cref{eq:bvp_c} would not be unique \cite[p.~16]{zbMATH06061716}. The resulting function $u_\scat$ is then called a \emph{radiating solution} to the Helmholtz equation.

The \emph{far field pattern} 
$u_{\varSigma,\infty}\pars{\cdot,d,\kappa}\in\Lebesgue[2][][\S^2][\C]$ 
associated to the solution $u_{\scat}$ is defined by the asymptotic behavior of the scattered wave. This function satisfies the equation \cite{zbMATH06061716}
\begin{align*}
    u_{\scat}
    =
    \frac{\ee^{\ii \, \kappa \abs{x}}}{ \abs{x} }
    \pars[\Big]{
        u_{\varSigma, \infty} \pars[\Big]{
            \fdfrac{x }{ \abs{x}}, d ,\kappa
        }
        +
        \mathcal{O}\pars[\Big]{ \fdfrac{1}{\abs{x}} }
    }
    .
\end{align*}
We often abbreviate notation by
$u_\infty\ceq u_{\varSigma,\infty}\pars{\cdot,d,\kappa}$.

\subsection{Formulation as integral equation}\label{sec:bie}

The boundary value problem \cref{eq:bvp} can equivalently be reformulated using integral operators on the boundary.
We denote by $\HsdM^2$ the two-dimensional Hausdorff measure and introduce the fundamental solution to the Helmholtz problem
\begin{align}\label{eq:fundamental_solution}
    \varPhi_\kappa(x,y) 
    \ceq  
    \frac{1}{4\uppi}\frac{\ee^{\ii \, \kappa\abs{x-y}}}{\abs{x-y}} 
    \quad \text{for}\ x\neq y.
\end{align}
The single- and double-layer boundary integral operators are densely defined on $\Holder[0][][\varSigma][\C]$
via
\begin{equation*}
    V \varphi(x) 
    \ceq 
    \int_\varSigma \varPhi_\kappa(x,y)\,\varphi(y) \dd \HsdM^2(y)
    \qand
    K \varphi(x) 
    \ceq 
    \int_\varSigma \frac{\partial\varPhi_\kappa(x,y)}{\partial\nu(y)} \, \varphi(y) \dd \HsdM^2(y)
    ,
\end{equation*}
where $\nu \colon \varSigma \to \S^2$ denotes the outward pointing normal vector field of $\varSigma = \partial \varOmega$; see \cite{zbMATH01446717}.
They can be extended to continuous operators 
\begin{equation*}
	V 
    \colon 
    \Bessel[-1/2][][\varSigma][\C]
    \to 
    \Bessel[1/2][][\varSigma][\C]
    \qand
	K
    \colon 
    \Bessel[1/2][][\varSigma][\C]
    \to  
    \Bessel[1/2][][\varSigma][\C]
    .
\end{equation*}
Moreover, both operators can be extended to continuous operators $\Lebesgue[2][][\varSigma][\C]\to\Lebesgue[2][][\varSigma][\C]$; see~\cite{Chandler}.

We decided to use the well-known \emph{mixed indirect potential approach} to the boundary integral equations; see, e.g., \cite{zbMATH06061716,GaGr,hohage_diss}.
We prefer this approach over the (also commonly used) \emph{double-layer approach}, since the latter requires that $\kappa^2$ is not an eigenvalue of the interior Neumann problem in $\varOmega$; see \cite{kress_parameter,mitrea}.
Moreover, one has to expect bad conditioning of the integral equations when $\kappa^2$ comes close to such an eigenvalue.
This can be mitigated by choosing a suitable coupling parameter $\eta>0$. 
Then the scattered wave $u_\scat$ of the solution $u=u_\inc+u_\scat$ to \cref{eq:bvp} can be equivalently written as
\begin{align}\label{eq:scattered}
u_\scat(x)
	= 
	\int_\varSigma 
    \,
    \pars*{
		\frac{\partial\varPhi_\kappa(x,y)}{\partial\nu(y)} - \ii \, \eta \,\varPhi_\kappa(x,y)
	}\, \varphi(y) \dd \HsdM^2(y)
    \quad 
    \text{for $x\in\R^3\!\setminus\!\overline{\varOmega}$,}
\end{align}
where the potential $\varphi\in  \Lebesgue[2][][\varSigma]$ is the unique solution of the integral operator equation
\begin{equation}\label{eq:integral}
	\pars[\big]{ \tfrac{1}{2} \, \Id -\ii \, \eta \, V + K } \,\varphi = -u_\inc |_{\varSigma}.
\end{equation}
Moreover, the associated far field has the representation 
\begin{equation}\label{eq:farfield}
	u_{\varSigma,\infty}(z) 
	= \frac{1}{4\uppi}
	\int_\varSigma 
    \,
    \pars*{
		\fdfrac{
                \partial\,\ee^{-\ii \, \kappa \inner{z,y}}
        }{
            \partial\nu(y)
        } - \ii \, \eta \, \ee^{-\ii \, \kappa \inner{z,y} }
	}\, \varphi(y) \dd \HsdM^2(y)
	\qquad \text{for $z\in\mathbb{S}^2$}
    .
\end{equation}
In practice, we will mostly use $\eta=\kappa$ as suggested in \cite{kress_parameter} to obtain a better conditioning in \cref{eq:integral}.
We also mention the direct approach leading to the transposed integral equation,
which has the normal derivative of the total field $\partial u/\partial \nu$ as unknown:
\begin{align}\label{eq:directmethod}
\pars[\big]{ \tfrac{1}{2} \, \Id -\ii \, \eta \, V + K' } \,\frac{\partial u}{\partial \nu} = \frac{\partial u_\inc}{\partial \nu}-\ii \,\eta \,u_\inc |_{\varSigma}
\end{align}
(see \cite[Thm. 3.15]{coltonkress} for the case $\eta=0$ in 3D and \cite{kress:95a} for $\eta>0$ in 2D). Here $K'$ denotes the normal derivative of the single layer potential, which is the transpose of $K$. The far field is then given by 
\[
    u_{\varSigma,\infty}(z) 
	= - \frac{1}{4\uppi}\int_\varSigma \frac{\partial u}{\partial \nu} \ee^{-\ii \, \kappa \inner{z,y}} \dd \HsdM^2(y)
    \quad \text{for $z\in\mathbb{S}^2$.}
\]

\subsection{The inverse problem}

The associated inverse problem to \cref{eq:bvp} is to reconstruct an approximation of the shape $\varSigma$ from one or several (noisy) far field measurement(s) in $\Lebesgue[2][][\S^2][\C]$ for known wave number(s) and wave direction(s).
More precisely, suppose that we are given finitely many 
incident wave directions $d_1,\dotsc, d_\NumWaves \in \S^2$ and wave numbers 
$\kappa_1,\dotsc,\kappa_\NumWaves \in \intervaloo{0,\infty}$.
Then, for a given obstacle $\varOmega$ with $\Holder[1, \alpha]$-boundary surface $\varSigma \ceq \partial \varOmega$ (where $\alpha \in \intervaloc{0,1}$), we obtain far field patterns
\begin{equation*}
    F(\varSigma) \ceq \pars[\Big]{ 
        u_{\varSigma,\infty}(\cdot,d_1,\kappa_1) 
        ,\dotsc,
        u_{\varSigma,\infty}(\cdot,d_\NumWaves,\kappa_\NumWaves)
    }
    .
\end{equation*}
This defines the \emph{forward operator} or \emph{boundary-to-far field map}
\begin{equation*}
    F \colon \dom(F) \subset \braces[\big]{ \text{closed $\Holder[1, \alpha]$-surfaces in $\R^3$} }  \to \Lebesgue[2][][\S^2][\C^\NumWaves]
    .
\end{equation*}
It is known that the domain of $F$ can be chosen such that $F$ is injective.
For instance, this is the case if $\dom(F)$ is restricted to surfaces contained in a ball of radius $R<\uppi/\kappa$; see \cite[Corollary~5.3]{zbMATH06061716}.
Moreover, $F$ is injective on the set of surfaces included in some \emph{arbitrary} bounded open ball, provided that the number of incident wave directions for a fixed wave number is chosen large enough; see \cite[Theorem~5.2]{zbMATH06061716}.

Now suppose that we are given a true obstacle $\varSigma^\dagger$ and a measurement $g^\delta \in \Lebesgue[2][][\S^2][\C^\NumWaves]$ of $F(\varSigma^\dagger)$, polluted by some noise $\varepsilon$ of noise level $\delta > 0$, i.e.,$\norm{\varepsilon} \leq \delta$:
\begin{equation}
    g^\delta = F(\varSigma^\dagger) + \varepsilon.
    \label{eq:InverseScattering}
\end{equation}
The \emph{inverse} problem is to reconstruct $\varSigma^\dagger$ from $g^\delta$.

We will often use the representation of a $C^{1, \alpha}$-surface $\varSigma$ by a parameterization $f \colon \Domain \to \R^3$ such that $f(\Domain) = \varSigma$, where $\Domain$ is a fixed, compact $2$-dimensional manifold without boundary.
Slightly abusing notation, we define $F(f) \ceq F(f(\Domain))$, where $\dom(F)$ is regarded a subset of the open set of embeddings in the Banach space $\Holder[1, \alpha][][\Domain][\R^3]$.
Notice that the image $f(\Domain)$ of $f$ does not change under reparameterizations $\varphi \colon \Domain \to \Domain$; thus we have $F(f \circ \varphi) = F(f)$.

\subsection*{Regularization methods}\label{sec:RegularizationMethods}
The operator given in \cref{eq:farfield} is strongly smoothing. 
Thus, as already mentioned in the introduction, the stated inverse problem of recovering the obstacle from the measured far field is ill-posed, and we need an appropriate regularization method to overcome instabilities in the reconstruction. 
One widely applied method is to approximate a generalized inverse to the forward operator by an  
estimator defined in terms of 
a minimization problem including a weighted stability term.
This method is usually called \emph{Tikhonov regularization} and was first introduced by Tikhonov \cite{tikhonov1,tikhonov2} and Phillips \cite{phillips:62}.
Let $F \colon \dom(F) \to \cY$ be a general operator mapping into a Hilbert space $\cY$ and denote by $\Energy \colon \dom(\Energy)\to \R_{\geq 0}$ the stability term.
Consider measured noisy data $g^\delta\in\cY$ and a regularization parameter $\alpha>0$.
Then one seeks to solve the minimization problem
\begin{equation*}
    f_\alpha
    \ceq 
    \argmin_{f \in \dom(F) \cap \dom(\Energy)} \; \norm{F(f)-g^\delta}_{\cY}^2 + \alpha \, \Energy(f)
    ,
\end{equation*}
where the functional to be minimized is called the \emph{Tikhonov functional}.
Now consider $\dom(F)$ to be a subset of some Banach space $\cX$ and $F$ to be nonlinear.
Then, in general, even if $\dom(F)$ and $\Energy$ are convex, the Tikhonov functional is nonconvex.
In this case, Tikhonov regularization is prone to running into local minima during an optimization process.
To mitigate this issue, a variety of iterative regularization methods have been introduced in the literature.
Here we are interested in Gauss--Newton type methods:
One starts by linearizing the fitting term, which requires Fréchet differentiability of~$F$.
In order to define an iterative method, one may consider a current guess $f_k \in \cX$ for $k \in \N$ and choose a functional $E_{f_k} \colon \cX \to \R$ that is convex and satisfies $\Energy(f_k + v) \approx E_{f_k}(v)$ near $v = 0$.
Then we determine the next iterate $f_{k+1}$ by a minimization problem as follows:
\begin{equation}
    f_{k+1}
    \ceq 
    f_k + v_k 
    \qwhere 
    v_k
    \ceq
    \argmin_{v \in \cX}\, \tfrac{1}{2} \norm[\big]{F(f_k) + DF(f_k) \, v-g^\delta}_{\cY}^2 + \alpha_k \, E_{f_k}(v),
    \label{eq:LocalQuadraticProxy}
\end{equation}
where the regularization parameter $\alpha_k\ceq \rho \, \alpha_{k-1}$ is updated in every step with some $0 < \varrho < 1$.
This method was first proposed as an iterative regularization method by Bakushinskij in \cite{zbMATH00205051} for the special case $\Energy(f) = \norm{f-f_0}_{\cX}^2$ with some $f_0\in \cX$ and with $E_{f_k}(v) = \Energy(f_k+v)$.
Convergence analysis for this method can be found, for instance, in \cite{zbMATH01033895,hohage_diss}.
Each update step can then be interpreted as the first iterate of a Gauss--Newton algorithm to minimize the Tikhonov functional.
Moreover, in~\cite{Eck} Eckardt suggests using the second order Taylor expansion of the stabilizing functional $E_{f_k}(v) = \Energy(f_k) + D\Energy(f_k) \, v + \frac{1}{2} \D^2 \Energy(f_k) (v,v)$ and shows that, under suitable assumptions, this approach defines an iterative regularization method.
As already pointed out in \cref{sec:Introduction}, $\D^2 \Energy(f_k)$ might be difficult to compute and may fail to be positive-semidefinite. We therefore replace $\D^2 \Energy(f_k)$ by a positive-semidefinite bilinear form, namely the Sobolev metric $\Metric_{f_k}$; see \cref{sec:Optimization,sec:IRGNM} and, in particular, \cref{eq:LocalQuadraticProxyMetric}.

\subsection{Fréchet derivative of the forward operator}\label{sec:FrechetDerivativeForwardOperator}

Gauss--Newton-like methods require the Fréchet derivative of the forward operator $F$.
This is what we briefly recall here. 
We will come back to this in \cref{sec:Optimization}, where we will study the optimization algorithm in more detail.

For $h \in \Holder[1][][\varSigma][\R^3]$ we define the offset surface $\varSigma_h \ceq \myset{ x + h(x) }{ x \in \varSigma}$.
Due to \cite[Theorem~1.9]{hohage_diss} 
the map 
$F_{\on{loc}} \colon \Holder[1][][\varSigma][\R^3]\to \Lebesgue[2][][\S^2][\C]$, 
$h\mapsto u_{\varSigma_h,\infty}(\cdot,d,\kappa)$ is analytic, and in particular 
Fréchet-differentiable at  $h = 0$.
Moreover, \cite[Main~Corollary, Remark~1]{djellouli_formula} states that the Fréchet derivative $DF_{\on{loc}}(0)\,h$ in direction $h \in \Holder[1][][\varSigma][\R^3]$ (even for Lipschitz boundaries $\varSigma$ and Lipschitz perturbations $h \in \Holder[0,1][][\varSigma][\R^3]$) is given by the far field $w_{\varSigma,\infty}$ of the solution $w$ of
\begin{equation}\label{eq:derivative_bvp}
\begin{alignedat}{2}
\Delta w + \kappa^2w &= 0  
&&
\text{on}\ \R^3\!\setminus\! \overline{\varOmega},
\\
w(x) &= -\fdfrac{\partial u}{\partial \nu}(x) \inner{\nu(x),h(x)} \qquad
&&
\text{for all}\ x\in \varSigma,
\\
\lim_{r\to\infty} \, r \pars*{\fdfrac{\partial w}{\partial r}- \ii \, \kappa \, w} &=0
&&
\text{where $r=\abs{x}$.}
\end{alignedat}
\end{equation}  
Here $\nu \colon \varSigma \to \S^2$ denotes the outward pointing normal vector field of $\varSigma$.
Similar results for variations of a $C^2$-domain can be found  in \cite{zbMATH06061716,kirsch,zbMATH00562794,simon:80}.
One even can show that for any $k\in\mathbb{N}$ the operator $F_{\on{loc}} \colon \Holder[1][][\varSigma][\R^3]\to \Holder[k][][\S^2][\C]$ is Fréchet differentiable (\cite[Theorem 4.2]{djellouli}). Hence, by compactness of the embedding $\Holder[k][][\S^2][\C]\hookrightarrow 
\Lebesgue[2][][\S^2][\C]$,
 the linear operator $DF_{\on{loc}}(0) \colon \Holder[1][][\varSigma][\R^3] \to \Lebesgue[2][][\S^2][\C]$ is compact, underlining the fact that we are dealing with ill-posed problems.
The above carries over straightforwardly to the many-wave forward operator.
For details see \cite[Theorem~3]{zbMATH05251341}.

For numerical computations it is convenient to use the integral equation \eqref{eq:directmethod} to compute $\partial u/\partial \nu$ in the course of evaluating the forward operator and the transposed equation \eqref{eq:integral} to evaluate $w$ in \cref{eq:derivative_bvp} or its far field pattern.


\section{The tangent-point energy}\label{sec:TPE}

Here and in the following we will always assume that $m$, $n \in \N$ satisfy $n < m$.
Let $\varSigma \subset \AmbSpace$ be a $C^1$-submanifold of dimension $n$ that is closed (i.e., compact and without boundary). 
For $1 \leq p < \infty$ its \emph{tangent-point energy} is given by the double integral
\begin{align}
	\label{eq:TPByRadius}
	\Energy_p(\varSigma)
	\ceq
	2^p
	\int_\varSigma \! \int_\varSigma 
		\frac{1}{r_{\TP}(\xi,\eta)^p} 
	\dd \HsdM^\DomDim(\xi) \dd \HsdM^\DomDim(\eta)
	,
\end{align}
where $\HsdM^\DomDim$ denotes the $n$-dimensional Hausdorff measure on $\AmbSpace$
and 
where $r_{\TP}(\xi,\eta)$ is the \emph{tangent-point radius}, i.e., the radius of the smallest sphere that passes through $\xi$ and $\eta$ and that is tangential to $\varSigma$ at the point $\xi$. 

This energy was proposed by Buck and Orloff in \cite{Buck:1995:ASE} and independently by Gonzales and Maddocks in \cite{MR1692638} as a self-avoidance energy for knots and links in $\R^3$.
This energy is intended to blow up to $\infty$ if a continuous family of knots gradually degenerates into  another knot class, e.g., by forming self-intersections.
The idea was that minimizing this energy within a knot class leads to ``nice'' representatives, i.e., relatively smooth curves with relatively large ``gaps'' between their strands.

In \cite{zbMATH06017037} Strzelecki and von~der~Mosel studied analytical properties of the tangent-point energy for curves.
Later \cite{zbMATH06203565}, they generalized the tangent-point energy to embedded manifolds of arbitrary dimension $\DomDim$ and proved fundamental results in the case that $p  > 2 \, \DomDim$: the regularization property (see \cref{thm:GeometricMorrey}) and the self-avoidance property (see \cref{rem:SelfAvoidance}).
Together with Kolasi\'nski~\cite{zbMATH07046194}, they established further important structural properties of the tangent-point energy, namely certain notions of rigidity, compactness of sublevel sets, and lower-semicontinuity. 
We will use these results in \cref{sec:Regularization} to prove that our reconstruction scheme is indeed regularizing.
Hence, we discuss these properties in more detail in \cref{sec:PropertiesTP}.

The tangent-point energies form only one of several families of functionals with interesting self-avoidance and regularization properties, e.g., the \emph{reach} or \emph{thickness} and their inverse, the \emph{ropelength}; \emph{O'Hara's energies}; and the \emph{Menger integral curvature energies}.
For more background and for an overview over the historic developments we refer the reader to the following survey by Strzelecki and von der Mosel \cite{zbMATH07063699}.

Due to their nonlocal nature, minimizing such self-avoidance functionals by numerical means has been a difficult and expensive task for a long time. 
The straightforward approach to discretizing the defining double integral \cref{eq:TPByRadius} by a Riemann double sum leads to costs quadratic in the number of degrees of freedom. 
Hence, this approach is only feasible for $1$-dimensional manifolds. 
The tangent-point energy has been successfully applied as a hard barrier function against self-contact in the context of elastic knots in \cite{zbMATH07356032} and \cite{zbMATH06983824}.
A similar approach has been proposed in \cite{zbMATH06660993}, utilizing the tangent-point radius to construct a \emph{soft} penalty for the overlap of thickened curves.

For handling two-dimensional surfaces, the quadrature rule for the double integral in \cref{eq:TPByRadius} has to be adapted for reasons of computational costs. 
A first approach in this direction was made in~\cite{zbMATH07547922}, where a hierarchy of triangulations on the integration domain was employed to coarsen the quadrature rule.
Even more efficient discretizations based on tree-accelerated codes were developed in the series of papers \cite{RepulsiveC,RepulsiveS,Sassen:2024:RS}, and these are the discretizations that we employ in our numerical simulations.
The crucial ingredient for the numerical optimization of $\Energy_p$ is to know the \emph{right} Banach space to work in, the so-called \emph{energy space}; see \cref{sec:SoboSloboSpaces,,sec:SoboSloboSpacesOnManifolds,,sec:EnergySpaces}.
Moreover, we have to study the differential of the energy; see \cref{sec:DEnergy}.
This allows us to construct a suitable \emph{Sobolev metric} (cf.~\cite{RepulsiveC,RepulsiveS}) and a computational feasible preconditioner (see ~\cite{RepulsiveS}). 
This metric and its preconditioner provide us with a conceptually simple, robust, and quick optimization method via steepest descent, see~\cref{sec:Preconditioner}.

\subsection{Rigidity, compactness, lower semicontinuity}
\label{sec:PropertiesTP}

Here we collect a couple of properties of the tangent-point energy that make it suitable as regularizer for surface reconstruction problems.
The first of these properties is fundamental to all the others.
In their landmark paper \cite{zbMATH06203565}, Strzelecki and von~der~Mosel showed that the tangent-point energy has the following regularizing effect; see \cite[Theorem 1.4]{zbMATH06203565}:

\begin{btheorem}[Geometric Sobolev–Morrey embedding]\label{thm:GeometricMorrey}
	Let $2 \, \DomDim < p < \infty$ such that $\alpha \ceq 1 - 2 \, \DomDim/p > 0$.
	For each $0 \leq E < \infty$ there are $r = r(p,E,\DomDim,\AmbDim)> 0$ and $0 < L = L(p,E,\DomDim,\AmbDim) <\infty$ such that the following holds true:
	\newline
	Each $\DomDim$-dimensional, closed $\Holder[1]$-submanifold $\varSigma \subset \AmbSpace$ satisfying
	$
		\Energy_p(\varSigma) \leq E
	$
	can be uniformly parameterized by graph patches of class $\Holder[1,\alpha]$ in the following sense:
	For each $\xi \in \varSigma$ the patch $\OpenBall{\xi}{r} \cap \varSigma$ can be parameterized by some function $h \colon U \subset T_\xi \varSigma \to \pars{T_\xi\varSigma^\perp}$ satisfying 
	\begin{align*}
		h(0) = \xi,
		\;\;
		D h(0) = 0,
		\;\;
		\Lip(h) \leq 1, 
		\;\; \text{and} \;\; 
		\seminorm{Dh}_{\Holder[0,\alpha]}
		\ceq 
		\smash{
			\esssup_{\eta,\, \zeta \in U} \frac{\abs{Dh(\zeta) - Dh(\eta)}}{\abs{\zeta - \eta}^\alpha}
		}
		\leq L.
	\end{align*}
	In particular, finite energy $\Energy_p(\varSigma) < \infty$ implies that $\varSigma$ is a manifold of class $\Holder[1,\alpha]$.
\end{btheorem}

Note that (i) the radii of the graph patches are uniformly bounded from below by $r > 0$
and that (ii) the modulus of continuity of the parameterizations are uniformly controlled.
This makes closed submanifolds with bounded tangent-point energy surprisingly rigid:
As soon as two such manifolds are sufficiently close in the Hausdorff distance $\HsdD$, they are already diffeomorphic and ambient isotopic; see \cite[Section~4]{zbMATH07046194}:

\begin{theorem}[Rigidity]\label{thm:Rigidity}
	Let $2 \, \DomDim < p < \infty$ such that $\alpha \ceq 1 - 2 \, \DomDim/p > 0$.
	Let $0 < R < \infty$, and $0 < E < \infty$. 
	Then there exist $\delta = \delta(R,E,\DomDim,\AmbDim) > 0$ and $C = C(R,E,\DomDim,\AmbDim) > 0$ such that the following holds true:
	Let $\varSigma \subset \AmbSpace$ and $\varGamma \subset \AmbSpace$ be two $\DomDim$-dimensional, closed $\Holder[1]$-submanifolds satisfying
	\begin{align*}
		\varSigma, \, \varGamma  \subset \ClosedBall{0}{R},
		\quad
		\Energy_p(\varSigma), \, \Energy_p(\varGamma) \leq E,
		\qand 
		\HsdD(\varGamma,\varSigma) \leq \delta
		.
	\end{align*}
	Then there is a $\Holder[1]$-diffeomorphism $\diffeo \colon \AmbSpace \to \AmbSpace$ such that $\diffeo$ is isotopic to $\id_{\AmbSpace}$ and such that
	\begin{align*}
		\diffeo(\varSigma) = \varGamma,
		\;\;
		\norm{\diffeo - \id_{\AmbSpace}}_{\Lebesgue[\infty]} \leq C \, \HsdD(\varGamma,\varSigma),
		\;\;\text{and}\;\;
		\norm{D(\diffeo - \id_{\AmbSpace})}_{\Lebesgue[\infty]}
			\leq 
			C \, \HsdD(\varGamma,\varSigma)^{\alpha/2}.
	\end{align*}
\end{theorem}

For details, we refer the readers to the original paper \cite{zbMATH07046194} by Kolasiński, Strzelecki, and  von~der~Mosel, in particular to Theorem 4 and Lemma 4.10 therein.

\begin{remark}\label{rem:SelfAvoidance}
	\cref{thm:Rigidity} also implies the self-avoidance property of $\Energy$:
	Let $\varSigma(t)$ be a $\HsdD$-continuous path $\varSigma(t)$ of closed, $\DomDim$-dimensional $\Holder[1]$-submanifolds satisfying $\Energy_p(\varSigma(t)) \leq E$ for all $t \in \intervalcc{0,1}$.
	Such a path cannot change the ambient isotopy class. 
	Indeed, since $\intervalcc{0,1}$ is compact, the path is a compact set with respect to $\HsdD$. Hence, there must exist $R>0$ such that $\varSigma(t) \subset \ClosedBall{0}{R}$ for all $t \in \intervalcc{0,1}$.
	Moreover, we can find finitely many $0 = t_0 < t_1 < \dotsm < t_N = 1$ such that $\HsdD( \varSigma(t_i), \varSigma(t_{i+1}) ) \leq \delta$. 
	So, \cref{thm:Rigidity} implies that $\varSigma(0) = \varSigma(t_0) \cong \varSigma(t_1) \cong \dotsm \cong \varSigma(t_K) = \varSigma(1)$ are all ambient isotopic.
	Put differently: when a change of topology (for example a self-intersection) occurs, then the tangent-point energy must blow up to $\infty$.
	
	This can be a blessing as well as a curse: On the one hand, the tangent-point energy is a barrier function for connected components of the space of embeddings. 
	It allows us to deal with topological constraints like \emph{being embedded} or \emph{fixing the ambient isotopy class}. 
	This can be desirable in many applications.
	On the other hand, one has to start the optimization in the right isotopy class. 
\end{remark}

Next we deal with a compactness result for sublevel sets of the tangent-point energy from \cite[Theorem~3]{zbMATH07046194}.
This result will be used in \cref{thm:Existence}, where we employ the direct method of calculus of variations in order to show existence of minimizers to the regularized inverse scattering problem.

\begin{theorem}[Compactness]\label{thm:Compactness}
	Let $2 \, \DomDim < p < \infty$ such that $\alpha \ceq 1 - 2 \, \DomDim/p > 0$. Let $0 < R < \infty$ and $0 < E < \infty$. 
	For each $k \in \N$ let $\varSigma_k \subset \AmbSpace$ be a closed, $n$-dimensional $\Holder[1]$-submanifold
	satisfying
	\begin{equation*}
		\varSigma_k \subset \ClosedBall{0}{R}
		\qand
		\Energy_p(\varSigma_k) \leq E
		.
	\end{equation*}
	Then there is a subsequence $(\varSigma_{k_\ell})_{\ell \in \N}$ and a closed, $n$-dimensional $\Holder[1]$-submanifold $\varSigma$
	such that
	\begin{equation*}
		\varSigma \subset \ClosedBall{0}{R}
		,
		\quad
		\Energy_p(\varSigma) \leq E
		,
		\qand
		\HsdD(\varSigma_{k_\ell},\varSigma) \converges[\ell \to \infty] 0
		.
	\end{equation*}
	Let $\delta = \delta(R,E,\DomDim,\AmbDim) > 0$ be the constant from \cref{thm:Rigidity}.
	As soon as $\HsdD(\varSigma_{k_\ell}, \varSigma) < \delta$, there are $\Holder[1]$-diffeomorphisms $\diffeo_\ell \colon \AmbSpace \to \AmbSpace$, each isotopic to $\id_{\AmbSpace}$, such that 
	\begin{equation}
		\diffeo_\ell(\varSigma) = \varSigma_{k_\ell}
		,
		\quad
		\norm{\diffeo_\ell - \id_\AmbSpace}_{\Lebesgue[\infty][][\AmbSpace]} \converges[\ell \to \infty] 0
		\qand
		\norm{D(\diffeo_\ell - \id_\AmbSpace)}_{\Lebesgue[\infty][][\AmbSpace]} \converges[\ell \to \infty] 0
		.	
		\label{eq:ConvergenceC1}
	\end{equation} 
	Hence, eventually, the surfaces $\varSigma_{k_\ell}$ are ambient isotopic to~$\varSigma$.
\end{theorem}

Notice the additional requirement that the submanifolds $\varSigma_k$ must stay in the prescribed, bounded set $\ClosedBall{0}{R}$ and cannot escape to ``infinity''.
An energy bound alone cannot prevent that because $\Energy_p$ is translation invariant. 
Moreover, the tangent-point energy goes \emph{down} when scaling up, i.e., $\Energy_p(k \cdot \varSigma) \to 0$, as $k \to \infty$, which also has to be prevented somehow.

\begin{definition}\label{dfn:ConvergenceC1}
	For convenience, we will refer to the notion of convergence from \cref{eq:ConvergenceC1} by saying that $\varSigma_k \to \varSigma$ \emph{in $\Holder[1]$}.
	We would also like to point out that \cref{eq:ConvergenceC1} also implies that 
	\begin{equation*}
		\diffeo_\ell^{-1}(\varSigma_{k_\ell}) = \varSigma
		,  
		\quad
		\norm{\diffeo_\ell^{-1} - \id_\AmbSpace}_{\Lebesgue[\infty][][\AmbSpace]} 
		\converges[\ell \to \infty] 0
		, 
		\qand
		\norm{D(\diffeo_\ell^{-1} - \id_\AmbSpace)}_{\Lebesgue[\infty][][\AmbSpace]} 
		\converges[\ell \to \infty] 0.	
	\end{equation*}
\end{definition}

The final ingredient that we require for Tikhonov regularization is the following statement on lower semicontinuity; see \cite[Theorem 2 (i)]{zbMATH07046194}.

\begin{theorem}[Lower semicontinuity]\label{thm:SemiContinuity}
	Let $2 \, \DomDim < p < \infty$ such that $1 - 2 \, \DomDim/p > 0$. 
	Let $\varSigma_k \subset \AmbSpace$, $k \in \N$ be closed, $\DomDim$-dimensional $\Holder[1]$-submanifolds.
	Suppose that there is a compact subset $\varSigma \subset \AmbSpace$ such that
	\begin{equation*}
		\smash{\liminf_{k \to \infty}} \; \Energy_p(\varSigma_k) < \infty 
		\qand  
		\HsdD(\varSigma_k,\varSigma) \converges[k \to \infty] 0.
	\end{equation*}
	Then $\varSigma$ is also a closed, $\DomDim$-dimensional $\Holder[1]$-submanifold, and we have
	\begin{equation*}
		\Energy_p(\varSigma) 
		\leq 
		\smash{\liminf_{k \to \infty}} \; \Energy_p(\varSigma_k)
		.
	\end{equation*}
\end{theorem}

\subsection{\SoboSlobo spaces}\label{sec:SoboSloboSpaces}

The above statements on $\Holder[1]$- and $\Holder[1,\alpha]$-regularity are sufficient to set up the Tiknonov regularization for the inverse scattering problem \cref{eq:Tikhonov}. 
However, for the numerical optimization it is important to know the energy space of the tangent-point energy more precisely; this is crucial for designing suitable preconditioners that make the numerical treatment practicable.

As it turns out, the ``right'' spaces  to use are \SoboSlobo spaces. 
These are one of several possible ways to generalize the broadly known Sobolev spaces to noninteger order of differentiability.
We will discuss this in more detail in \cref{sec:EnergySpaces} below.
For now, we start by defining the \SoboSlobo norm of a measurable function $u \colon U \to \TargetSpace$
on an open set $U\subset\DomSpace$ with values in a finite-dimensional Hilbert space $(\TargetSpace,\inner{\cdot,\cdot}, \abs{\cdot})$. 
The space $H$ will typically be $\R$, $\AmbSpace$, or $\Hom(\AmbSpace;\AmbSpace)$, the latter equipped with the Frobenius inner product.

For $p \in \intervalco{1,\infty}$ and $\sigma\in \intervaloo{0,1}$ we define the \emph{\SoboSlobo space} $\Sobo[\sigma,p][][U][\TargetSpace]$ to be the vector space of functions $u\in \Lebesgue[p][][U][\TargetSpace]$ such that the following seminorm is finite:
\begin{equation}
	\seminorm{u}_{\Sobo[\sigma,p]}
	\ceq 
	\norm{\deltaOp[\sigma] u}_{\Lebesgue[p][][U\times U](\mu)}
	\ceq
	\pars*{
		\int_U \! \int_U
			\abs[\big]{ \deltaOp[\sigma] u(x,y)}^p
		\dd \mu(x,y)
	}^{1/p}
	,
	\label{eq:SoboSloboNormEulcidean}
\end{equation}
where we define 
\begin{equation}
	\deltaOp[\sigma] u(x,y) \ceq \frac{ u(y)-u(x) }{ \abs{y-x}^{\sigma} }
	\qand 
	\dd \mu(x,y) = \frac{\dd x \dd y}{\abs{y-x}^\DomDim}
	.
	\label{eq:deltaOpAndmu}
\end{equation}
In the case $p = \infty$, we put $\seminorm{u}_{\Sobo[\sigma,\infty]} \ceq \norm{ \deltaOp[\sigma] v }_{\Lebesgue[\infty][][U \times U]}$, rendering $\Sobo[\sigma,\infty][][U][\TargetSpace]$ identical to the Hölder space $\Holder[0,\sigma][][U][\TargetSpace]$.
Together with the norm $\norm{u}_{\Sobo[\sigma,p]} \ceq \norm{u}_{\Lebesgue[p]} + \seminorm{Du}_{\Sobo[\sigma,p]}$, the space $\Sobo[\sigma,p][][U][\TargetSpace]$ is a Banach space. It is reflexive if and only if $1 < p < \infty$.

For $s \in \intervaloo{1,2}$ we define the \emph{\SoboSlobo space} $\Sobo[s,p][][U][\TargetSpace]$ to be the vector space of functions $u\in \Sobo[1,p][][U][\TargetSpace]$ such that $\seminorm{Du}_{\Sobo[s-1,p]} < \infty$. 
Equipped with the norm
\begin{equation*}
	\norm{u}_{\Sobo[s,p]}
	\ceq
	\norm{u}_{\Sobo[1,p]} + \seminorm{Du}_{\Sobo[s-1,p]}
	=
	\norm{u}_{\Lebesgue[p]}
	+
	\norm{Du}_{\Lebesgue[p]} + \seminorm{Du}_{\Sobo[s-1,p]}
	,
\end{equation*}
it becomes a Banach space. It is reflexive if and only if $1 < p < \infty$.

As it is common practice, we define $\Sobo[s,p][0][U][\TargetSpace]$ for $0 < s < 2$ to be the closure of the subspace
$
	\myset{ u \in \Sobo[s,p][][U][\TargetSpace] }{ \supp(u) \subset\subset U} 
$
in the norm $\norm{\cdot}_{\Sobo[s,p]}$ and denote by $\Sobo[-s,p][][U][\TargetSpace]$ its dual space equipped with the dual norm 
\begin{equation*}
	\norm{\xi}_{\Sobo[-s,p][][U]}
	\ceq 
	\sup_{u \in \Sobo[s,p][0][U][\TargetSpace] \setminus \braces{0}}
	\frac{ \inner{\xi , u } }{ \norm{u}_{\Sobo[s,p][][U]} }
	.
\end{equation*}
As introductory text to fractional Sobolev spaces we recommend the excellent book \cite{zbMATH07647941} by Leoni.

\subsection{\SoboSlobo spaces on manifolds}\label{sec:SoboSloboSpacesOnManifolds}

From now on, $\Domain$ will denote a compact $\DomDim$-dimensional smooth manifold without boundary.
We may define \SoboSlobo spaces on $\Domain$ in various ways:
A~first approach is to define \SoboSlobo functions locally via charts and to patch these together by subordinate partitions of unity; see \cite{SobMan}. However, this approach is quite elaborate and also somewhat unnatural for our application.

As a second method we could employ a (sufficiently smooth) Riemannian metric on $\Domain$ to make sense of integration and to use the Riemannian distance function as replacement for $\abs{y-x}$. For $0 < \sigma < 1$, this allows a straightforward generalization of $\deltaOp[\sigma] u$ and of the seminorm $\seminorm{u}_{\Sobo[\sigma,p]}$.
But for generalizing $\seminorm{u}_{\Sobo[s,p]}$, $1 < s < 2$ we would need a replacement for $\deltaOp[\sigma] D u(x,y)$; and this would require us to take the difference of differentials $Du(x) \in T_x^* \Domain$ and $Du(y) \in T_y^* \Domain$, which live in quite different cotangent spaces. 
While this could be done via parallel transport, it would lead to tremendous technical overhead.

We therefore choose a more elementary approach by defining fractional Sobolev norms dependent on a (sufficiently smooth) embedding $f \colon \Domain \to \AmbSpace$. 
As we deal with embedded manifolds anyway, this is also more natural for our application.
Nonetheless, this is a very brief introduction; a more detailed account can be found in the dissertation by Axel Wings \cite{WingsPHDThesis}.

Let us denote the Riemannian pullback metric on $M$ induced by $f$ by
\begin{equation}\label{eq:RiemannianMetric}
	g_f(x)(X,Y) 
	\ceq 
	\inner{\dd f(x) \, X, \dd f(x) \, Y}_{\mathbb{R}^{m}}, \quad \text{for all $X$, $Y \in T_xM$.}
\end{equation}
Then the Riemannian volume measure  $\vol_f$ associated to $g_f$ allows us to perform integration.
For a measurable map $u \colon \Domain \to \TargetSpace$ we define the Lebesgue norms
\begin{equation*}
	\norm{u}_{\Lebesgue[p](f)} 
	\ceq 
	\pars[\Big]{ \int_\Domain \abs{u(x)}^p \dvol_f(x)}^{1/p},
	\quad 
	1 \leq p < \infty
	\qand
	\norm{u}_{\Lebesgue[\infty]}
	\ceq \esssup_{x\in\Domain} \abs{u(x)}
	.
\end{equation*}
As a notion of distance, we can simply employ the chord distance $\abs{f(y)-f(x)}$. 
For $0 < \sigma < 1$ and $1 \leq p < \infty$ we define the $f$-dependent low-order \SoboSlobo seminorms as follows:
\begin{equation}
	\seminorm{u}_{\Sobo[\sigma\!,p](f)}\!
	\ceq 
	\norm{\DeltaOp[\sigma][f] u}_{\Lebesgue[p](\mu_f)}
	\ceq 
	\pars[\Big]{
		\int_\Domain \! \int_\Domain
			\abs[\big]{ \DeltaOp[\sigma][f] u(x,y)}^p  
		\dd \mu_f(x,y)
	}^{1/p}
	,
	\label{eq:SoboSloboNormLow}
\end{equation}
where
\begin{equation}
	\DeltaOp[\sigma][f] u(x,y) \ceq \sdfrac{ u(y) \!-\! u(x) }{ \abs{f(y)\!-\! f(x)}^{\sigma} }
	\qand 
	\dd \mu_f(x,y) \ceq \sdfrac{\dvol_f(x) \dvol_f(y)}{ \abs{f(y)\!-\!f(x)}^\DomDim}.
	\label{eq:DeltaOpAndmuf}
\end{equation}
For $p = \infty$ we get the Hölder seminorms $\seminorm{u}_{\Holder[0,\sigma][](f)} \ceq \seminorm{u}_{\Sobo[\sigma,\infty](f)} \ceq \norm{\DeltaOp[\sigma][f] u}_{\Lebesgue[\infty][][\Domain \times \Domain]}$.

In order to create a substitute for $\deltaOp[s-1] D u$, we employ the Moore-Penrose pseudoinverse 
$\dd f(x)\pinv$ of the differential.
If we fix any Riemannian metric on $\Domain$, then we can write $\dd f(x)\pinv = (\dd f(x)\adj \dd f(x))^{-1} \dd f(x)\adj$, where $\dd f(x)\adj$ denotes the adjoint of $\dd f(x)$.
(Notice  that  $\dd f(x)\pinv$ does \emph{not} depend on the actual choice of Riemannian metric.)
The pseudoinverse $\dd f(x)\pinv$ is a linear map from $\AmbSpace$ to $T_x \Domain$, we can use it to pull back the differential $\dd u(x) \colon T_x \Domain \to \TargetSpace$, giving rise to the operator $\DOp[f]$:
\begin{equation}
	\label{eq:OperatorD}
	\DOp[f] u(x) \ceq  \dd u(x)  \dd f(x)\pinv \in \Hom(\AmbSpace;\TargetSpace).
\end{equation}
This operator has the nice property that it is covariant in the following sense:
For every sufficiently smooth diffeomorphism $\varphi \colon \Domain \to \Domain$ we have
\begin{equation}
	\DOp[\,\pars{f \circ \varphi}] \pars{u \circ \varphi}(x)
	=
	\DOp[f] u( \varphi(x) )
	.
	\label{eq:DfIsCovariant}
\end{equation}
The map $\DOp[f] u$ is vector space-valued, so $\DeltaOp[s-1][f] \DOp[f] u$ makes sense, and we can define
\begin{equation*}
	\norm{u}_{\Sobo[s,p](f)}
	\ceq
	\norm{u}_{\Lebesgue[p](f)}
	+
	\norm{ \DOp[f] u }_{\Lebesgue[p](f)} + \seminorm{ \DOp[f] u }_{\Sobo[s-1,p](f)}
	.
\end{equation*}
With \cref{eq:DfIsCovariant} and with the transformation formula for integrals one can show that the norms $\norm{u}_{\Sobo[\sigma,p](f)}$ and $\norm{u}_{\Sobo[s,p](f)}$ are invariant under reparameterization.

Now we pick a smooth embedding $f_0 \colon \Domain \to \AmbSpace$ and define the \SoboSlobo spaces
\begin{equation*}
	\Sobo[\sigma,p][][\Domain][\TargetSpace]
	\ceq
	\myset[\big]{ 
		u \in \Lebesgue[p][][\Domain][\TargetSpace]
	}{
		\norm{u}_{\Sobo[\sigma,p](f_0)} < \infty
	}
	\quad 
	\text{for $0 < \sigma < 2$}
	.
\end{equation*}
Analogously to the Euclidean case, we may define $\Sobo[\sigma,p][0][\Domain][\TargetSpace]$ and $\Sobo[-\sigma,p][][\Domain][\TargetSpace]$ as its dual space with the dual norm $\norm{\cdot}_{\Sobo[-\sigma,p](f_0)}$.

Sobolev and Morrey inequalities known for bounded, open subsets of Euclidean space with smooth boundary 
(see, e.g., \cite[Theorem 92]{SobMan}) 
carry over to $\Sobo[s,p][][\Domain][\AmbSpace]$ by standard techniques, of course with proportionality constants $C(f_0)$ depending on $f_0$:
\begin{align}
	\norm{u}_{\Sobo[t,q](f_0)}
	&\leq 
	C_{\mathrm{S}}(f_0) \norm{u}_{\Sobo[s,p](f_0)}
	&
	&
	\text{for $t \geq 0$, $1 \leq q < \infty$ if $t < s$, $t - \DomDim/q \leq s - \DomDim/p$,}
	\label{eq:SobolevInequality}
	\\
	\norm{u}_{\Holder[k,\alpha](f_0)}
	&\leq 
	C_{\mathrm{M}}(f_0) \norm{u}_{\Sobo[s,p](f_0)}
	&
	&
	\text{for $k \in \N_0$, $0 < \alpha <1$ if $k + \alpha \leq s - \DomDim/p$.}
	\label{eq:MorreyInequality}
\end{align}
For $\alpha \ceq s - \DomDim / p - 1 > 0$ we have a Morrey embedding $\Sobo[s,p][][\Domain][\AmbSpace] \hookrightarrow \Holder[1,\alpha][][\Domain][\AmbSpace]$. 
Thus,
\begin{equation*}
	\Emb[s,p][][\Domain][\AmbSpace]
	\ceq 
	\myset[\big]{ f \in \Sobo[s,p][][\Domain][\AmbSpace] }{ \text{ $f$ is a $\Holder[1]$-embedding } }
\end{equation*}
is a well-defined and open subset of $\Sobo[s,p][][\Domain][\AmbSpace]$.
Moreover, the $f$-dependent norms still make sense and are all equivalent as long as $f$ has some minimal smoothness:
\begin{lemma}\label{lem:equivalence}
	Let $1 < s < 2$ and $1 < p < \infty$ satisfy $\alpha \ceq s - \DomDim / p - 1 > 0$.
	Let $f_0 \colon \Domain \to \AmbSpace$ be a smooth embedding and let $f \in \Emb[s,p][][\Domain][\AmbSpace]$.
	Then the norms $\norm{\cdot}_{\Sobo[s,p](f_0)}$ and $\norm{\cdot}_{\Sobo[s,p](f)}$ are equivalent.
	In particular, the topological vector space $\Sobo[s,p][][\Domain][\AmbSpace]$ does not depend on the actual choice of $f_0$.
\end{lemma}
\begin{proof}
	For the sake of brevity, we sketch this only partially.
	In particular, we skip the details on why $\Lebesgue[p]$-norms induced by $f$ and $f_0$ are equivalent.
	Now we observe that the functions
	\begin{equation}
		\varLambda_f(x,y) \ceq \sdfrac{ \abs{f_0(y) - f_0(x)} }{ \abs{f(y) - f(x)} }
		\qand
		\varLambda_f(x,y)^{-1} = \sdfrac{ \abs{f(y) - f(x)} }{ \abs{f_0(y) - f_0(x)} }
		\label{eq:Lambda}
	\end{equation}
	are continuous in $(x,y)$ because both $f$ and $f_0$ are embeddings of class $\Holder[1]$.
	This and the identity $\DeltaOp[s-1][f] u = \varLambda_f^{(s-1)} \DeltaOp[s-1][f_0] u$ already suffice to show equivalence of the seminorms $\norm{\cdot}_{\Sobo[s-1,p](f_0)}$ and $\norm{\cdot}_{\Sobo[s-1,p](f)}$.
	For the higher order seminorms we also employ the identity
	\begin{align}
	\begin{split}
		\DOp[f] u(x) 
		&= \dd u(x) \dd f(x)\pinv
		= \dd u(x) \dd f_0(x)\pinv \dd f_0(x) \dd f(x)\pinv
		\\
		&= \dd u(x) \dd f_0(x)\pinv \pars[\big] {\dd f(x) \dd f_0(x)\pinv}\pinv
		= \pars[\big]{ \DOp[f_0] u(x) } \pars[\big]{ \DOp[f_0] f(x) }\pinv
		.
		\label{eq:FortunateIdentity}
	\end{split}
	\end{align}
	(Note that the identity $(A \, B)\pinv = B\pinv \, A\pinv$ does not hold true in general. Here we really exploit that both $\dd f_0(x)$ and $\dd f(x)$ are injective.)
	Then the norm equivalence emerges from the following three facts: (i) the
	map $A \mapsto A\pinv$ is a smooth diffeomorphism on the smooth manifold of matrices of rank $\DomDim$; see \cite{zbMATH03408807}; (ii) we have the ``product rule'' 
	\begin{equation*}
		\DeltaOp[s-1][f_0] \pars{u \, v}(x,y)
		=
		\DeltaOp[s-1][f_0] u(x,y) \, v(x) + u(y) \DeltaOp[s-1][f_0] v(x,y)
		;
	\end{equation*}
	and (iii) each of the maps $\DOp[f_0] u$, $\DOp[f] u$, $\pars{\DOp[f_0] f}\pinv$, and $\pars{\DOp[f] f_0}\pinv$ is continuous and thus bounded.
\end{proof}

\subsection{Energy space}\label{sec:EnergySpaces}

Next we explain why the \SoboSlobo spaces are the ``right'' spaces for handling the tangent-point energy. The reason for this is the following theorem by Blatt \cite{zbMATH06214305}:

\begin{theorem}\label{thm:EnergySpace}
	Let $\varSigma \subset \AmbSpace$ be a closed, $\DomDim$-dimensional submanifold of class $\Holder[1]$, 
	let $2 \,\DomDim < p < \infty$, and put $s \ceq 2 - \DomDim/p$.
	Then the following statements are equivalent:
	\begin{enumerate}[noitemsep]
		\item $\Energy_p(\varSigma) < \infty$.
		\item $\varSigma$ is of class $\Sobo[s,p]$, i.e., $\varSigma$ is locally the graph of a function of class $\Sobo[s,p][]$.
		\item There is a closed, $\DomDim$-dimensional, and smooth manifold $\Domain$ and an embedding $f \in \Emb[s,p][][\Domain][\AmbSpace]$ such that $\varSigma = f(\Domain)$.
	\end{enumerate}
\end{theorem}

Note that the condition 
$p > 2 \, \DomDim$
is equivalent to $\alpha \ceq 1 - 2 \, \DomDim / p > 0$, and that the latter is precisely the requirement for the Morrey embedding $\Sobo[s,p]\hookrightarrow \Holder[1,\alpha]$.
So in particular, this shows once again that submanifolds with finite energy are also of class $\Holder[1,\alpha]$ (cf. \cref{thm:GeometricMorrey}).

The third statement is actually not part of Blatt's result. Indeed, it requires some extra work as the second statement only shows that $\varSigma$ allows an atlas of class $\Sobo[s,p] \hookrightarrow \Holder[1]$ and that $f \colon \varSigma \to \AmbSpace$, $f(\xi) = \xi$ is an embedding of class $\Sobo[s,p]$. 
However, it is a well-known result by Whitney that every $\Holder[1]$ manifold has also a compatible atlas of class $\Holder[\infty]$. This result is somewhat implicitly contained in \cite{zbMATH02530161} and is detailed out further, e.g., in the textbook \cite[Chapter 2, Section 2, Theorem 2.9]{zbMATH03555096}.

\subsection{Tangent-point energy for parameterized surfaces}
\label{sec:ParameterizedSurfaces}

The third statement of \cref{thm:EnergySpace} allows us to shift our attention from submanifolds to \emph{embeddings} of manifolds. This way we may employ mapping spaces with their natural topology and differential structure.
We can represent every submanifold $\varSigma \subset \AmbSpace$ with finite tangent-point energy by an embedding $f \colon \Domain \to \AmbSpace$ of class $\Sobo[s,p] \subset \Holder[1]$ satisfying $\varSigma = f(\Domain)$. 
Note that this embedding is far from unique: Every smooth diffeomorphism $\varphi \colon \Domain \to \Domain$ allows us to use $f \circ \varphi$ instead of $f$ as representative of $\varSigma$.

Next we derive an expression for $\Energy_p(\varSigma)$ in terms of a parameterization $f$. 
Denoting the orthogonal projector onto the normal space $(T_\xi \varSigma)^\perp$ by $\Npr(\xi)$, we can rewrite the energy as
\begin{equation}\label{eq:TP}
	\Energy_p(\varSigma)
	=
	\int_\varSigma \int_\varSigma
	\frac{ 
		\abs{ \Npr(\xi) \, (\eta-\xi) }^{p}
	}{
		\abs{ \eta-\xi }^{2p}
	}
	\dd \HsdM^\DomDim(\xi) \dd \HsdM^\DomDim(\eta).
\end{equation}
This expression is often easier to handle; and it also explains the somewhat odd normalization factor of $2^p$ in \cref{eq:TPByRadius}. 
This integral is also the starting point for reformulating $\Energy_p(\varSigma)$ in terms of~$f$.
In \cref{eq:OperatorD} we defined the operator
$
	\DOp[f] u(x) \ceq  \dd u(x)  \dd f(x)\pinv \in \Hom(\AmbSpace,\AmbSpace)
$
for every weakly differentiable map $u \colon \Domain \to \AmbSpace$.
Applied to $u = f$ we can easily check that 
$
	\DOp[f] f(x) \dd f(x)
	=
	\dd f(x) \dd f(x)\pinv \dd f(x) 
	=
	\dd f(x)
$,
$
	\DOp[f] f(x) \DOp[f] f(x)  = \DOp[f] f(x)
$ and $
	\DOp[f] f(x)^\transp = \DOp[f] f(x)
$. 
Thus, 
\begin{align}
	\label{eq:Tpr}
	\Tpr_f(x) 
	\ceq  
	\DOp[f] f(x) 
	\qand
	\Npr_f (x) \ceq  \id_{\AmbSpace} - \Tpr_f (x)
\end{align}
are the orthogonal projectors onto the embedded tangent space $T_{f(x)} \varSigma = \dd f( T_x \Domain)$
and onto the normal space $(T_{f(x)} \varSigma)^\perp = (\dd f( T_x \Domain))^\perp$, respectively.
Together with the Riemannian volume measure $\vol_f$ induced by $f$, this allows us to rewrite the energy as
\begin{align}\label{eq:gTPpar}
	\Energy_{p}(f)
	\ceq 
	\Energy_p(f(\Domain))
	=
	\int_\Domain \! \int_\Domain
		\frac{ 
			\abs{ \Npr_f(x) \, (f(y)-f(x)) }^{p}
		}{
			\abs{ f(y) - f(x) }^{2p}
		}
	\dvol_f(x) \dvol_f(y)
	.
\end{align}
It will turn out useful to introduce the following operator:
\begin{align}
	\RemainderOp[s][f] u (x,y)
	\ceq
	\frac{u(y)-u(x) - \DOp[f] u(x) \, (f(y)- f(x)) }{\abs{f(y) - f(x)}^s}
	\quad 
	\text{for $1 < s < 2$.}
	\label{def:Op}
\end{align}
One should read the term 
$
	u(x) + \DOp[f] u(x) \, (f(y)- f(x))
$
as the first-order Taylor approximation of $u(y)$ at center $x$. Thus, $\RemainderOp[s][f] u (x,y)$ measures the \emph{remainder} of this Taylor approximation against a fractional power of the chord distance $\abs{f(y) - f(x)}$.
This operator is tightly connected to the energy because we may write the following for $s \ceq 2 - \DomDim/p$:
\begin{equation*}
	\frac{ 
			\abs{ \Npr_f(x) \, (f(y)-f(x)) }^{p}
		}{
			\abs{ f(y) - f(x) }^{2p}
		}
	=
	\frac{ 
		\abs{ f(y)-f(x) - \DOp[f] f(x) \, (f(y)-f(x)) }^{p}
	}{
		\abs{ f(y) - f(x) }^{2p}
	}
	=
	\frac{\abs{ \RemainderOp[s][f] f(x,y) }^p }{\abs{ f(y) - f(x) }^\DomDim}
	.
\end{equation*}
Now, with the measure $\mu_f$ from \cref{eq:DeltaOpAndmuf}, we can write the energy in compactly form as follows:
\begin{equation}
	\Energy_p^s(f)
	\ceq
	\int_\Domain \! \int_\Domain
		\abs{ \RemainderOp[s][f] f(x,y) }^p
	\dd \Measure[f](x,y)
	\label{eq:gTPparOp}
	.
\end{equation}
In fact, this energy makes sense also for $s \neq 2 - \DomDim/p$. It corresponds to the so-called \emph{generalized tangent-point energy} introduced by Blatt and Reiter in \cite{MR3331696}. 
Provided that $s < 2$ and $\alpha \ceq s - \DomDim/p > 1$ (which enables the Morrey embedding $\Sobo[s,p] \hookrightarrow \Holder[1,\alpha]$), it has all the properties outlined in \cref{sec:TPE}.
For $s > 3/2$ and $\DomDim = 1$ it allows us to set $p = 2$ without losing these properties. 
This turns the energy space into the Hilbert space $\Sobo[s,2] = \Bessel[s]$. 
In fact, it was the motivation for introducing the generalized tangent-point energy.
Alas, because of the restriction $s< 2$, this does not help us with dealing with surfaces ($\DomDim=2$). So we have to settle with the Banach space setting.

\subsection{Differential of the tangent-point energy}
\label{sec:DEnergy}

In this section we analyze smoothness properties of the tangent-point energy and its derivative. 
We will see that the tangent-point energy is sufficiently smooth to apply techniques from ``smooth optimization'', which typically refers to methods that require that the objective function is of class $\Holder[1,1][\loc]$ or $\Holder[2,1][\loc]$.
This fact is somewhat folklore (see for example \cite{2501.16647} for the case $\DomDim = 1$), but we have not found any detailed account for it in the literature. 
So we include a derivation here.
A central role in all this is played by the operator $\RemainderOp[s]$, and thus we have to study it first.
Given two sufficiently smooth embeddings $f$, $f_0 \colon \Domain \to \AmbSpace$, we can express $\RemainderOp[s][f]$ in terms of $\RemainderOp[s][f_0]$ as follows:
\begin{equation}
	\RemainderOp[s][f] u(x,y)
	=
	\varLambda_f^s (x,y)
	\pars[\Big]{
		\RemainderOp[s][f_0] u(x,y)
		-
		\DOp_{f} u (x) \RemainderOp[s][f_0]f(x,y)
	},
	\label{eq:MagicFormula}
\end{equation}
with $\varLambda_f$ as defined in \cref{eq:Lambda}. 
Indeed, this can be verified by 
expanding the definition of $\RemainderOp[s]$
and by recalling from \cref{eq:FortunateIdentity} that $\DOp[f] u(x) = \DOp[f_0] u(x) \DOp[f] f_0(x)$:
\begin{align*}
	\MoveEqLeft[0]
	\abs{f_0(y)-f_0(x)}^s
	\pars[\Big]{
		\RemainderOp[s][f_0] u (x,y)
		-
		\DOp_{f} u(x) \RemainderOp[s][f_0] f (x,y)
	}
	\\
	&=
	\pars[\Big]{
		u(y) - u(x) - \DOp[f_0] u(x) \pars{f_0(y) - f_0(x)}
	}
	\\
	&\qquad
	-
	\DOp[f] u(x) \pars[\Big]{f(y) - f(x) - \DOp[f_0] f(x) \pars{f_0(y) - f_0(x)} }
	\\
	&=
	u(y) - u(x) - \DOp[f] u(x) \pars{f(y) - f(x)}
	=
	\abs{f(y)-f(x)}^s \RemainderOp[s][f] u (x,y)
	.
\end{align*}
Another central property of $\RemainderOp[s]$ is that it is covariant in the following sense: Let $f \colon \Domain \to \AmbSpace$ be a $\Holder[1]$-embedding, let $u \in \Holder[1][][\Domain][\AmbSpace]$, and let $\varphi \colon \Domain \to \Domain$ be a $\Holder[1]$-diffeomorphism. 
Then by \cref{eq:DfIsCovariant} we have the following for all $x$, $y \in \Domain$, $x \neq y$:
\begin{equation*}
	\RemainderOp[s][f \circ \varphi] (u \!\circ\! \varphi)(x,y)
	=
	u(\varphi(y)) 
	\!-\!
	u(\varphi(x))
	\!-\!
	\DOp[f] u(\varphi(x)) \pars{f(\varphi(y)) \!-\! f(\varphi(x))}
	=
	\RemainderOp[s][f] u( \varphi(x), \varphi(y) )
	.
\end{equation*}

\begin{lemma}\label{lem:RfsIsBounded}
	Let $1 < s < 2$ and $1 < p \leq \infty$ such that $s - \DomDim/p > 1$.
	Then for every $f \in \Embsp$
	the linear operator $\RemainderOp[s][f] \colon \Sobo[s,p][][\Domain][\AmbSpace] \to \Lebesgue[p][\Measure[f]][\Domain\times \Domain][\AmbSpace]$ is bounded, where $\Measure[f]$ is defined as in \cref{eq:DeltaOpAndmuf}.
\end{lemma}
\begin{proof}
	Recall also the definition of $\mu$ from \cref{eq:deltaOpAndmu} and of the Sobolev norms from \cref{eq:SoboSloboNormEulcidean} and \cref{eq:SoboSloboNormLow}.
	Note that $\Domain$ is compact, hence the space $\Lebesgue[p][\Measure[f]][\Domain\times \Domain][\AmbSpace]$ does not really depend on the choice of $f$.
	Since $\RemainderOp[s]$ is covariant, we may employ charts. 
	Without loss of generality, we may assume that $f$ is a graph of some function $h \in \Sobo[s,p][][V][\NormalSpace]$ over an open ball $V \subset \DomSpace$ around $0$, where $h \in \Sobo[s,p][][V][\NormalSpace]$.
	More precisely, we may assume that $f(x) = (x,h(x))$ for $x \in V$.
	We can arrange it so that $h(0)=0$ and $Dh(0) = 0$. Since $Dh$ is continuous, we may shrink $V$ if necessary to achieve $\norm{D h}_{\Lebesgue[\infty][][V]} \leq 1/2$.
	Moreover, we denote by $\prx\colon(x,y)\mapsto x$ the projection onto the first coordinate.
	Choosing $f_0 \colon V \to \AmbSpace$, $f_0(x) \ceq (x,0)$, one sees that \cref{eq:MagicFormula} turns into 
	\begin{align}
		\RemainderOp[s][f] u
		=
		\pars[\big]{1 + \abs{\deltaOp[1] h}^2}^{-\frac{s}{2}}
		\pars*{
			\remainderOp[s] u
			-
			\pars[\Big]{\pars[\big]{Du \pars{\id_\DomSpace + Dh^\transp Dh}^{-1} Dh^\transp }\circ \prx}
			\remainderOp[s] h
		},
		\label{eq:RsfInGraphPatch}
	\end{align}
	where
	\begin{equation}
		\remainderOp[s] u(x,y) 
		\ceq 
		\frac{ u(y) - u(x) - Du(x) \pars{y-x} }{\abs{y-x}^{s}} .
		\label{eq:DefremainderOp}
	\end{equation}
	Since $h$ has slope $\leq 1/2$, one has $(4/5)^{s/2} \leq \pars[\big]{1 + \abs{\deltaOp[1] h}^2}^{-s/2} \leq 1$. 
	Moreover, one can find $0 < c \leq C < \infty$ such that
	$c\, \mu \leq \Measure[f] \leq C \, \mu$.
	Hence, (with some different $C > 0$) we get that
	\begin{equation}
		\norm{ \RemainderOp[s][f] u }_{\Lebesgue[p][][V\times V](\Measure[f])}
		\leq 
		C   \pars[\Big]{
			\norm{ \remainderOp[s] u }_{\Lebesgue[p][][V \times V](\mu)}
			+
			\norm{Du}_{\Lebesgue[\infty][][V](\mu)}
			\norm{ \remainderOp[s] h }_{\Lebesgue[p][][V\times V](\mu)}
		}
		.
		\label{eq:RsfInGraphPatch2}
	\end{equation}
	\cref{lem:EuclideanRsIsBounded} below implies
	\[ 
		\norm{ \remainderOp[s] u }_{\Lebesgue[p][][V \times V](\mu)} 
		\leq 
		C \norm{D u}_{\Sobo[s-1,p][][V]}
		\qand
		\norm{ \remainderOp[s] h }_{\Lebesgue[p][][V\times V](\mu)}
		\leq 
		C \norm{D h}_{\Sobo[s-1,p][][V]}
		.
	\]
	Morrey's inequality \cref{eq:MorreyInequality} yields
	$\norm{ D u }_{\Lebesgue[\infty][][V]} \leq C \norm{ Du }_{\Sobo[s-1,p][][V]} = C \norm{ \DOp[f_0]  u }_{\Sobo[s-1,p][][V](f_0)}$, thus
	\begin{equation*}
		\norm{ \RemainderOp[s][f] u }_{\Lebesgue[p][][V\times V](\Measure[f])}
		\leq 
		C \pars[\big]{ 
			1
			+
			\norm{ Dh }_{\Sobo[s-1,p][][V]}
		}
		\norm{ \DOp[f_0]  u }_{\Sobo[s-1,p][][V](f_0)}
		.
	\end{equation*}
	Since the slope of $h$ is bounded and since $\overline V$ is compact, we have
	$\norm{ \DOp[f_0]  u }_{\Sobo[s-1,p][][V](f_0)} \leq C \norm{ \DOp[f] u }_{\Sobo[s-1,p][][V](f)}$. This yields the local bound
	\begin{equation*}
		\norm{ \RemainderOp[s][f] u }_{\Lebesgue[p][][V\times V](\Measure[f])}
		\leq 
		C \pars[\big]{ 
			1
			+
			\norm{ Dh }_{\Sobo[s-1,p][][V]}
		}
		\norm{ \DOp[f] u }_{\Sobo[s-1,p][][V](f)}
		.
	\end{equation*}
	To get a global bound, we cover the diagonal $\myset{ (x,x) }{ x \in \Domain }$ by finitely many graph patches.
	Then there is a $\delta > 0$ so that the complement of all these patches is contained in $W_\delta \ceq \myset{(x,y) }{ \abs{f(y)-f(x)} > \delta}$.
	For $(x,y) \in W_\delta$ we have
	\begin{align*}
		\dd \Measure[f](x,y) \leq \frac{\vol_f(x) \vol_f(y)}{\delta^{\DomDim}}
		\;\;\text{and}\;\;
		\abs{\RemainderOp[s][f] u(x,y)}
		\leq 
		\frac{
			\pars{\abs{u(y)}  + \abs{u(x)}} + \delta^{-(s-1)}\abs{\DOp[f] u(x)}
		}{
			\delta^{s} 
		}
		,
	\end{align*}
	which allows us to bound $\norm{ \remainderOp[s] u }_{\Lebesgue[p][][W_\delta](\mu)}$ in terms of $\norm{u}_{\Lebesgue[p][](f)}$ and $\norm{\DOp[f] u}_{\Lebesgue[p][](f)}$.
\end{proof}

\begin{lemma}\label{lem:EuclideanRsIsBounded}
	Let $V \subset \DomSpace$ be a bounded, convex, open set, let $1 < s < 2$, and let $1 \leq p \leq \infty$.
	Then there is a constant $C = C(s,p)$ such that holds true for the operator $\remainderOp[s]$ from \cref{eq:DefremainderOp}:
	\begin{equation}
		\norm{\remainderOp[s] u}_{\Lebesgue[p][][V \times V](\mu)}
		\leq
		C \norm{D u}_{\Sobo[s-1,p][][V]}
		\quad 
		\text{for all $u \in \Sobo[s,p][][V][\AmbSpace]$.}
		\label{eq:RsfInGraphPatch3}
	\end{equation}
\end{lemma}
\begin{proof}
	We employ the technique from \cite[p.~264]{zbMATH06214305}.
	The fundamental theorem of calculus implies
	\[
		\remainderOp[s] u(x,y) = \frac{1}{\abs{y-x}^{s}} \int_0^1 
			\pars[\big]{ Du(x + t \, (y-x)) - Du(x) } \pars{y-x} 
		\dd t.
	\]
	For $p = \infty$ we may use 
	$\abs{ Du(x + t \, (y-x)) - Du(x) } \leq \abs{x + t \, (y-x) - x}^{s-1} \norm{D u}_{\Sobo[s-1,\infty][][V]}
	\leq t^{s-1} \abs{y-x}^{s-1} \norm{D u}_{\Sobo[s-1,\infty][][V]}$, thus
	\begin{align*}
		\abs{\remainderOp[s] u(x,y)}
		\leq 
		\pars*{\textstyle\int_0^1 \! t^{s-1} \dd t} \norm{u}_{\Sobo[s,\infty][][V]} 
		\qand
		\norm{\remainderOp[s] u}_{\Lebesgue[\infty][][V \times V]}
		\leq
		s^{-1} \norm{D u}_{\Sobo[s-1,\infty][][V]}
		.
	\end{align*}
	For $p < \infty$ we need a bit more effort. Jensen's inequality and Fubini's theorem yield
    \begin{align*}
        \norm{ \remainderOp[s] u }_{\Lebesgue[p][][V\times V	](\mu)}^p 
        \leq 
        \int_0^1 \!\!\!\! \int_{V} \! \int_{V} 
            \frac{ \abs{ D u( x + t \, (y-x)) - D u(x)  }^p }{ \abs{y-x}^{(s-1)\,p + \DomDim } }
        \dd y \dd x \dd t 
        .
    \end{align*}
    Next we substitute $z = x + t \, (y - x)$. 
    With $V(t,x) \ceq \myset{ z = x + t \, (y - x) }{ y \in V}$, the above equals
    \begin{align*}
        \int_0^1 \!\!\!\! \int_{V} \! \int_{V(t,x)} 
        \frac{
            \abs{Du(z) - Du(x) }^p 
        }{
            \abs{(z - x)/t}^{(s-1) \, p + \DomDim}
        } 
        \frac{\dd z}{t^n} \dd x  \dd t
		\leq
		\int_0^1 \!\!\!\! \int_{V} \! \int_{V} 
        \frac{
            \abs{Du(z) - Du(x) }^p 
        }{
            \abs{(z - x)/t}^{(s-1) \, p + \DomDim}
        } 
        \frac{\dd z}{t^n} \dd x  \dd t
		.
    \end{align*}
    Here we used that 
	(i) $V$ is convex so that the set $V(t,x)$ is a subset of $V$ and that 
	(ii) the integrand is nonnegative. 
	Finally, we obtain the desired statement for arbitrary $p$:
    \begin{align*}
        \norm{ \remainderOp[s] u }_{\Lebesgue[p][][V \times V](\mu)}^p 
        \leq
        \int_0^1\!\!\!\!\int_{V} \! \int_{V} 
            \frac{\abs{Du(z) - Du(x) }^p }{\abs{(z - x)/t}^{(s-1)p}} 
        \frac{\dd x \dd z}{\abs{z-x}^\DomDim} \dd t
        =
        \pars*{ \textstyle\int_0^1 \! t^{(s-1) p}  \dd t} \norm{ Du }_{\Sobo[s-1,p][][V]}^p
        .
    \end{align*}
\end{proof}

\begin{remark}
	For $u = f$ \cref{lem:RfsIsBounded} shows that every $f \in \Emb[s,p][][\Domain][\AmbSpace]$ has finite energy 
	\begin{equation*}
		\Energy_p(f) = \norm{\RemainderOp[s][f] f}_{\Lebesgue[p][](\Measure[f])}^p < \infty.
	\end{equation*}
	This is the first half of the proof for \cref{thm:EnergySpace}.
	To give at least an idea of the second half, we sketch the following two ingredients:
	First, 	\cref{eq:RsfInGraphPatch} for $u = h$ and with $\norm{Dh}_{\Lebesgue[\infty][][V]} \leq 1/2$ yields the pointwise bound of the form 
	$\abs{\RemainderOp[s][f] f} \geq c \abs{\remainderOp[s] h }$ with some $c>0$.
	Hence, (by shrinking $c$ if necessary) we obtain
	$\norm{\RemainderOp[s][f] f}_{\Lebesgue[p][][V \times V](\Measure[f])} \geq c \norm{\remainderOp[s] h}_{\Lebesgue[p][][V\times V](\mu)}$.
	Second, for $v \in \Sobo[s,p][][\DomSpace][\NormalSpace]$ we have
	\begin{gather*}
		\textstyle
		\pars*{
			\int_\DomSpace \! \int_\DomSpace
			\abs[\big]{ \frac{v(x + \xi) - 2 v(x) + v(x - \xi)}{\abs{\xi	}^s} }^p
			\dd x \, \frac{\dd \xi}{\abs{\xi}^{\DomDim}}
		}^{1/p}
		=
		\pars*{
			\int_\DomSpace \! \int_\DomSpace
			\abs[\big]{
				\remainderOp[s] v(x,x + \xi) 
				+
				\remainderOp[s] v(x,x - \xi)
			}^p
			\dd x \, \frac{\dd \xi}{\abs{\xi}^{\DomDim}}
		}^{1/p}
		\\
		\qquad 
		\leq 
		\textstyle
		\pars*{
			\int_\DomSpace \! \int_\DomSpace
			\abs[\big]{
				\remainderOp[s] v(x,x + \xi) 
			}^p
			\dd x \, \frac{\dd \xi}{\abs{\xi}^{\DomDim}}
		}^{1/p}
		+
		\pars*{
			\int_\DomSpace \! \int_\DomSpace
			\abs[\big]{
				\remainderOp[s] v(x,x - \xi) 
			}^p
			\dd x \, \frac{\dd \xi}{\abs{\xi}^{\DomDim}}
		}^{1/p}
		.
	\end{gather*}
	In \cite[Theorem~11.27]{zbMATH07647941} Leoni shows that the first term is equivalent to $\norm{Du}_{\Sobo[s-1,p][][\DomSpace]}$.
	The substitutions $y = x + \xi$ and $y = x - \xi$ show that each of the last two terms
	equals $\norm{\remainderOp[s] v}_{\Lebesgue[p][][\DomSpace \times \DomSpace](\mu)}$.
	Now we only have to localize this. 
	Assuming that the graph patch $f(x) = (x,h(x))$ is supported in $V = B(0,\varrho) \subset \DomSpace$, $\varrho>0$,
	we choose
	$W \ceq \OpenBall{0}{\varrho/2} \subset V$ and $\chi \in \Holder[\infty][][\DomSpace][\intervalcc{0,1}]$ satisfying $\chi|_W = 1$ and $\supp(\chi) \subset \OpenBall{0}{3/4\, \varrho}$. 
	Then we define $v \in \Sobo[s,p][][\DomSpace][\NormalSpace]$ as the extension of $\chi  h$ by $0$. Now, since $\DomSpace \setminus V$ and $\supp(\chi)$ have minimal distance $\geq \varrho/4 > 0$, we get
	\begin{align*}
		\norm{h}_{\Sobo[s,p][][W]}
		\leq 
		\norm{v}_{\Sobo[s,p][][\DomSpace]}
		\leq 
		C \norm{\remainderOp[s] v}_{\Lebesgue[p][][\DomSpace \times \DomSpace](\mu)}
		\leq 
		C' \pars[\big]{
			\norm{\remainderOp[s] h}_{\Lebesgue[p][][V \times V](\mu)}
			+
			\norm{ h}_{\Sobo[1,p][][V]}
		}
		,
	\end{align*}
	which suffices for showing that $f \in \Sobo[s,p][][\Domain][\AmbSpace]$.
	We leave the remaining details to the reader.
\end{remark}

Next we show that $f \mapsto \RemainderOp[s][f]$ is analytic in $f$. For curves this has been shown already in \cite[Section~2]{2511.07214}.

\begin{lemma}\label{lem:RfsIsSmooth}
	Under the conditions of \cref{lem:RfsIsBounded} the following map is analytic:
	\begin{align*}
		\RemainderOp[s] 
		\colon 
		\Emb[s,p][][\Domain][\R^{m}]
		\to 
		\ContOps \pars[\big]{
			\Sobo[s,p][][\Domain][\AmbSpace]
			;
			\Lebesgue[p][\Measure[f_0]][\Domain \times \Domain][\AmbSpace]
		}
		,
		\quad
		f \mapsto \RemainderOp[s][f]
		,
	\end{align*}
	where $\ContOps(X;Y)$ denotes the normed vector space of bounded, linear operators from normed space $X$ to normed space $Y$, equipped with the operator norm.
	The derivative in direction $v \in \Sobo[s,p][][\Domain][\AmbSpace]$ is given by
	\begin{align}
		D \pars{ f \mapsto \RemainderOp[s][f] u} v
		=
		- \pars{\DOp[f] u \!\circ\! \prx} \RemainderOp[s][f] v
		\!-\! \pars{\DOp[f] u \!\circ\! \prx} \pars{\DOp[f] v \!\circ\! \prx}^\transp
		\!\RemainderOp[s][f] f
		\!-\! 
		s \RemainderOp[s][f] u
		\inner[\big]{ 
		    \DeltaOp[1][f] \! f
		    , 
		    \DeltaOp[1][f] v
		}
		.
		\label{eq:DRsf}
	\end{align}
\end{lemma}
\begin{proof}
	We employ \cref{eq:MagicFormula} to play this back to the smoothness of the following maps for $\beta > 0$:
	\begin{gather*}
		\varLambda^\beta \colon \Emb[s,p][][\Domain][\AmbSpace] \to \Lebesgue[\infty][][\Domain][\R], 
		\quad 
		f \mapsto \varLambda_{\smash{f}}^\beta,
		\quad 
		\text{and}
		\\
		\DOp 
		\colon 
		\Sobo[s,p][][\Domain][\AmbSpace] 
		\to 
		\ContOps \pars[\big]{
			\Sobo[s,p][][\Domain][\AmbSpace]
			;
			\Lebesgue[\infty][][\Domain][\Hom(\AmbSpace;\AmbSpace)]
		}
		,
		\quad
		f \mapsto \DOp[f]
		.
	\end{gather*}
	For $\varLambda^\beta$ we use the Morrey embedding $\Emb[s,p][][\Domain][\AmbSpace] \subset \Holder[1,\alpha][][\Domain][\AmbSpace]$ to see that the map
	$\varLambda^\beta$ is continuous.
	It is then elementary to compute the Gateaux derivative
	(see \cref{eq:DeltaOpAndmuf} for the definition of $\DeltaOp[1][f]$):
	\begin{align*}
		D\varLambda^\beta_{\smash{f}}(f) \, u 
		= 
		- \beta \inner{ \DeltaOp[1][f] \! f, \DeltaOp[1][f] u } \, \varLambda_{\smash{f}}^\beta
		=
		- \beta \inner{ \DeltaOp[1][f_0] f, \DeltaOp[1][f_0] u } \, \varLambda_{\smash{f}}^{\beta+2}
		\quad 
		\text{for $u \in \Sobo[s,p][][\Domain][\AmbSpace]$.}
	\end{align*}
	So, we see that $\smash{D \varLambda^\beta_f}$ is continuous, too, and hence $\varLambda^\beta$ is Fréchet differentiable.
	Since $\varLambda^\beta_{\smash{f}}$ is uniformly bounded away from $0$, it follows that $\varLambda^\beta$ is analytic.
	For analyzing $\DOp$, we observe
	\[
		\cD_f u 
		= 
		\dd u \pars{\dd f}\pinv
		=
		\dd u \pars{\dd f^{*} \dd f}^{-1} \dd f^{*}
		,
	\]
	where the adjoint is with respect to the pullback metric $g_{f_0}$.
	This shows that $(f,u) \mapsto \cD_f u $ is analytic as a map to $\Lebesgue[\infty]$ and that
	\begin{equation}
		\pars[\big]{ D \pars{f \mapsto \DOp[f] u }(f) \, v }
		=
		- \DOp[f] u \DOp[f] v
        +
        \DOp[f] u \DOp[f] v^\transp \pars{ \id_\AmbSpace - \DOp[f] f }
		.
		\label{eq:DDf}
	\end{equation}
	Now recursion shows the smoothness of $\DOp$.
	Finally, we apply the product rule to \cref{eq:MagicFormula} to show smoothness of $\RemainderOp[s]$ and to derive the formula \cref{eq:DRsf}.
\end{proof}

Continuous Fréchet diffentiability of the generalized tangent-point energy 
\[
	\TP^{p s + \DomDim,p}(f) = \norm{\RemainderOp[s][f] f}_{\Lebesgue[p](\mu_f)}^p
\]
has been shown with quite involved techniques in \cite[Theorem~4.29]{WingsPHDThesis}.
The tangent-point energy $\Energy_p$ is a special instance of it.
The smoothness $f \mapsto \RemainderOp[s][f]$ makes it now much easier to analyze the differentiablity of $\Energy_p$.
For $p = 2$ and $\DomDim = 1$, analycity of $\TP^{2 s + 1,2}$ was used in \cite[Section~2]{2511.07214} to establish a Łojasiewicz-Simon gradient inequality.

\begin{proposition}\label{prop:FrechetDerivative}
Let $2 \, \DomDim < p < \infty$ and $s \ceq 2 - \DomDim/p$ such that $\alpha \ceq s - \DomDim/p > 1$. 
Then the energy $\Energy_p$ is $\ceil{p-1}$-times continuously Fréchet differentiable on $\Embsp$ (and analytic if $p$ is an even integer).
The derivative is given by
\begin{equation}
	D\Energy_p(f)\,v = p \, \cB_1(f)(f,v)- 2\,p \, \cB_2(f)(f,v) + \cB_3(f)(f,v),
	\label{eq:DE}
\end{equation}
where the bilinear forms $\cB_1(f)$, $\cB_2(f)$, and $\cB_3(f)$ are given by
\begin{align*}
	\cB_1(f) ( u, w )
	& \ceq 
	\int_\Domain \! \int_\Domain
		\abs{ \RemainderOp[s][f] f }^{p-2} 
		\inner{ \RemainderOp[s][f] u, \RemainderOp[s][f] w}
	\dd \Measure[f] 
	,
	\\
	\cB_2(f) ( u, w )
	& \ceq 
	\int_\Domain \! \int_\Domain
		\abs{ \RemainderOp[s][f] f }^{p} 
		\inner{ \DeltaOp[1][f] u, \DeltaOp[1][f] w }
	\dd \Measure[f]
	,
	\\
	\cB_3(f) ( u, w )
	&\ceq 
	\int_\Domain \! \int_\Domain
		\abs{ \RemainderOp[s][f] f }^{p} 
		\pars[\Big]{
			\inner[\big]{ \DOp[f] u, \DOp[f] w} \circ \prx 
			+ 
			\inner[\big]{ \DOp[f] u, \DOp[f] w} \circ \pry 
		}
    \dd \Measure[f]
	,
\end{align*}
where $\prx(x,y) = x$, $\pry(x,y) = y$ 
denote the projections onto the Cartesian factors.
\end{proposition}
\begin{proof}
	We fix some arbitrary embedding $f_0 \in \Emb[s,p][][\Domain][\AmbSpace]$ and 
	express the Riemannian volume measure $\vol_f(x)$ of $f$ in terms 
	of Riemannian volume measure $\vol_{f_0}(x)$ of $f_0$ 
	and
	the Jacobian $J_f(x) \ceq \dvol_{f} \big/ \dvol_{f_0} = \sqrt{\det( \dd f(x)\adj \dd f(x) )}$
	(where the adjoint is with respect to the pullback metric $g_{f_0}$).
	Thus,
	\[
		\Energy_p(f)
		=
		\int_\Domain \! \int_\Domain
			\abs{ \RemainderOp[s][f] f (x,y)}^{p} 
			\,
			\varLambda_f^\DomDim(x,y)
			\, J_f(x) \, J_f(y)
		\dd \Measure[f_0](x,y) 
		.
	\]
	We know from \cref{lem:RfsIsSmooth} that 
	$f \mapsto \varLambda^\DomDim_{\smash{f}} $
	and
	$f \mapsto \RemainderOp[s][f] f$ are analytic.
	Because $J_f$ can be represented in terms of a determinant and a square root,
	and because $\dd f^* \dd f$ is uniformly bounded away from $0$,
	the map $J \colon \Emb[s,p][][\Domain][\AmbSpace] \to \Lebesgue[\infty][][\Domain][\R]$, $f \mapsto J_f$ is analytic. A~straightforward calculation reveals that $DJ(f) \, u = \smash{\inner{\DOp[f] f,\DOp[f] u} \, J_{\smash{f}}}$.
	Now the formula \cref{eq:DE} for $D \Energy_p$ can be derived from this 
	by using the chain and the product rules and by keeping in mind that $\RemainderOp[s][f] f(x,y)$ is always perpendicular to the image of $\dd f(x)$. 
	We also see that the only limitation to smoothness of $\Energy_p$ is the limited smoothness of
	$z \mapsto \abs{z}^p$.
\end{proof} 


\subsection{Sobolev metric}
\label{sec:Preconditioner}

As already sketched in the introduction, we seek to use the tangent-point energy $\Energy \ceq \Energy_p$ for Tikhonov regularization.
To make use of this approach by numerically minimizing the functional (see \cref{sec:Optimization}), we need a suitable preconditioner, which can be derived from a preconditioner for the gradient descent minimization of the tangent-point energy.
Such a preconditioner was proposed and tested in \cite{RepulsiveC} and \cite{RepulsiveS}, where the approach taken there is motivated by \cite{zbMATH07456145}:
Instead of the Hessian $D^2 \Energy(f)$ on  $X \times X$,
a bilinear form $\Metric_f$ on  $X \times Y$ is employed. 
Here $Y \supset X$ is some further Banach space.
The difficulty is to guarantee that the associated Riesz operator $\Riesz_f \colon X \to Y'$ is continuously invertible and that $D \Energy(f) \in Y' \subset X'$.
If this is the case and if $\Metric_f$ is chosen to be symmetric and positive-definite on $X \times X \subset X \times Y$, then the update direction, the \emph{downward Sobolev gradient}, $v(f) \ceq - \Riesz_f^{-1} D \Energy(f)$ is guaranteed to be descending (unless $f$ is a critical point).
Moreover, the following are desirable:
\begin{enumerate}[label=(D\arabic*),ref=(D\arabic*),noitemsep]
	\item The update direction $v(f)$ is at least as regular as $f$.
	\label{item:DesirableRegularity}
	\item The map $f \mapsto v(f)$ is locally Lipschitz continuous (or even smooth).
	\label{item:DesirableLipschitz}
	\item The construction is covariant in the sense that $v(f \circ \varphi) = v(f) \circ \varphi$ for all diffeomorphisms $\varphi \colon \Domain \to \Domain$.
	\label{item:DesirableCovariance}
\end{enumerate}
The first condition guarantees that no regularity is lost, so that further optimization steps can be taken one after the other. 
The second condition guarantees stability of the method, provided that some rudimentary line search is performed. 
While the third condition is not strictly necessary, it is still natural to require, since the displacement field $v(f) \circ f^{-1} \colon \varSigma \ceq f(M) \to \AmbSpace$ is then independent of the choice of the parameterization.

Typically, fulfilling these three conditions leads to some nice robustness of the algorithm against the choice of discretization, in particular against the choice of meshes representing the surface $\varSigma$.
In particular, the first condition prevents the preconditioning quality of $\Metric$ from deteriorating under mesh refinement. 

\medskip

We now turn to the construction of such a Sobolev metric $\Metric$. Since this construction involves various technicalities that are beyond the scope of the present work, we only sketch the relevant ideas.

\subsubsection{\texorpdfstring{$D\Energy$}{DEp} as a nonlinear operator of order \texorpdfstring{$2 \, s$}{2s}}
\label{sec:DEAsNonlinearOperator}

The first observation is that under certain circumstances, the operator $f \mapsto D \Energy(f)$ can be interpreted as a nonlinear differential operator of fractional order $2 \, s$.
(This has to be taken with a grain of salt; as we will see, the order of differentiability of such a nonlinear operator is a somewhat flexible concept.)

\begin{proposition}\label{prop:DerivativeExtended}
	Let $2 \, \DomDim < p < r < \infty$, $s=2-\DomDim/p > 1$, and $\DomDim/r \leq \diffoffset < \DomDim/p $.
	Then for every $f \in \Emb[s+\diffoffset,r][][\Domain][\AmbSpace]$,
	the linear form $D\Energy_p(f)$ can be extended to a linear form 
	$\Extension_p(f) \in (\Sobo[s-\diffoffset,r'\!][][\Domain][\AmbSpace])'$,
	where $r' \ceq r / (r-1)$ is the Hölder conjugate of $r$.
	The nonlinear operator
	\begin{equation*}
		\Extension_p \colon \Emb[\diffoffset+s,r][][\Domain][\AmbSpace]
		\to  
		(\Sobo[s-\diffoffset,r'\!][][\Domain][\AmbSpace])'
		= \Sobo[\diffoffset-s,r][][\Domain][\AmbSpace] 
	\end{equation*}
	is $\ceil{p-2}$-times continuously differentiable (and smooth if $p$ is an even integer).
\end{proposition}

\begin{proof}
	Recall from \cref{eq:DE} that the derivative in direction $v$ can be written as
	$D\Energy_p(f) \, v = p \, \cB_1(f)(f,v)- 2\,p \, \cB_2(f)(f,v) + \cB_3(f)(f,v)$.
	The existence of the extension and certain smoothness results for $\cB_1$, $\cB_2$, $\cB_3$ are shown in \cref{prop:BilinearFormsExtended} below.
	For the smoothness observe that albeit $\cB_1$ is only $\ceil{p-3}$-times differentiable, $\cB_1(f)(f,\cdot)$ yields a further derivative 
	since the map $z \mapsto \abs{z}^{p-2} \, z$ is $\ceil{p-2}$-times continuously differentiable.
\end{proof}

\begin{proposition}\label{prop:BilinearFormsExtended}
	Under the conditions of \cref{prop:DerivativeExtended}
	each of $\cB_1$, $\cB_2$, $\cB_3$  from \cref{prop:FrechetDerivative}
	constitutes a bounded bilinear form 
	\begin{equation*}
		\cB_k(f) \colon \Sobo[s+\diffoffset,r][][\Domain][\AmbSpace] \times \Sobo[s-\diffoffset,r'\!][][\Domain][\AmbSpace] \to \R,
	\end{equation*}
	where $r' \ceq r/(r-1)$ is the Hölder conjugate of $r$.
	Moreover, the map
	\begin{equation*}
		\cB_1 \colon \Emb[s+\diffoffset,r][][\Domain][\AmbSpace] \times \Sobo[s+\diffoffset,r][][\Domain][\AmbSpace] \times \Sobo[s-\diffoffset,r'\!][][\Domain][\AmbSpace] \to \R,
		\quad 
		(f,u,w) \mapsto \cB_1(f)(u,w)
	\end{equation*}
	is $\ceil{p-3}$-times continuously differentiable,
	and the corresponding maps $\cB_2$ and $\cB_3$ are $\ceil{p-1}$-times continuously differentiable (and smooth if $p$ is an even integer).
\end{proposition}
\begin{proof}
	For well-definedness it suffices to show that there are constants $C_k(f) \geq 0$ such that
	\begin{align}
		\abs{ \cB_k(f)(u,w) } 
		\leq 
		C_k(f) \norm{u}_{\Sobo[s+\diffoffset,r](f)}  \norm{w}_{\Sobo[s-\diffoffset,r'\!](f)}
		,
		\quad 
		k \in \braces{1,2,3}
		.
		\label{eq:BkInequality}
	\end{align}
	We start with $\cB_1$.
	Let $u\in \Sobo[s+\diffoffset,r][][\Domain][\AmbSpace]$ 
	and 
	$w\in\Sobo[s-\diffoffset,r'\!][][\Domain][\AmbSpace]$. 
	Now we rearrange the exponents in the denominators of $\RemainderOp[s][f] f$ and $\RemainderOp[s][f] w$ to arrive at the following:
	\begin{align*}
		\cB_1(f) ( u, w )
		& =
		\int_\Domain\!\int_\Domain
			\abs{ \RemainderOp[s+\varepsilon][f] f }^{p-2} 
			\inner{ \RemainderOp[s+\varepsilon][f] u, \RemainderOp[s-\diffoffset][f]w}
		\dd \Measure[f]
		\qwith
		\varepsilon \ceq \fdfrac{\diffoffset}{p-1}
		.
	\end{align*}
	We note that $p > 2$, thus $s + \varepsilon < s + \diffoffset < s + \DomDim/p < 2$ and
	$s - \diffoffset > s - \DomDim/p > 1$.
	So the upper indices of $\RemainderOp[*][f]$ lie in the interval $\intervaloo{1,2}$ as they should.
	Now we put
	\[ 
		q \ceq (p-1)\,r
		\qand 
		a \ceq \fdfrac{\pars{p-1}}{\pars{p-2}} \,r 
		= \fdfrac{q}{\pars{p-2}}
		\quad \text{so that} \quad
		\fdfrac{1}{a} + \fdfrac{1}{q} + \fdfrac{1}{r'} = 1
		.
	\]
	By Hölder's inequality, we obtain
	\begin{align*}
		\abs{\cB_1(f) ( u, w )}
		&\leq
		\norm{\RemainderOp[s+\varepsilon][f] f}_{\Lebesgue[q ](\Measure[f])}^{p-2} 
		\norm{\RemainderOp[s+\varepsilon][f] u}_{\Lebesgue[q ](\Measure[f])}^{\phantom{p}}
		\norm{\RemainderOp[s-\diffoffset][f] w}_{\Lebesgue[r'\!](\Measure[f])}^{\phantom{p}}
		.
	\end{align*}
	Notice that we have a Sobolev embedding $\Sobo[s+\diffoffset,r] \hookrightarrow \Sobo[s+\varepsilon,q]$ (cf.~\cref{eq:SobolevInequality}).
	Indeed, we have
	\begin{equation*}
		s + \diffoffset
		>
		s + \varepsilon 
		\qand
		s + \diffoffset - \fdfrac{\DomDim}{r}
		\geq 
		s + \fdfrac{1}{p-1} \pars*{ \diffoffset - \fdfrac{\DomDim}{r}}
		=
		s + \varepsilon - \fdfrac{\DomDim}{q},
	\end{equation*}
	where we used that $p > 2$ and $\diffoffset \geq \DomDim/r$.
	Moreover, we have $s + \varepsilon - \DomDim/q \geq s > 1$. Hence, \cref{lem:RfsIsBounded} shows that $\RemainderOp[s+\varepsilon][f] \colon \Sobo[s+\diffoffset,r] 	\subset \Sobo[s+\varepsilon,q] \to \Lebesgue[q][\Measure[f]]$ is bounded and \cref{lem:RfsIsSmooth} shows that $f \mapsto \RemainderOp[s+\varepsilon][f]$ is smooth.
	So we are left to show that 
	$\RemainderOp[s-\diffoffset][f] w \in \Lebesgue[r'][\Measure[f]][\Domain \times \Domain][\AmbDim]$
	and that it depends smoothly on $f$.
	This is a bit involved, so we redirect this task to \cref{lem:RfsIsSmoothLowOrder} below.

	\medskip

	Let us continue with $\cB_2$.
	Because of $1<s - \diffoffset < 2 - \diffoffset$, we have $0 < 1 - \diffoffset < 1$. 
	So, we can move some fractional derivative from $w$ to $f$:
	\[
		\cB_2(f) ( u, w )
		\ceq 
		\int_\Domain \! \int_\Domain
			\abs{ \RemainderOp[s +\diffoffset/p][f] f }^{p} 
			\inner{ \DeltaOp[1][f] u, \DeltaOp[1-\diffoffset][f] w }
		\dd \Measure[f]
		.
	\]
    Combined with the Hölder inequality, we obtain
	\begin{align*}
		\abs{\cB_2(f) ( u, w )}
		&\leq
		\norm{\RemainderOp[s+\diffoffset/p][f] f}_{\Lebesgue[p r](\Measure[f])}^{p} 
		\norm{\DeltaOp[1][f] u}_{\Lebesgue[\infty]}
		\norm{\DeltaOp[1-\diffoffset][f] w}_{\Lebesgue[r'](\Measure[f])}
		.
	\end{align*}
	The last factor on the right-hand side is bounded by $\norm{w}_{\Sobo[1-\diffoffset,r'](f)} \leq C(f) \norm{w}_{\Sobo[s-\diffoffset,r'](f)}$.
	By \cref{thm:GeometricMorrey}, there is a radius $\delta = r(p,\Energy(f),\DomDim,\AmbDim) > 0$ such that for every $(x,y) \in M \times M$ satisfying $\abs{f(x) - f(y)} < \delta$ 
	there is a smooth curve $\gamma$ from $x$ to $y$ such that the arc length of $f\circ \gamma$
	is greater or equal to $\abs{f(y) - f(x)}/2$. (Such a curve can be found in a graph patch common to $f(x)$ and $f(y)$.)
	So for $\abs{f(x) - f(y)} < \delta$ we may bound as follows:
	\begin{align*}
		\abs{u(y) - u(x)}
		&\leq
		\int_0^1 \abs*{ \tfrac{\dd}{\dd t} (u \circ \gamma)(t) } \dd t
		\leq 
		\int_0^1 \abs*{ \dd u(\gamma(t)) \dd f(\gamma(t))\pinv \, \dd f(\gamma(t)) \, \gamma'(t) } \dd t
		\\
		&\leq
		\norm{ \cD_f u}_{\Lebesgue[\infty]}
		\int_0^1  \abs*{\tfrac{\dd}{\dd t} (f\circ \gamma)(t)} \dd t
		\leq 
		\fdfrac{1}{2} \norm{ \cD_f u}_{\Lebesgue[\infty]} \abs{f(y) - f(x)}.
	\end{align*}
	And for $\abs{f(x) - f(y)} > \delta$ we have 
	\[
		\fdfrac{\abs{u(y) - u(x)}}{\abs{f(x) - f(y)}} 
		\leq 
		\fdfrac{1}{\delta} \pars[\big]{\abs{u(x)} + \abs{u(y)}}.
	\]
	Combined with the Morrey embedding, we obtain 
	\[
		\norm{\DeltaOp[1][f] u}_{\Lebesgue[\infty]}
		\leq 
		\fdfrac{1}{2} \norm{ \cD_f u}_{\Lebesgue[\infty]} + 
		\fdfrac{2}{\delta} \norm{u}_{\Lebesgue[\infty]}
		\leq
		C(f) \norm{u}_{\Sobo[s+\diffoffset,r](f)}
		.
	\]
	By \cref{lem:RfsIsBounded}, we can bound 
	\[
		\norm{\RemainderOp[s+\diffoffset/p][f] f}_{\Lebesgue[p r](\Measure[f])} 
		\leq 
		C(f) \norm{f}_{\Sobo[s+\diffoffset/p,p r](f)}.
	\]
	To bound the latter in terms of $\norm{f}_{\Sobo[s+\diffoffset,r](f)}$, we use the
	Sobolev embedding $\Sobo[s+\diffoffset,r] \subset \Sobo[s+\diffoffset/p,p r]$.
	Indeed, the conditions for this embedding are fulfilled:
	\begin{align*}
		s+\diffoffset > s + \fdfrac{\diffoffset}{p}
		\qand 
		s+\diffoffset - \fdfrac{\DomDim}{r} \geq s + \fdfrac{\diffoffset}{p} - \fdfrac{\DomDim}{p r}
		.
	\end{align*}
	And as above, the first one follows from $p > 1$ and the second is equivalent to $\diffoffset \geq \frac{\DomDim}{p}$.
	This completes the analysis of $\cB_2$:
	\begin{align*}
		\abs{\cB_2(f) ( u, v )}
		&\leq
		C(f)
		\norm{f}_{\Sobo[s+\diffoffset,r](f)}^{p} 
		\norm{u}_{\Sobo[s+\diffoffset,r](f)}
		\norm{w}_{\Sobo[s-\diffoffset,r'](f)}
		.
	\end{align*}
	For the third term $\cB_3$ we observe that
	\begin{align*}
		\abs{\cB_3(f) ( u, v )}
		&\leq
		2 \norm{ \DOp[f] u}_{\Lebesgue[\infty]}
		\int_\Domain \! \int_\Domain
			\abs{ \RemainderOp[s+\diffoffset/p][f] f(x,y) }^{p} 
			\,
			\pars[\big]{ \norm{\DOp[f] v(x)} \abs{f(y)-f(x)}^\diffoffset }
		\dd \Measure[f](x,y)
		\\
		&\leq
		2 \norm{ \DOp[f] u}_{\Lebesgue[\infty]}
		\norm{\RemainderOp[s+\diffoffset/p][f] f}_{\Lebesgue[p r](\Measure[f])}^p
		\pars[\Big]{
			\int_\Domain 
				\norm{\DOp[f] v(x)}^{r'}
				\,
				(\varPsi(x))^{r'}
			\dvol_f(x)
		}^{1/r'}
		,
	\end{align*}
	where
	$
		\varPsi(x)
		\ceq
		\int_\Domain
		\abs{f(y)-f(x)}^{\diffoffset-\DomDim}
		\dvol_f(y)
	$. 
	Since $\diffoffset-\DomDim > - \DomDim$ and because of \cref{thm:GeometricMorrey}, 
    this is a bounded function. (This can be seen by using polar coordinates in a graph patch around $x$).
	By reusing $\norm{\RemainderOp[s+\diffoffset/p][f] f}_{\Lebesgue[p r](\Measure[f])} \leq C(f) \norm{f}_{\Sobo[s,r](f)}^p$, we finally gain a bound of the desired form:
	\begin{align*}
		\abs{\cB_3(f) ( u, v )}
		\leq 
		2 \, C(f)
		\norm{f}_{\Sobo[s+\diffoffset,r](f)}^p
		\norm{u}_{\Sobo[s+\diffoffset,r](f)}
		\norm{v}_{\Sobo[1,r'](f)}
		.
	\end{align*}
\end{proof}

\begin{lemma}\label{lem:RfsIsSmoothLowOrder}
	Let $f_0$ be an arbitrary smooth embedding of $\Domain$ into $\AmbSpace$.
	Under the conditions of \cref{prop:DerivativeExtended} the following map is well-defined and smooth:
	\begin{equation*}
		\RemainderOp[s - \diffoffset]
		\colon 
		\Emb[s+\diffoffset,r][][\Domain][\AmbSpace]
		\to 
		\ContOps \pars[\big]{
			\Sobo[s-\diffoffset,r'\!\!][][\Domain][\AmbSpace]
			;
			\Lebesgue[r'\!][\Measure[f_0]][\Domain \times \Domain][\AmbSpace]
		}
		,
		\quad
		f \mapsto \pars[\big]{w \mapsto \RemainderOp[s - \diffoffset][f] w}.
	\end{equation*}
\end{lemma}
\begin{proof}
	The proof is similar to the one of \cref{lem:RfsIsBounded}:
	We consider a local graph patch $f(x) = (x,h(x))$,
	where
	$h \in \Sobo[s+\diffoffset,r][][V][\NormalSpace]$ 
	satisfies $\norm{D h}_{\Lebesgue[\infty]} \leq 1/2$ and $V \subset \DomSpace$ is a ball of finite radius.
	For $u = w$ we obtain the following pointwise inequality from \cref{eq:RsfInGraphPatch}:
	\begin{equation*}
		\abs{\RemainderOp[s-\diffoffset][f] w(x,y)}
		\leq
		C
		\pars[\big]{
			\abs{\remainderOp[s-\diffoffset] w(x,y)}
			+
			\abs{ D w (x)}
			\,
			\abs{\remainderOp[s-\diffoffset] h(x,y)}
		}
		.
	\end{equation*}
	This implies the following norm bound (note that $\Measure[f]$ and $\Measure$ control each other):
	\begin{equation*}
		\norm{\RemainderOp[s-\diffoffset][f] w}_{\Lebesgue[r'\!][][V\times V](\Measure[f])}
		\leq
		C \norm{\remainderOp[s-\diffoffset] w}_{\Lebesgue[r'\!][][V\times V](\Measure)}
			+
		C \norm[\big]{
			\abs{ D w \circ \prx}
			\,
			\abs{\remainderOp[s-\diffoffset] h}
		}_{\Lebesgue[r'\!][][V\times V](\Measure)}
		.
	\end{equation*}
	Since $s - \diffoffset > 0$, we get \cref{lem:EuclideanRsIsBounded} that
	\[ 
		\norm{\remainderOp[s-\diffoffset] w}_{\Lebesgue[r'\!][][V\times V](\Measure)} 
		\leq 
		C \norm{w}_{\Sobo[s-\diffoffset,r'][][V]}.
	\]
	Now we have to bound the $\Lebesgue[r'][\mu][V\times V]$-norm of
	$\abs{ D w \circ \pi_1 }\abs{\remainderOp[s-\diffoffset] h}$.
	In contrast to our approach in \cref{lem:RfsIsBounded}, 
	we cannot exploit here anymore that $D w$ is bounded as the regularity of $w$ is rather low:
	\[
		(s-\diffoffset-1) - \fdfrac{\DomDim}{r'} \leq s - 1 - \fdfrac{\DomDim}{r} - \fdfrac{\DomDim}{r'} = s - 1 - \DomDim < s - 2 < 0.
	\]
	Instead of the Morrey embedding, we may still use the Sobolev embedding (see \cref{eq:SobolevInequality})
	$\Sobo[s-\diffoffset -1,r'\!][][V] \hookrightarrow \Lebesgue[\varrho][][V]$,
	where $\varrho$ solves the equation 
	\[
		0 - \fdfrac{\DomDim}{\varrho} = (s-\diffoffset-1) - \fdfrac{\DomDim}{r'}.
	\]
	Notice that
	\[
		- \fdfrac{\DomDim}{\varrho} 
		=  
		(s-\diffoffset-1) - \fdfrac{\DomDim}{r'}
		>
		s - \fdfrac{\DomDim}{p} - 1 - \fdfrac{\DomDim}{r'}
		>
		0 - \fdfrac{\DomDim}{r'} \geq - \DomDim
	\]
	implies $\varrho \geq 1$, so it is a valid exponent for a Lebesgue space.
	Next we apply the Hölder inequality
	\[
		\norm{\varphi \, \psi}_{\Lebesgue[r'][][V\times V](\Measure)}
		\leq 
		\norm{\varphi}_{\Lebesgue[\varrho][][V\times V](\Measure)}
		\norm{\psi}_{\Lebesgue[\tau][][V\times V](\Measure)}
		\quad \text{for}\quad
		\fdfrac{1}{r'}  = \fdfrac{1}{\varrho} + \fdfrac{1}{\tau}
	\]
	to the functions 
	$\varphi \ceq \abs{ D w \circ \prx } \abs{\pry- \prx}^\diffoffset$
	and
	$\psi \ceq \abs{\remainderOp[s] h}$.
	Note that we indeed have $\varphi \, \psi = \abs{ D w \circ \pi_1 }\abs{\remainderOp[s-\diffoffset] h}$.
	Also, observe that
	\begin{equation*}
		\norm{
			\varphi
		}_{\Lebesgue[\varrho][][V\times V](\Measure)}^\varrho
		=
		\textstyle
		\int_V 
			\abs{ D w(x)}^\varrho 
			\,
			\varPsi(x)
		\dd x
		\qwith
		\varPsi(x)
		\ceq 
		\textstyle
		\int_V 
			\abs{y-x}^{\varrho \, \diffoffset -\DomDim}
		\dd y
		.
	\end{equation*}
	Since $\DomDim - \varrho \, \diffoffset < \DomDim$ and because $V$ is bounded, the function $\varPsi$ is bounded, and we obtain
	\begin{align*}
		\norm{
			\varphi
		}_{\Lebesgue[\varrho][][V\times V](\Measure)}^{\phantom{\varrho}}
		\leq
		\norm{D w}_{\Lebesgue[\varrho][][V](\Measure)}^{\phantom{\varrho}}
		\norm{\varPsi}_{\Lebesgue[\infty][][V]}^{1/\varrho}
		\leq
		\norm{w}_{\Sobo[s-\diffoffset,r'][]}^{\phantom{\varrho}}
		\norm{\varPsi}_{\Lebesgue[\infty][][V]}^{1/\varrho}
		.
	\end{align*}
	So, everything hinges now on a bound for $\norm{\psi}_{\Lebesgue[\tau][][V\times V](\Measure)}$.
	Because of $s-\diffoffset > 1$, we can use \cref{lem:EuclideanRsIsBounded} again:
	\[ 
		\norm{\psi}_{\Lebesgue[\tau][][V\times V](\mu)}
		=
		\norm{\remainderOp[s-\diffoffset] h}_{\Lebesgue[\tau][][V\times V](\mu)}
		\leq 
		\norm{h}_{\Sobo[s-\diffoffset,\tau][][V]}
		.
	\]
	By assumption, we have $\diffoffset \geq \DomDim/r$, hence
	\[
		s + \diffoffset - \fdfrac{\DomDim}{r} \geq s > s - \fdfrac{\DomDim}{\tau}.
	\]
	Thus, there is an embedding $\Sobo[s+\diffoffset,r] \hookrightarrow \Sobo[s,\tau]$,
	showing that 
	$\norm{\psi}_{\Lebesgue[\tau][][V\times V](\mu)}
	\leq 
	C
	\norm{h}_{\Sobo[s+\diffoffset,r][][V]}$.
	Altogether, we obtain the bound
	\begin{equation*}
		\norm{\RemainderOp[s-\diffoffset][f] w}_{\Lebesgue[r'\!][][V\times V](\mu)}
		\leq 
		C
		\pars[\Big]{
			1			
			+
			\norm{\varPsi}_{\Lebesgue[\infty][][V]}
			\norm{h}_{\Sobo[s+\diffoffset,r][][V]}
		}
		\norm{w}_{\Sobo[s-\diffoffset,r'\!][][V]}
		.
	\end{equation*}
	From here we can proceed as in the proof of \cref{lem:RfsIsBounded} to get the global bound.
	Smoothness of $f \mapsto \RemainderOp[s-\diffoffset][f]$ can be shown by using \cref{eq:DRsf} with $u = w$ recursively, employing the same technique we developed here to handle products of $\DOp[f] w \circ \prx$ with $\RemainderOp[s-\diffoffset][f] f$ or $\RemainderOp[s-\diffoffset][f] v$, $v \in \Sobo[s+\diffoffset,r]$.
\end{proof}

\subsubsection{Construction of the metric}
\label{sec:ConstructionMetric}

Recall that \cref{prop:DerivativeExtended} basically tells us that applying $D \Energy_p$ to $f$ causes the loss of $2 \, s$ derivatives.
Thus, if we want $\smash{v(f) \ceq - \Riesz_f^{-1} D \Energy_p(f)}$ to have the same regularity as $f$, then $\smash{\Riesz_f^{-1}}$ has to restore $2 \, s$ derivatives. 
So a natural choice for $\Riesz_f$ would be the fractional Laplacian $(-\Delta_f)^s$, where $\Delta_f$ is the Laplace--Beltrami operator of $f$.
Albeit the discretization of $\Delta_f$ being a sparse matrix, the fractional Laplacian $(-\Delta_f)^s$ would lead to a dense matrix (due to its nonlocal nature). 
This makes $(-\Delta_f)^s$ prohibitively expensive in practice. 
Fortunately, $\Riesz_f$ is only meant as a preconditioner, which gives us some leeway.
On $\DomSpace$ and for $1 < s < 2$  there is a constant $C(s,\DomDim) >0$  such that
the weak form of the fractional Laplacian can be written in Gagliardo form:
\begin{align}
	\begin{split}
	\int_\DomSpace \inner[\big]{ (- \Laplacian)^s u (x),w(x)} \dd x
	&=
	\int_\DomSpace \inner[\big]{ (- \Laplacian)^{s-2} Du (x),Dw(x)} \dd x
	\\
	&=
	C(s,\DomDim) \!
	\int_\DomSpace\!\int_\DomSpace \inner*{ \deltaOp[s] u, \deltaOp[s] w }
	\dd \mu,
	\end{split}
	\label{eq:FractionalLaplacianDomSpace}
\end{align}
where we abbreviate $\deltaOp[s] \ceq \deltaOp[s-1] D$ (see \cref{eq:deltaOpAndmu}).
This can be seen via the Plancherel identity for the Fourier transform; one merely needs to mimic the procedure in \cite[Proposition~3.4]{hitchhiker}.
Motivated by this, we use the following bilinear form as main building block of our metric $\Metric_f$:
\begin{equation}
	\cA_f(u,w)
	\ceq
	\int_\Domain \! \int_\Domain
		\inner[\big]{ \DeltaOp[s][f] u, \DeltaOp[s][f] w } 
	\dd \Measure[f] 
	\qwhere
	\DeltaOp[s][f] \ceq \DeltaOp[s-1][f] \DOp[f]
	.
	\label{eq:A}
\end{equation}
 To make the metric $\Metric_f$ a bit more aware of the tangent-point energy, we also include~$\cB_2(f)$:
 \begin{equation}
	\Metric_f(u,w)
	\ceq 
	\cA_f(u,w)
	+
	\cB_2(f)(u,w)
	.
	\label{eq:Metric}
\end{equation}
In our experiments it turned out that adding $\cB_2(f)$ greatly boosts performance. 
As a side note, we built $\Metric_f$ from covariant operators, thus it is covariant itself, i.e., we have 
\begin{equation}
	\Metric_{(f \circ \varphi)}( u \circ \varphi, w \circ \varphi)
	=
	\Metric_{f}( u, w)
	\quad 
	\text{for all diffeomorphisms $\varphi \colon \Domain \to \Domain$.}
	\label{eq:MetricIsCovariant}
\end{equation}
This also implies that the direction of steepest descent is covariant in the sense of \ref{item:DesirableCovariance}.

\subsubsection{Nondegeneracity of the metric}
\label{sec:NondegeneracityMetric}

Next we show that $\Metric_f$ leads to a well-defined fractional differential operator of order $2\,s$ and that it depends smoothly on $f$.

\begin{proposition}\label{prop:welldef}
Assume the conditions of \cref{prop:DerivativeExtended}.
Then for $f\in \Emb[s+\diffoffset,r][][\Domain][\AmbSpace]$ the bilinear form $\Metric_f$ induces the following bounded linear operator ``of order'' $2\,s$:
\begin{align}
	\Riesz_f 
	\colon 
	\Sobo[s+\diffoffset,r][][\Domain][\AmbSpace] 
	\to 
	\Sobo[-s+\diffoffset,r][][\Domain][\AmbSpace],
	\quad 
	u\mapsto \pars[\big]{w\mapsto \Metric_f(u,w)}.
	\label{eq:RieszOperator}
\end{align}
Moreover, the map 
\begin{equation*}
	\Riesz \colon 
	\Emb[s+\diffoffset,r][][\Domain][\AmbSpace] 
	\to
	\ContOps \pars[\big]{
		\Sobo[s+\diffoffset,r][][\Domain][\AmbSpace] 
		;
		\Sobo[\diffoffset-s,r][][\Domain][\AmbSpace]
	},
	\quad 
	f \mapsto \Riesz_f
\end{equation*}
is $\ceil{p-1}$-times continuously differentiable (and smooth if $p$ is an even integer).
\end{proposition}
\begin{proof}
	It suffices to show that there is a constant $C \geq 0$ such that
	\begin{align*}
		\abs{ \Metric_f(u,w) } \leq C \norm{u}_{\Sobo[s+\diffoffset,r](f)} \norm{w}_{\Sobo[s-\diffoffset,r'\!](f)}
	\end{align*}
	holds true for all 
			$u\in \Sobo[s+\diffoffset,r][][\Domain][\AmbSpace]$ 
			and 
			$w\in\Sobo[s-\diffoffset,r'\!][][\Domain][\AmbSpace]$.
	We have done this already for the lower order term $\cB_2(f)$ in \cref{prop:BilinearFormsExtended}, so let us focus here on $\cA_f$.
	Because of $s+\diffoffset - \DomDim/r > 1$, with $H \ceq \Hom\pars{\AmbSpace;\AmbSpace}$ we have
	\begin{align*}
		\DOp[f] f, \DOp[f]u \in \Sobo[s+\diffoffset-1,r ][][\Domain][H],
		\qand
		\DOp[f]w \in \Sobo[s-\diffoffset-1,r'\!][][\Domain][H]
		.
	\end{align*}
	Since $s-\diffoffset > 1$, $\DOp[f]u$ can absorb some of the fractional derivatives that $\DOp[f]w$ cannot digest:
	\begin{align*}
		\abs*{
			\int_{\Domain} \! \int_{\Domain}
			\inner[\big]{ \DeltaOp[s][f] u, \DeltaOp[s][f] w } 
		\dd \Measure[f] 
		}
		&=
		\abs*{
			\int_{\Domain} \! \int_{\Domain}
				\inner[\big]{ \DeltaOp[s+\diffoffset][f] u, \DeltaOp[s-\diffoffset][f] w } 
			\dd \Measure[f] 
		}
		\\
		&\leq 
		\norm{\DeltaOp[s{+}\diffoffset][f] u}_{\Lebesgue[r](\Measure[f])}
		\norm{ \DeltaOp[s{-}\diffoffset][f] w}_{\Lebesgue[r'\!](\Measure[f])}
		.
	\end{align*}
	Finally, $\Riesz$ is smooth because the following maps are smooth:
	\begin{align*}
		\DOp 
		\colon 
		\Sobo[s+\diffoffset,r][][\Domain][\AmbSpace] 
		&\to 
		\ContOps \pars[\big]{
			\Sobo[s+\diffoffset,r][][\Domain][\AmbSpace]
			;
			\Sobo[s+\diffoffset-1,r][][\Domain][\TargetSpace]
		}
		\quad \text{and}
		\\
		\DOp 
		\colon 
		\Sobo[s+\diffoffset,r][][\Domain][\AmbSpace] 
		&\to 
		\ContOps \pars[\big]{
			\Sobo[s-\diffoffset,r'\!\!][][\Domain][\AmbSpace]
			;
			\Sobo[s-\diffoffset-1,r'\!\!][][\Domain][\TargetSpace]
		}
		.
	\end{align*}
	This can be shown by using
	$\DOp[f] u = \pars{ \DOp[f_0] u } \pars{ \DOp[f_0] f }\pinv$ from \cref{eq:FortunateIdentity} and 
	by employing the smoothness of the pseudoinverse on the space of $\AmbSpace \times \AmbSpace$-matrices of rank~$\DomDim$.
\end{proof}

Finally, we show (or at least sketch) that the operator $\Riesz_f$ is elliptic, i.e., it can ``repair'' the loss of $2 \, s$ orders of differentiability.
Strictly speaking, $\Riesz_f$ is not invertible, but this can easily be facilitated by imposing suitable constraints. E.g., one may restrict oneself to the submanifold of those embeddings $f$ that have a fixed center of mass. Restricting $\Metric_f$ to the tangent space of that manifold at point $f$ leads to a positive-definite bilinear form, and thus to an invertible Riesz operator $\Riesz_f$.
Alternatively, one may add a small multiple of $\inner{\cdot,\cdot}_{\Lebesgue[2](f)}$ to $\Metric_f$.
A third option is to rely on the assumption that the intersection of the nullspace of $\Metric_f$ and the 
nullspace of the linearized forward operator is trivial and that this operator has suitable mapping properties, 
in analogy to the usual procedure in regularization with differential operators (see \cite{EHN:96}). 
We will demonstrate these properties for 
our forward operator in Lemma \ref{lem:UpdateDirectionIsLipschitzContinuous} and use this last approach in our implementation. 
Together with \cref{prop:BilinearFormsExtended}, this shows that \ref{item:DesirableLipschitz} and \ref{item:DesirableRegularity} are satisfied.

\begin{proposition}[Ellipticity]\label{prop:LpEllipticity}
	For every $f \in \Sobo[s+\diffoffset,r][][\Domain][\AmbSpace]$
	there is a $0 < C < \infty$ such that
	\begin{align}
		\norm{u}_{\Sobo[s+\diffoffset,r](f)}
		\leq 
		C \norm{\Riesz_f u}_{\Sobo[-s+\diffoffset,r](f)}
		\quad 
		\text{for all $u \in \Sobo[s+\diffoffset,r][][\Domain][\AmbSpace]$}
		.
		\label{eq:EllipticEstimate}
	\end{align}
	The operator 
	$
		\Riesz_f 
		\colon 
		\Sobo[s+\diffoffset,r][][\Domain][\AmbSpace] 
		\to 
		\Sobo[-s+\diffoffset,r][][\Domain][\AmbSpace]
	$
	is a Fredholm operator of index $0$.
	The null space of $\Riesz_f$ consists of the constant functions.
\end{proposition}
\begin{proof}
	The range of $\Riesz_f$ is closed due to the ``elliptic Schauder estimate'' \cref{eq:EllipticEstimate}. 
	Because the operator is nonlocal, only globally constant functions are in the null space. (This is in contrast to local operators like the Laplacian where the null space consists of \emph{locally} constant functions.) We leave the details to the reader.
	Because $\Metric_f$ is symmetric on a dense subset and because the range of $\Riesz_f$ is closed, $\Riesz_f$ is a Fredholm operator of index $0$.
	So the real difficulty lies in showing the Schauder estimate \cref{eq:EllipticEstimate}.
	The central idea is to employ local graphs of the form $f_0(x) = (x,0)$ and $f(x) = (x,h(x))$ with $h \in \Sobo[s,r][][\DomSpace][\NormalSpace]$ and $Dh(0) = 0$. Then we have
	\begin{align*}
		\int_\DomSpace\!\int_\DomSpace\!
			\inner[\big]{ \DeltaOp[s+\diffoffset][f_0] u, \DeltaOp[s-\diffoffset][f_0] w }
		\dd \Measure[f_0]
		&=
		\int_\DomSpace\!\int_\DomSpace\!
			\inner[\big]{ \deltaOp[s] u, \deltaOp[s] w }
		\dd \Measure
		\stackrel{\cref{eq:FractionalLaplacianDomSpace}}{=}
		C\!
		\int_\DomSpace \!\! \inner[\big]{ (- \Laplacian)^s u (x),w(x)} \dd x. 
	\end{align*}
	The second ingredient is the fact that 
	because of $Dh(0) = 0$ we may find for every $\varepsilon>0$ a~sufficiently small neighborhood $V$ around $0$ so that
	\begin{gather*}
		\abs*{
			\int_V\!\int_V\!\!
				\inner[\big]{ \DeltaOp[s+\diffoffset][f] u, \DeltaOp[s-\diffoffset][f] w }
			\dd \Measure[f_0]
			-
			\!
			\int_V\!\int_V\!\!
				\inner[\big]{ \DeltaOp[s+\diffoffset][f_0] u, \DeltaOp[s-\diffoffset][f_0] w }
			\dd \Measure[f_0]
		}
		\leq 
		\varepsilon \! \norm{u}_{\Sobo[s+\diffoffset,r][][V]} \! \norm{w}_{\Sobo[s-\diffoffset,r'\!][][V]}
		.
	\end{gather*}
	Thus, we may apply classical Schauder techniques, using the ``operator with constant coefficients''
	$\pars{\DeltaOp[s-\diffoffset][f_0]}\adj\pars{\DeltaOp[s+\diffoffset][f_0]}$ 
	as local approximation of the 
	``operator with variable coefficients'' $\pars{\DeltaOp[s-\diffoffset][f]}\adj\pars{\DeltaOp[s+\diffoffset][f]}$.
	The analysis is further complicated by the fact that we are dealing with nonlocal operators here. 
	Nonetheless, the theory can be employed because 
	the ``meat'' of these nonlocal operators always sits close to the diagonal in $\Domain \times \Domain$ and because the operators $\DeltaOp[s\pm\diffoffset]$ satisfy certain product rules.
	As carrying this out would provide little further insight,
	we resist the temptation to go into even more detail here.
\end{proof}


\section{Regularization}
\label{sec:Regularization}

From now on we focus on surfaces ($\DomDim = 2$) in $3$-dimensional Euclidean space ($\AmbDim = 3$).
We fix some $p > 2 \, \DomDim = 4$ (in practice we use $p = 6$) and abbreviate $\Energy \ceq \Energy_p$.
Suppose we are given an unknown surface $\varSigma^\dagger \subset \R^3$ and that we know a~priori that
\begin{align}
	\Energy(\varSigma^\dagger) < E \qand \varSigma^\dagger \subset \ClosedBall{0}{R}
	\label{eq:APriori}
\end{align}
for some $0 < E < \infty$ and $0 < R <\infty$.
Moreover, suppose that instead of the exact far field $F(\varSigma^\dagger) \in \Lebesgue[2][][\mathbb{S}^2][\C^{\NumWaves}]$ of $\NumWaves$ incoming plane waves, we only have access to a noisy measurement $g^\delta \in \Lebesgue[2][][\mathbb{S}^2][\C^{\NumWaves}]$ with noise level
$
\norm{ g^\delta - F(\varSigma^\dagger) } \ceq \norm{ g^\delta - F(\varSigma^\dagger) }_{\Lebesgue[2]} \leq \delta
$.
Further, we suppose that the operator $F$ is injective on the space of surfaces that satisfy \cref{eq:APriori}. 
This is not an overly bold assumption because this can be achieved, for instance, 
by using sufficiently many incident directions (see \cite[Theorem~5.2]{zbMATH06061716})
or by choosing at least one wave number $\kappa < \uppi/R$ (see \cite[Corollary~5.3]{zbMATH06061716}). 

We attempt to find an approximation $\varSigma_\alpha^\delta$ of the true solution $\varSigma^\dagger$ by picking an appropriate $\alpha > 0$ and by minimizing the following \emph{generalized Tikhonov functional}:
\begin{align}
	\label{Tik}
	\varSigma_\alpha^\delta \in \argmin( \Tikhonov_\alpha^\delta ),
	\quad
	\text{where}
	\quad
	\Tikhonov_\alpha^\delta(\varSigma)
	\ceq 
	\begin{cases}
		\tfrac{1}{2} \norm{ F(\varSigma) - g^\delta }^2 + \alpha \, \Energy(\varSigma)
		,
		& \varSigma \subset \ClosedBall{0}{R}
		,
		\\
		\infty, &\text{else.}
	\end{cases}
\end{align}
The constraint $\varSigma \subset \ClosedBall{0}{R}$ is here to employ the results from \cref{sec:PropertiesTP}.
In principle, we could impose this constraint by using $\varSigma \mapsto \Energy(\varSigma \cup S)- \Energy(S)$ as regularizer, where $S \ceq \partial \ClosedBall{0}{R}$. 
However, we did not find it necessary to include this term into our numerical experiments as we never experienced such a drift towards infinity.

Next we demonstrate that the tangent-point energy is appropriate as a stabilizing functional:
We show that minimizers of the generalized Tikhonov functional $\Tikhonov_\alpha^\delta$ exist (see \cref{thm:Existence}) and that the regularizing property \cref{thm:Regularization} is fulfilled.
Afterwards we propose an algorithm for optimizing $\Tikhonov_\alpha^\delta$.

\subsection{Existence and regularizing property}\label{sec:ExRegProp}
The following two theorems adapt classical results from regularization 
theory for operators in Banach spaces (see, e.g., \cite{Scherzer_etal:09}) to operators defined 
on shape spaces with the tangent point energy as penalty term. Here we make heavy use of the properties of this energy derived in \cref{sec:TPE}.

\begin{theorem}[Existence]\label{thm:Existence}
	Suppose that $\varSigma^\dagger \subset \AmbSpace$ is a closed $\DomDim$-dimensional $\Holder[1]$-surface 
	that satisfies the a~priori conditions \cref{eq:APriori}.
	Let $\delta > 0$ and $g^\delta \in \Lebesgue[2][][\mathbb{S}^2][\C^{\NumWaves}]$
	satisfy $\norm{ F(\varSigma^\dagger) - g^\delta }_{\Lebesgue[2]} \leq \delta$.
	\newline
	Then for every $\alpha > 0$ the generalized Tikhonov functional $\Tikhonov_\alpha^\delta$ from \cref{Tik} attains its infimum.
\end{theorem}
\begin{proof}

We employ the direct method of calculus of variations.
Note that $ \inf(\Tikhonov_\alpha^\delta) < \infty$ because there certainly exists a feasible surface with finite tangent-point energy (e.g., some $\varSigma^\dagger$).
Let $(\varSigma_k)_{k\in\N}$ be a minimizing sequence of $\Tikhonov_\alpha^\delta$ such that 
$\Tikhonov_\alpha^\delta(\varSigma_k) < \infty$ and $\Tikhonov_\alpha^\delta(\varSigma_k) \searrow \inf(\Tikhonov_\alpha^\delta)$ as $k \to \infty$.
Thus, we have
\begin{align*}
\alpha \, \Energy(\varSigma_k)
	\leq 
	\tfrac{1}{2} \norm{ g^{\delta} - F(\varSigma_k) }^2
	+
	\alpha \, \Energy(\varSigma_k)
	=
	\Tikhonov_\alpha^\delta(\varSigma_k)
	\leq 
	\Tikhonov_\alpha^\delta(\varSigma_1)
	\qec c.
\end{align*}
By the compactness result \cref{thm:Compactness}, there must be a subsequence $(\varSigma_{k_\ell})_{\ell \in \N}$ that converges to some submanifold $\varSigma$ in Hausdorff distance $\HsdD$.
By \cref{thm:SemiContinuity}, the tangent-point energy $\Energy$ is lower semicontinuous with respect to $\HsdD$, hence we have
$
	\Energy(\varSigma) \leq \liminf_{\ell \to \infty} \Energy(\varSigma_{k_\ell}).
$
\cref{thm:Compactness} also shows that $\varSigma_k \converges[k \to \infty] \varSigma$ in $\Holder[1]$, i.e.,
there are $\Holder[1]$-diffeomorphisms $\diffeo_\ell \colon \AmbSpace \to \AmbSpace$ such that $\diffeo_\ell(\varSigma) = \varSigma_{k_\ell}$ and  $\diffeo_\ell \converges[\ell \to \infty] \id_{\AmbSpace}$ in $\Holder[1]$.
The forward operator $F$ is continuous with respect to this convergence (see, e.g., \cite[Lemma~4.2]{djellouli}),
thus we get 
\begin{align*}
	\inf(\Tikhonov_\alpha^\delta)
	\leq
	\Tikhonov_\alpha^\delta(\varSigma)
	=
	\tfrac{1}{2} \norm{F(\varSigma) - g^\delta}^2
	+
	\alpha\,\Energy(\varSigma)
	&\leq 
	\lim_{\ell \to \infty}
	\tfrac{1}{2} \norm{F(\varSigma_{k_\ell}) - g^\delta}^2
	+
	\alpha\,\liminf_{\ell \to \infty}
	\Energy(\varSigma_{k_\ell})
	\\
	&=
	\liminf_{\ell \to \infty}
	\Tikhonov_\alpha^\delta(\varSigma_{k_\ell})
	=
	\inf(\Tikhonov_\alpha^\delta)
	.
\end{align*}
This shows that $\varSigma$ is a minimizer of $\Tikhonov_\alpha^\delta$ and thus completes the proof.
\end{proof}

\begin{theorem}[Regularizing property]\label{thm:Regularization}
Suppose that $F$ is injective and 
that $\varSigma^\dagger \subset \AmbSpace$ is an $\DomDim$-dimensional $\Holder[1]$-surface 
that satisfies the a~priori conditions \cref{eq:APriori}.
Let $(g^{\delta_k})_{k\in\N}$ be a sequence with $\norm{ g^{\delta_k} - F(\varSigma^\dagger) } \leq \delta_k \converges[k \to \infty] 0$. 
Choose regularization parameters $\alpha_k>0$ such that $\alpha_k \converges[k \to \infty] 0$ and $\delta_k^2 / \alpha_k \converges[k \to \infty] 0$.
Let $\varSigma_k$ be a minimizer of $\Tikhonov_{\alpha_k}^{\delta_k}$.
Then 
\begin{equation}
	\varSigma_k \converges[k \to \infty] \varSigma^\dagger 
	\;\text{in $\Holder[1]$}\;
	\quad\text{and}\quad
	F(\varSigma_k) \converges[k \to \infty] F(\varSigma^\dagger)
	.
	\label{eq:Tikhonov}
\end{equation}
\end{theorem}
\begin{proof} 
	Observe that the number
	$E \ceq \sup_{k\in\N} \delta_k^2/(2\,\alpha_k) +\Energy(\varSigma^\dagger)$ 
	is finite because $\delta_k^2/\alpha_k $ converges to~$0$. 
	Since $\varSigma_k$ is a minimizer of $\Tikhonov_{\alpha_k}^{\delta_k}$, we have
	\begin{align*}
		\alpha_k \, \Energy(\varSigma_k)
		\leq 
		\Tikhonov_{\alpha_k}^{\delta_k}( \varSigma_k )
		\leq 
		\Tikhonov_{\alpha_k}^{\delta_k}( \varSigma^\dagger )
		\leq 
		\tfrac{1}{2} \, \delta_k^2
		+
		\alpha_k \, \Energy(\varSigma^\dagger)
		\leq 
		\alpha_k \, E,
		\quad
		\text{hence}
		\quad 
		\Energy(\varSigma_k) \leq E.
	\end{align*}
	Using the above inequality once more, we obtain
	\begin{align*}
		\tfrac{1}{2}
		\norm{ F(\varSigma_k)  - g^{\delta_k}}^2
		\leq 
		\Tikhonov_{\alpha_k}^{\delta_k}( \varSigma_k )
		\leq
		\alpha_k \, E,
		\quad
		\text{hence}
		\quad 
		\norm{  F(\varSigma_k) - g^{\delta_k}} \converges[k \to \infty] 0
		.
	\end{align*}
	Thus, $F(\varSigma_k)$ converges to the correct far field:
	\begin{align*}
		\norm{ F(\varSigma^\dagger) - F(\varSigma_k) }
		\leq
		\norm{ F(\varSigma^\dagger) - g^{\delta_k} }
		+
		\norm{ g^{\delta_k}- F(\varSigma_k) }
		\leq 
		\delta_k + \norm{ F(\varSigma_k) - g^{\delta_k} }
		\converges[k \to \infty] 0
		.
	\end{align*}
	Next we show that $\varSigma_k$ converges to $\varSigma^\dagger$ in the Hausdorff distance $\HsdD$.
	Since $\HsdD$ is a metric,
	this is true if and only if every subsequence $(\varSigma_{k_\ell})_{\ell \in \N}$ has a subsequence 
	that converges in $\HsdD$ to $\varSigma^\dagger$. 
	So, let $\varGamma_\ell \ceq \varSigma_{k_\ell}$, $\ell \in \N$ be an arbitrary subsequence. 
	Using the compactness result \cref{thm:Compactness}, we find a subsequence $(\varGamma_{\ell_i})_{i \in \N}$ that converges in $\HsdD$ and in $\Holder[1]$ to some surface~$\varGamma$.  
	Hence, $F(\varGamma_{\ell_i})$ converges to $F(\varGamma)$ (again, by virtue of \cite[Lemma 4.2]{djellouli}). 
	But we already know that $F(\varSigma_k)$ and $F(\varGamma_{\ell_i})$ must converge to $F(\varSigma^\dagger)$; thus $F(\varGamma) = F(\varSigma^\dagger)$. 
	By assumption, $F$ is injective, hence $\varGamma = \varSigma^\dagger$ and $(\varGamma_{\ell_i})_{i \in \N}$ converges to $\varSigma^\dagger$ in $\HsdD$.
	Because we did this with every subsequence of $(\varSigma_{k})_{k \in \N}$, the latter must converge to $\varSigma^\dagger$ in $\HsdD$.

	Finally, we use \cref{thm:Rigidity} to show that $\varSigma_k$ converges to $\varSigma^\dagger$ also in $\Holder[1]$. 
	Recall that we have shown that $\Energy(\varSigma_k) \leq E$ for all $k \in \N$.
	Let $\delta = \delta(R,E,2,3)$ and $C = C(R,E,2,3)$ be the constants from \cref{thm:Rigidity}.
	Because $\varSigma_{k} \converges[k \to \infty] \varSigma^\dagger$ in $\varSigma^\dagger$, there is a $K \in \N$ such that $\HsdD(\varSigma_{k},\varSigma) \leq \delta$ for all $k \geq K$.  
	Hence, \cref{thm:Rigidity} implies that there are $\Holder[1]$-diffeomorphisms $\diffeo_k \colon \AmbSpace \to \AmbSpace$ such that $\diffeo_k(\varSigma) = \varSigma_k$,
	$\norm{\diffeo_k - \id_{\AmbSpace}}_{\Lebesgue[\infty]} \leq C \, \HsdD(\varSigma_k,\varSigma) \converges[k \to \infty] 0$
	and
	$\norm{D(\diffeo_k - \id_{\AmbSpace})}_{\Lebesgue[\infty]}
		\leq 
		C \, \HsdD(\varSigma_k,\varSigma)^{\alpha/2} \converges[k \to \infty] 0
	$.
	This concludes the proof.
\end{proof}

\begin{remark}
An inspection of the proofs shows that Theorems \ref{thm:Existence} and \ref{thm:Regularization} remain valid for 
any operator $F$ which  maps surfaces $\varSigma^{\dagger}$ satisfying condition \eqref{eq:APriori}
to some Banach space $Y$ and which is injective and continuous with respect to the 
$C^1$-topology. 
\end{remark}

\subsection{The optimization algorithm}\label{sec:Optimization}

We would like to perform the numerical optimization of the generalized Tikhonov functional in the open subset $\Emb[s,p][][\Domain][\R^3]$ of the Banach space $\Sobo[s,p][][\Domain][\R^3]$.
For $f \in \Emb[s,p][][\Domain][\R^3]$ we set $\Tikhonov_\alpha^\delta(f) \ceq \Tikhonov_\alpha^\delta(f(\Domain))$. 
As $p > 2 \,\DomDim = 4$, we have seen in \cref{sec:Scat} and \cref{prop:FrechetDerivative} that
$F$ and $\Energy$ and thus $\Tikhonov_\alpha^\delta$ are at least three times Fréchet differentiable, and thus ``smooth'' from the perspective of numerical optimization.
The canonical idea would be to apply Newton's method. 
This would amount to solving the following linear equation in weak form:
\begin{align*}
	D^2 \!\Tikhonov_\alpha^\delta(f) (v,w) = - D \Tikhonov_\alpha^\delta(f) \, w
	\quad 
	\text{for all $w \in \Sobo[s,p][][\Domain][\R^3]$.}
	\label{eq:Newton}
\end{align*}
Then one would use $f + v$ as updated state.
However, this would require the second derivative 
\begin{equation*}
	D^2 \!\Tikhonov_\alpha^\delta(f) (v,w)
	=
	\inner{F(f) - g^\delta\!, D^2 F(f)(v,w)}_{\Lebesgue[2]}
	+
	\inner{DF(f) \, v, DF(f) \, w}_{\Lebesgue[2]}
	+
	\alpha \, D^2 \Energy(f)(v,w)
	,
\end{equation*}
at least in the form of the action of $v \mapsto D^2 \Tikhonov_\alpha^\delta(f) (v,\cdot)$, 
so that the update direction $v$ may be computed by an iterative linear solver.
Unfortunately, the second derivative $D^2 F(f)$ of the forward operator would be particularly expensive, and it is also not expected to have a significant effect on $D^2 \Tikhonov_\alpha^\delta(f)$ when $F(f) - g^\delta$ is small (as it is finally expected for low noise level).
Hence, it has become a common practice to drop the term $\inner{F(f) - g^\delta, D^2 F(f)(v,w)}_{\Lebesgue[2]}$ entirely and merely use the Gauss--Newton approximation $\inner{DF(f) \, v, DF(f) \, w}_{\Lebesgue[2]}$; and so do we.
Also, $D^2 \Energy(f)$ is non-trivial to implement, 
so we instead use the metric $\Metric_f = \cA_f + \cB_2(f)$ from \cref{eq:Metric}.
In summary, we compute the search direction $v$ by solving the following linear equation in weak formulation:
\begin{align}
	\label{eq:update_step_weak}
	\inner{DF(f) \, v, DF(f) \, w}_{\Lebesgue[2]}
	+
	\alpha \, \Metric_f(v,w)
	=
	-
	\inner{F(f) - g^\delta\!, D F(f) \, w }_{\Lebesgue[2]}
	-
	\alpha
	\,
	D \Energy(f) \, w
\end{align}
for all $w \in \Holder[\infty][][\Domain][\AmbSpace]$.
We may also interpret this as minimizing the local, quadratic proxy of $\Tikhonov_\alpha^\delta(f+v)$
from \cref{eq:LocalQuadraticProxy} with
\begin{align}
	E_f(v)
	\ceq 
	\Energy(f)
	+
	D\Energy(f) \, v
	+
	\tfrac{1}{2} \, \Metric_f(v,v)
	.
	\label{eq:LocalQuadraticProxyMetric}
\end{align}
Note that the left-hand side of \cref{eq:update_step_weak} is a positive semidefinite bilinear form in $v$ and $w$.
So the solution $v$ is guaranteed to be a descending direction for $\Tikhonov_\alpha^\delta$, i.e., $D \Tikhonov_\alpha^\delta(f) \, v \leq 0$.
This is something that is not guaranteed by Newton's method for a nonconvex optimization problem like ours.

We have yet to show that equation \cref{eq:LocalQuadraticProxyMetric} admits a unique solution $v = v(f)$.
Moreover, it makes sense to use $v(f)$ as update only if $f + v(f)$ has at least the regularity of~$f$; otherwise regularity would deteriorate, and we could not take further steps afterwards.
We claim that we can guarantee both if we assume that $f$ has lightly higher regularity.
More precisely, we claim the following:

\begin{lemma}\label{lem:UpdateDirectionIsLipschitzContinuous}
	Suppose $p > 4$ so that $s=2-2/p > 1$. (This is the condition to make the tangent-point energy repulsive.)
	Moreover, suppose that $\diffoffset > 0$ and  $r \in \intervaloo{p,\infty}$ satisfy
	$2/r \leq \diffoffset < 2/p $ (this is the condition from \cref{prop:DerivativeExtended}).
	Then for $\alpha > 0$ and each $f \in \Emb[s+\diffoffset,r][][\Domain][\R^3]$ there is a unique weak solution $v(f) \in \Sobo[s+\diffoffset,r][][\Domain][\R^3]$ of \cref{eq:update_step_weak}.
\end{lemma}

\begin{proof}
    Recall from \cref{prop:welldef,prop:LpEllipticity} that 
	$
		\smash{\Metric_f \colon \Sobo[s+\diffoffset,r][][\Domain][\R^3] \times \Sobo[s-\diffoffset,r'\!][][\Domain][\R^3] \to \R}
	$
	is a bounded, bilinear form and that the induced linear operator is Fredholm of index $0$. 
	We first show that the first term on the left-hand side of \cref{eq:update_step_weak} is also a bounded, bilinear form on $\Sobo[s+\diffoffset,r][][\Domain][\R^3] \times \Sobo[s-\diffoffset,r'\!][][\Domain][\R^3]$ and that the induced operator is compact. 

	As $s+\diffoffset > 1$ and $s+\diffoffset-\frac{2}{r} \geq s>\frac{3}{2}> 1$, the Sobolev embedding theorem implies that $\Sobo[s+\diffoffset,r][][\Domain][\R^3]$ is compactly embedded into $\Holder[1][][\Domain][\R^3]$. Since $F$ is Fréchet differentiable from $\Holder[1][][\Domain][\R^3]$ to $\Lebesgue[2][][\mathbb{S}^2][\C^{\NumWaves}]$ (see \cref{prop:FrechetDerivative}), the operator $DF(f)$ is bounded from $\Sobo[s+\diffoffset,r][][\Domain][\R^3]$ to $\Lebesgue[2][][\mathbb{S}^2][\C^{\NumWaves}]$. 

	Next we study $DF(f)$ as an operator from $\Sobo[s-\diffoffset,r'\!][][\Domain][\R^3]$ to $\Lebesgue[2][][\mathbb{S}^2][\C^{\NumWaves}]$. 
	Note that $\Sobo[s-\diffoffset,r'\!]$ is continuously embedded into $\Bessel[1/2][][\Domain][\R^3]$ because 
	\[
		s-\diffoffset 
		> 
		s - \sdfrac{2}{p} 
		> 
		1 
		> \sdfrac{1}{2}
		\qand
		(s-\diffoffset) - \sdfrac{2}{r'}
		>
		\pars[\Big]{\sdfrac{3}{2}-\sdfrac{2}{r}} - \pars[\Big]{\sdfrac{2}{1} - \sdfrac{2}{r}}
		=
		\sdfrac{1}{2}
		=
		\sdfrac{1}{2} - \sdfrac{2}{2}
		.
	\]
	Since the norm on $\Sobo[s-\diffoffset,r'\!]$ is rather weak, Fréchet differentiability results for $F$ with respect to this norm are not available. However, we can establish the required mapping properties of $DF(f)$ by using the characterization of $DF(f)$ from \cref{sec:FrechetDerivativeForwardOperator}:
	Let $\nu \colon \varSigma \to \R^3$ be the outward normal of $\varOmega$ and denote by $u = (u_1,\dotsc,u_\NumWaves) \colon \R^3 \setminus \varOmega \to \C^\NumWaves$ the radiating solutions on the exterior domain.
	Now we can write  
	\[
        DF(f) \, v = S_f \, A_f v,
        \quad 
        \text{where}
        \quad 
        A_f v \ceq - \pars[\Big]{\sdfrac{\partial u}{\partial \nu} \circ f} \inner{\nu \circ f,v}
    \]
	and where $S_f\colon  \Bessel[1/2][][\Domain][\C^{\NumWaves}] \to \Lebesgue[2][][\mathbb{S}^2][\C^{\NumWaves}]$ is the far field operator. It maps Dirichlet boundary data on the surface $\varSigma = f(\Domain)$ to far field patterns of the corresponding radiating solution of the Helmholtz equation in the exterior domain.
	It remains to show that $A_f$ is bounded from $\Sobo[s-\diffoffset,r'\!][][\Domain][\R^3]$ 
	(or  $\Bessel[1/2][][\Domain][\R^3]$) to $\Bessel[1/2][][\Domain][\C^{\NumWaves}]$.
	To this end, note that the surface $\varSigma=f(\Domain)$ is of class $\Holder[1,\sigma]$ with Hölder exponent 
	\[
		\sigma 
		= 
		1 - \sdfrac{2}{p} + \diffoffset - \sdfrac{2}{r} 
		\geq 
		1 - \sdfrac{2}{p} 
		> 
		\sdfrac{1}{2}
		.
	\]
	Hence, we have $\nu_f \in \Holder[0,\sigma][][\Domain][\R^3]$.
	It follows from elliptic regularity results for such domains (see, e.g., \cite[Thm. 8.33]{GT:77}) that $\partial u / \partial \nu\in \Holder[0,\sigma ][][\varSigma][\C^{\NumWaves}]$ and thus also $\partial u / \partial \nu \circ f \in \Holder[0,\sigma][][\Domain][\C^{\NumWaves}]$. 
	Hence, we see that 
	\[
		 - \pars[\Big]{\sdfrac{\partial u}{\partial \nu} \circ f} \pars{\nu \circ f}
		 \in \Holder[0,\sigma][][\Domain][\Hom(\R^3;\C^\NumWaves)].
	\]
	By \cref{lem:ProductRuleWspWtinfty} in \cref{sec:ProductRule} for $\tau = 1/2$, $p = 2$,
	the multiplication operator $A_f$ is bounded from 
	$\Bessel[1/2][][\Domain][\R^3]$ to $\Bessel[1/2][][\Domain][\C^N]$.
	The strong formulation of \cref{eq:update_step_weak} becomes
	\begin{align}\label{eq:update_step2}
		L_f v 
		\ceq 
		\pars[\big]{A_f'\, S_f' \, J_{L^2}DF(f)+\alpha \, \Riesz_{f} } \, v 
		= 
		- \pars[\big]{ 
			A_f' \, S_f' \, J_{L^2} \pars{F(f) - g^\delta} + \alpha \,  \overline{D\Energy}(f)
		}
		,
	\end{align}
	where $J_{L^2} \colon \Lebesgue[2][][\mathbb{S}^2][\NumWaves]\to \Lebesgue[2][][\mathbb{S}^2][\NumWaves]'$ is the Riesz isomorphism. 
	From the above it follows that
	\[
		L_f \colon 
		\Sobo[s+\diffoffset,r][][\Domain][\R^3] 
		\to 
		\pars{\Sobo[s-\diffoffset,r'\!][][\Domain][\R^3]}' 
		= 
		\Sobo[\diffoffset+s,r\!][][\Domain][\R^3]  
	\] 
	is the sum of a compact operator and a Fredholm operator of index~$0$. Thus, $L_f$ is a Fredholm of index~$0$, too. 
	The above arguments also show that the first term on the right-hand side belongs to the range of $L_f$. 
	And for the second term this has been established already in Proposition \cref{prop:DerivativeExtended}.
	
	It remains to show that $L_f$ is injective. Let $v \in \Sobo[s+\diffoffset,r][][\Domain][\R^3]
	$ be an element of its nullspace.
	Because of the Sobolev embeddings $\Sobo[s+\diffoffset,r][][\Domain][\R^3] \embeds \Sobo[s,p][][\Domain][\R^3] \embeds \Sobo[s-\diffoffset,r'][][\Domain][\R^3]$,
	we may test with $w = v$ to obtain
	\begin{equation}
		0 
		=
		\inner{L_f \, v,v}
		= \inner[\big]{
			\pars[\big]{A_f' \, S_f' \, J_{\Lebesgue[2]} \, DF(f)+ \alpha \, \Riesz_{f} } \, v, v			
			}
		=
		\|DF(f) \, v\|_{\Lebesgue[2]}^2 + \alpha \, \cG_f(v,v)
		.
		\label{eq:PositiveDefiniteness}
	\end{equation}
	This implies $\cG_f(v,v) = 0$ and $DF(f) \, v = 0$.
	The former implies that $\alpha \, \Riesz_{f} \, v = 0$.
	By \cref{prop:LpEllipticity}, the nullspace of $\alpha \, \Riesz_{f}$ consists only of constant functions $v$ of the form $v(x) \ceq \xi$, where $\xi \in \R^3$ is a translation vector. 
	As $F_k(f+\xi)=\ee^{-\ii \, \kappa \inner{\xi ,\cdot}} \, F_k(f)$ for all $k=1,\dots,\NumWaves$ (see, e.g., \cite{KR:97}), we have $DF_k(f) \, \xi = -(\ii \, \kappa\inner{\xi,\cdot}) F_k(f)$. 
	As the zeros of the function $z\mapsto -\ii \, \kappa \inner{\xi, z}$ on the unit sphere $\mathbb{S}^2$ form a nullset unless $\xi=0$ and since the far-field patterns $F_k(f)$ cannot vanish due to Rellich's lemma 
	(see \cite[Lemma 2.12]{coltonkress}), we get $\xi=0$.
	Thus, $L_f$ is injective, and hence bijective with bounded inverse.
\end{proof}

\begin{remark}
	It follows from the implicit function theorem for Banach spaces that under the assumptions of \cref{lem:UpdateDirectionIsLipschitzContinuous}, the mapping
	\[
		v \colon \Emb[s+\diffoffset,r][][\Domain][\R^3] \to \Sobo[s+\diffoffset,r][][\Domain][\R^3],
		\quad 
		f \mapsto v(f)
	\]
	is differentiable (or at least locally Lipschitz-continuous) if we can show that all operators in \cref{eq:update_step2} depend differentiably  (or locally Lipschitz-continuously) on $f$. 
	The only problematic term is $f\mapsto A_f$, more precisely differentiability of $f\mapsto \frac{\partial u_k}{\partial \nu} \circ f$ with 
	respect to the norm $\Holder[0,\sigma][][M][\C]$. 
	While this seems plausible, we are not aware of such a result in the literature. 
	A possible approach to prove this would be to use characterization 
	\cref{eq:directmethod} via boundary integral equations and then generalize differentiability results 
	for boundary integral operators with respect to the boundary in \cite{zbMATH00562794} from $C^2$ to 
	$C^{1,\sigma}$ surfaces. 
	Differentiability of $f\mapsto S_f$ in the given norms can be derived from 
	the arguments in \cite[sec.~1.1]{hohage_diss}.
\end{remark}

\begin{remark}\label{rem:LineSearch1}
	In numerical applications we also have to perform a line search to make the method stable.
	Local Lipschitz continuity of $f \mapsto v(f)$ would mean that a simple Armijo backtracking line search suffices to guarantee a numerically stable descent algorithm. 
	Our experiments showed that we do not get any excessively small time steps. 
	This supports our hypothesis that $f \mapsto v(f)$ is locally Lipschitz continuous.
\end{remark}

\begin{remark}\label{rem:LineSearch2}	
	Moreover, we have to guarantee that the update is an embedding of the same isotopy class as $f$.
	In practice, we use standard techniques from continuous collision detection to find the maximal step size $t_{\max} >0$ such that the linear homotopy $f_t \ceq f + t \, v$ is an isotopy for $t \in \intervalco{0,t_{\max}}$.
	Then we apply Armijo backtracking from there to find a step size $t = t(f,v) \in \intervaloo{0,t_{\max}}$ such that
	$
		\Tikhonov_\alpha^\delta(f + t \, v) 
		\leq  
		\Tikhonov_\alpha^\delta(f) + \sigma_{\mathrm{Armijo}} \, t \, D\Tikhonov_\alpha^\delta(f) \, v
	$,
	where
	$0 < \sigma_\mathrm{Armijo} < 1$ is a fixed constant.
	Finally, we use $f + t \, v$ as the update.
\end{remark}

\begin{remark}
Notice that all the ingredients $\Energy$, $F$, and $\Metric$ for our update direction $v(f)$ are covariant.
Thus, $v(f)$ is covariant, too, i.e., we have $v(f \circ \varphi) = v(f) \circ \varphi$ for all reparameterizations $\varphi \colon \Domain \to \Domain$. Moreover, also the step size $t(f)$ obtained by line search is invariant under reparameterization.
\end{remark}

\subsection{Iteratively regularized Gauss--Newton method for surfaces}\label{sec:IRGNM}

Minimizing this Tikhonov functional by a Gauss--Newton method for a fixed regularization parameter $\alpha$ and successively shrinking $\alpha$ after reaching the minimum takes many iterations.
Hence, we propose an adapted version of the iteratively regularized Gauss Newton method (see \cref{sec:RegularizationMethods}), which was first introduced for the case of Hilbert spaces by Bakushinskij in \cite{zbMATH00205051}.
The method is based on the ideas of Tikhonov regularization.
The algorithm is as follows:
First we choose a starting regularization parameter $\alpha_0>0$ and a scaling factor $\rho\in(0,1)$.
Then we determine $f_{k+1}$, $k \in \N \cup \braces{0}$ iteratively:
In the $k$-th iteration, we first
solve \cref{eq:update_step2} with $f = f_k$ and $\alpha = \alpha_k$; we denote the solution by $v_k$.
Then we determine a step size $t_k = t(f_k,v_k)$ and set 
\begin{align}
	f_{k+1} = f_k + t_k \, v_k
	\qand
	\alpha_{k+1} = \varrho \, \alpha_k.
	\label{eq:update}
\end{align}
We stop this iteration by using the \emph{discrepancy principle}:
Since the input data $g^\delta$ is not exact anyway, there is little gain in fitting $F(f_k)$ overly well to $g^\delta$.
In practice, we assume that we have a rough idea of the noise level $\delta$ and choose some small $\tau>1$. Then we terminate the iteration as soon as we have $\norm{F(f_k)-g^\delta}_{\Lebesgue[2]}<\tau\,\delta$.
This stopping rule is well-known to make the iteratively regularized Gauss--Newton method an iterative regularization method \cite{zbMATH01033895,hohage_diss} under 
additional assumptions.


\section{Implementation}\label{sec:Implementation}

This section is dedicated to the implementation of the algorithm proposed in \cref{sec:Optimization} and \cref{sec:IRGNM}.
In particular, we split this task into the discretization of the forward operator and the implementation of the tangent-point energy.
We close the section by a brief explanation on how the calculation of the update step is implemented.

\subsection{Discretization of the forward operator}
\label{sec:discretisation}

To handle the boundary operators from the previous section numerically, one needs to approximate them. 
We chose to do so by a \emph{boundary element method} (BEM).

\subsubsection{Boundary element methods}

The advantages of BEMs over finite element methods are 
that they do not require to mesh the unbounded domain $\R^3 \setminus \overline{\varOmega}$; instead, one only needs to mesh the boundary $\varSigma = \partial \varOmega$. 
This is particularly useful in our regularization procedure, because the 3-dimensional domain $\varOmega$ changes from iteration to iteration.
Furthermore, the Sommerfeld radiation condition is implicitly satisfied by this ansatz.
The downside of this method is that we have to solve boundary integral equations whose discretizations lead to dense matrices. 
However, this issue can be amended by employing suitable fast multipole methods or by approximating the integral operators data-sparsely by hierarchical matrices. 
We found it difficult to find a freely available and portable implementation that can be tailored to our needs.
Since our main goal is to demonstrate that the tangent-point energy is a viable regularizer for inverse scattering problems, we implemented the action of the boundary integral operators in a straightforward matrix-free manner in OpenCL and executed it on a GPU (see \cref{sec:MatrixProduct}).

There are several boundary element discretizations on triangulated meshes.
For instance there are the method of collocation \cite{bem}, Nystr\"om methods, and the Galerkin formulation \cite{bem,bempp}.
We decided to use a Galerkin formulation.
Our eventual input data will be noisy. 
So, it suffices for our needs to use a method that provides coarse approximations but that is computationally efficient.
We aim for rather fine meshes and discrete function spaces of low order, because we would like to resolve fine surface details and because our discretization of the tangent-point energy hinges upon fine polyhedral discretizations, anyway (see \cref{sec:DiscTP}.)

Consider a triangle mesh $\Domain \ceq (V,E,T)$ with vertex set $V$, edge set $E$, and triangle set $T$ approximating a $C^{1,\alpha}$ surface $\varSigma$.
Again, we denote by $\kappa>0$ the wave number.
According to the suggestions in \cite{bemprodalg}, we choose the function space $\mathcal{S}^1(\Domain)$ of continuous piecewise linear functions, to discretize both function spaces $\Bessel[\frac{1}{2}][][\varSigma][\C]$ and $\Bessel[-\frac{1}{2}][][\varSigma][\C]$.
Abbreviating $\dd x = \dd \HsdM^2(x)$ and $\dd y = \dd \HsdM^2(y)$, the boundary integral operators 
$
\mathcal{V},\mathcal{K} \colon \mathcal{S}^{1}(\Domain) \to \mathcal{S}^{1}(\Domain)
$
are then discretized in the weak formulation by 
\begin{subequations}\label{eq:galerkin}
\begin{align}
	\inner{\psi,\mathcal{V} \varphi}_{\Lebesgue[2]}
	&= 
	\sum_{t_1,t_2\in T}
	\int_{t_1}\int_{t_2} \varPhi_\kappa(x,y) \, \varphi(y) \, \psi(x) \dd y \dd x
	\label{eq:sing}
	\\
	\inner{\psi,\mathcal{K} \varphi}_{\Lebesgue[2]}
	&= 
	\sum_{t_1,t_2\in T}
	\int_{t_1}\int_{t_2} \frac{\partial\varPhi_\kappa(x,y)}{\partial\nu(y)} \, \varphi(y) \, \psi(x) \dd y \dd x
\end{align}
\end{subequations}
for $\varphi,\psi \in\mathcal{S}^1(\Domain)$.
We further denote the space of piecewise constant functions on the faces by $\mathcal{S}^0(\Domain)$,
and introduce the area-weighted averaging operator
$
	\Av \colon \mathcal{S}^1(\Domain) \to \mathcal{S}^0(\Domain)
$
by 
$\pars{\Av \, \varphi}(t) \ceq \int_t \varphi(x) \dd x$.
For two distinct triangles $t_1$, $t_2 \in T$ the integrals \cref{eq:galerkin} are non-singular, and we can therefore use a simple midpoint quadrature:
We denote the barycenters of the triangles by $m_1$ and $m_2$, respectively.
In this case we obtain the following approximations:
\begin{align*}
	\int_{t_1}\int_{t_2} \varPhi_\kappa(x,y) \,\varphi(y) \, \psi(x) \dd y \dd x
	&\approx 
	\varPhi_\kappa(m_1,m_2) \pars{\Av \varphi}(t_2) \pars{\Av \psi}(t_1) 
	\\
	\int_{t_1}\int_{t_2} \frac{\partial\varPhi_\kappa(x,y)}{\partial\nu(y)} \, \varphi(y) \, \psi(x) \dd y \dd x
	&
	\approx  
	\frac{\partial\varPhi_\kappa(m_1,m_2)}{\partial\nu(m_2)} \pars{\Av \varphi}(t_2) \pars{\Av \psi}(t_1) 
\end{align*}
The integral kernel of the single-layer operator is weakly singular. 
Hence, the case $t_1=t_2$ cannot be neglected and requires a special treatment, which we shall outline now.

\subsubsection{Semi-analytic quadrature}

\begin{figure}[t]
	\centering
	\includegraphics[width=0.45\textwidth,trim=40 40 100 60,clip]{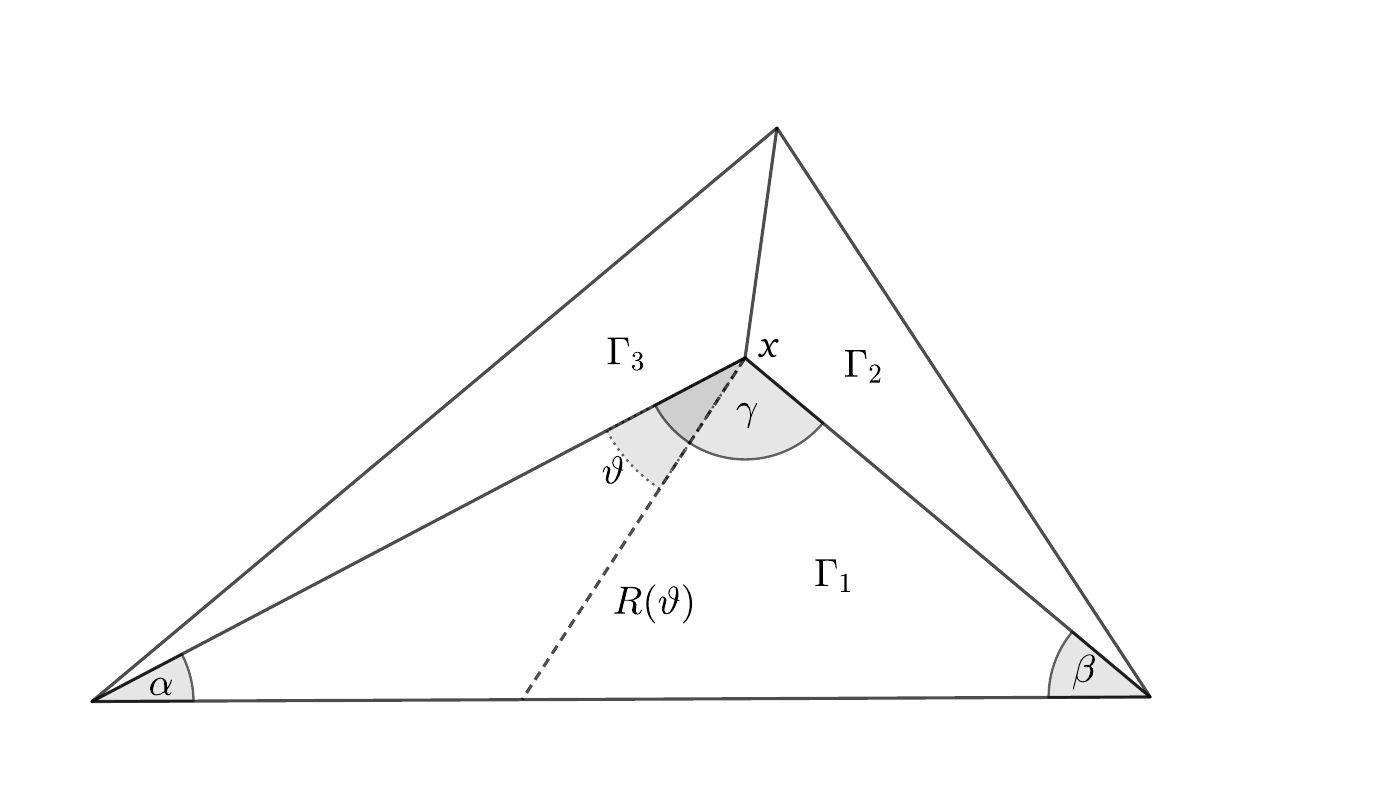}
	\captionof{figure}{Splitting of the triangle for the semi-analytic quadrature.}
	\label{fig:triangle}
\end{figure}

To handle the singular part, we require a special quadrature rule which is well-described in \cite{bem}.
Let $t\in T$ be any triangle and $ x\in t$ an arbitrary point.
Later, for the quadrature, this point will be specified to be the barycenter.
Using the notation from the last section, the inner part of the four-dimensional integral \cref{eq:sing} over $t\times t$ splits into a weakly singular and two nonsingular integrals:
\begin{align*}
	\int_t \varPhi_\kappa(x,y) \, \varphi(y) \dd y
	&= 
	\int_t \frac{\varphi(x)}{4\uppi \abs{x-y} } \dd y
	+ 
	\int_t \frac{\varphi(y)-\varphi(x)}{4\uppi \abs{ x-y }} \dd y
	+ 
	\int_t \frac{\ee^{\ii \, \kappa \abs{x-y}}-1}{4\uppi \abs{ x-y }}\varphi(y) \dd y.
\end{align*}
Using polar coordinates, one obtains that the second term on the right-hand side is of order $\mathcal{O}(h^2)$, where $h$ denotes the mesh size.
Therefore, this term can be neglected as we cannot expect a higher order of convergence than $\mathcal{O}(h)$. 
This leads to the following approximation at the midpoint $m$ of $t$ by using midpoint quadrature:
\begin{align*}
	\int_t \varPhi_\kappa(m,y) \, \varphi(y) \dd y
	\approx 
	\frac{\varphi(m)}{4\uppi}\int_t \frac{1}{\vert m-y\vert} \dd y
	+ 
	\frac{\ii  \kappa}{4\uppi} \pars{\Av \varphi}(t)
\end{align*}
An elementary computation reveals that the singular integral on the subtriangle $\varGamma_{1} \subset t$ (see \cref{fig:triangle}) has the following closed form:
\begin{align*}
	\int_{\varGamma_{1}} \frac{1}{\vert x - y\vert} \dd y
	&= 
	\int_{0}^{\gamma} R(\vartheta) \dd \vartheta 
	= 
	R(0) 
	\cos(\alpha)
	\pars*{ 
		\ln \pars[\big]{ \cot\pars[\big]{\tfrac{\beta}{2}}\cdot\cot\pars[\big]{\tfrac{\alpha}{2}} }
	}\\
	&=
	R(0)
	\cos(\alpha)
	\pars[\big]{ \atanh(\cos(\alpha)) + \atanh(\cos(\beta)) }
	.
\end{align*}
From this one compute the integral on the whole triangle $t = \varGamma_1 \cup \varGamma_2 \cup \varGamma_3$.

\subsubsection{Matrix-matrix product}
\label{sec:MatrixProduct}

In our applications we use several incident waves with the same wavelength, but with various wave directions  at once. Sometimes we also use multiple different wave numbers for the reconstruction.
This means that we solve multiple boundary value problems as in \cref{eq:bvp} simultaneously.
We denote the number of incident directions by $D\in\N$ and the wave numbers by $\kappa_1,\ldots,\kappa_L$ for $L \in \N$.
The number of total waves is $N=L \cdot D$.
When solving a discretized version of \cref{eq:integral} with an iterative solver, we thus need to perform matrix-matrix products of the form
\begin{equation*}
	[A_{\kappa_1}\cdot B_{\kappa_1},\ldots,A_{\kappa_L}\cdot B_{\kappa_L}]
	\;\;\text{with}\;\;
	A_{\kappa_\ell}\in \C^{\abs{T}\times\abs{T}}
	\;\;\text{and}\;\;
	B_{\kappa_\ell}\in \C^{\abs{T}\times D}
	\;\; 
	\text{for $\ell = 1,\ldots, L$}.
\end{equation*}
Here the square matrices $A_{\kappa_\ell}$ represent the discretized boundary integral operators with respect to $\kappa_\ell$ (as for instance in \cref{eq:integral}) and $B_{\kappa_\ell}$ represents a set of $D$ functions on $\varSigma = \partial \varOmega$, one for each incident direction.
For a given $\kappa = \kappa_\ell$, $\ell=1,\ldots,L$, the matrix-matrix product $C_\kappa = A_\kappa \cdot B_\kappa$ is calculated on the GPU with the following scheme, which is visualized in \cref{fig:matmat}:

Each work group of size $W$ holds a matrix $a$ of size $W \times W$, a matrix $b$ of size $W \times D$, and a matrix $c$ of size $W \times D$, each representing a block of $A_\kappa$, $B_\kappa$, and $C_\kappa$, respectively. 
More precisely, each thread holds a single line of $a$ and a single line of $c$ in local memory; $b$ is held in shared memory.
Then the work group assembles a block of $A_\kappa$ into $a$.
Next the corresponding $W \times D$ block of $B_\kappa$ is loaded into $b$.
These blocks are multiplied with each other and the result is added into the local copy: $c \leftarrow c +  a \cdot b$. 
Afterwards $a$ is overwritten by the next block of $A_\kappa$ to the \emph{right}; the corresponding block $B_\kappa$ \emph{below} the former one is loaded into $b$; then $c \leftarrow c +  a \cdot b$ is computed; and so on.
When a whole row of blocks of $A_\kappa$ is processed, then $c$ is written back into $C_\kappa$ in global memory. 
This way, each row of $C_\kappa$ is written to by exactly one thread, avoiding write conflicts.
If $D$ is too large, then the $W \times D$ blocks of $C_\kappa$ and $B_\kappa$ won't fit into local and shared memory. On our hardware, $D \leq 16$ worked well. If there are more waves with the same wavenumber, then they have to be divided into smaller groups.

\begin{figure}[t]
	\centering
	\includegraphics[trim = 0 585 0 60, 
		clip = true, width=0.85\textwidth]{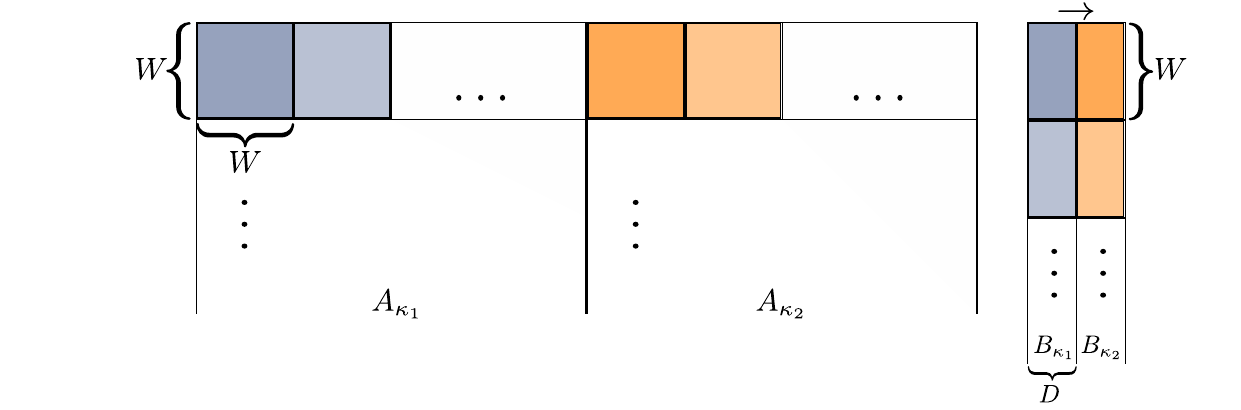}
	\captionof{figure}{Visualization of the matrix-matrix product. Each work group handles a block of size $W \times W$ of $A_{\kappa_\ell}$
	and a block of size $W \times D$ of $B_{\kappa_\ell}$, where  $W$ is the work group size and $D$ is the number of incident wave directions.}
	\label{fig:matmat}
\end{figure}

This approach never stores any of the matrices $A_{\kappa_1},\dotsc,A_{\kappa_\ell}$ in full;
this would not be feasible for the mesh resolutions we use, let alone storing all of them at the same time.
Nonetheless, if there are multiple incoming waves with the same wavelength, then the assembled blocks of $A_{\kappa_i}$ are reused multiple times, minimizing the number of reassemblies.
This allows us to efficiently handle multiple right-hand sides in the operator equation \cref{eq:integral}.
We are not aware of any Helmholtz BEM library that was capable of this functionality before.

\subsubsection{Tests and benchmarks}

\begin{table}[t]
	\centering
	\begin{tabular}{l|l|l|l|l}
	\emph{shape} & \emph{$\#$ triangles} & error\ \  $\mathcal{V}$ & error\ \ $\mathcal{K}$ &  error \  $F$\\
	\hline
	Sphere & $40560$ & $8.868 \cdot 10^{-4}$ & $2.147 \cdot 10^{-3}$ & $2.524 \cdot 10^{-3}$\\
	Sphere & $81920$ & $4.037 \cdot 10^{-4}$ & $1.391 \cdot 10^{-3}$ & $1.771 \cdot 10^{-3}$\\
	Cow	   & $93696$ & $7.772 \cdot 10^{-4}$ & $9.641 \cdot 10^{-3}$ & $6.790 \cdot 10^{-3}$\\
	\end{tabular}
	\begin{tabular}{l}
	\end{tabular}
	\caption{Relative max error to the BEMpp solution for $\kappa=\uppi$.
	$F$ denotes the discretized boundary-to-far field map.}
	\label{tab:bempperr}
\end{table}

\begin{figure}[t]
	\centering
	\includegraphics[width=0.8\textwidth]{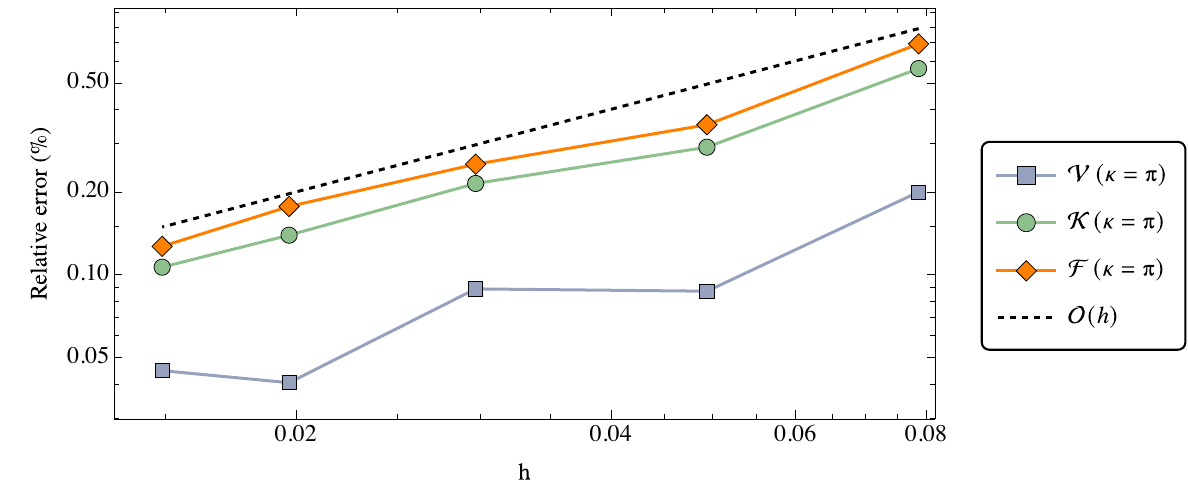}
	\caption{%
	A~log-log plot of the relative max error of the operator images against the average edge length $h$ for a collection of spheres.
	$F$ denotes the discretized boundary-to-far field map.}
	\label{fig:bempperr}
\end{figure}

We present some accuracy tests and benchmarks for our solver as a justification for its usage.
Unfortunately, for non-trivial input functions and surfaces, a closed formula for the integral operators is not available. 
For this reason, we tested BAEMM against a well-established BEM library.
We chose to use BEMpp \cite{bempp} as ground truth, as our approach to the integral equation is similar to theirs; just their quadrature rules are a lot more sophisticated and thus more accurate than ours.
For the tests, displayed in \cref{tab:bempperr} and \cref{fig:bempperr}, the wavenumber is set to $\kappa = \uppi$ and the operators are applied to a plane wave function.
These tests suggest that in the example of spheres one indeed obtains a convergence of order $\mathcal{O}(h)$ to the true solution of \cref{eq:integral}.

Since one of our main goals was to keep an eye on performance, we now turn to benchmarks of the operator evaluation.
\cref{tab:time} shows the timings of the boundary operator evaluations on $16$ planar waves, respectively the far field evaluations for $16$ incident waves, simultaneously with wavenumber $\kappa=\uppi$. 
The timings are averaged over $10$ runs.
We employed an NVIDIA A40 GPU on an AMD EPYC 7313 16-Core CPU.

\begin{table}[ht]
\begin{center}
\begin{tabular}[ht]{l|l|l|l|l}
\emph{shape} & \emph{$\#$ triangles} & apply \ $\mathcal{V}$ & apply \ $\mathcal{K}$ & eval \ $F$\\
\hline
Sphere   & $40560$ & $0.0612\,\mathrm{s}$ & $0.0583\,\mathrm{s}$ & $1.0771\,\mathrm{s}$\\
Sphere   & $81920$ & $0.2256\,\mathrm{s}$ & $0.2193\,\mathrm{s}$ & $1.7708\,\mathrm{s}$\\
``Spot'' & $93696$ & $0.2802\,\mathrm{s}$ & $0.3233\,\mathrm{s}$ & $5.7504\,\mathrm{s}$
\end{tabular}
\caption{%
Averaged runtimes for $16$ simultaneous operator evaluations in seconds with wave number $\kappa=\uppi$. For the evaluation of the far field operator $F$, the discretized version of \cref{eq:integral} is solved and \cref{eq:farfield} is evaluated at $2562$ evaluation points on $\S^2$. 
The relatively long evaluation time for the far field of ``Spot'' can be explained by bad conditioning of the operator equation~\cref{eq:integral}.
The mesh ``Spot'' is shown in \cref{tab:meshes}.}
\label{tab:time}
\end{center}
\end{table}

\subsection{Discretization of the tangent-point energy and fractional operators} 
\label{sec:DiscTP}

To compute the tangent-point energy and its derivative, we employ the \emph{Repulsor} library \cite{Schumacher_Repulsor}.
Although the tangent-point energy is nonlocal, \emph{Repulsor} can do this in time roughly proportional to the number of triangles in the mesh by employing a zeroth order multipole approximation and a tree-accelerated code. 
The library also provides a fast implementation (i.e., nearly linear in the number of triangles) of the action of the operator $\Riesz_f$ (see \cref{eq:RieszOperator}, \cref{prop:welldef}), accelerated by the same tree structures and stored as a hierarchical matrix.
For more details on the discretization of the energy $\Energy(f)$ and the operator $\Riesz_f$ see \cite{Sassen:2024:RS,RepulsiveS,RepulsiveC}.

Being a differential operator of order $3 < 2 \, s < 4$, the operator $\Riesz_f$ is badly conditioned with respect to the $\Lebesgue[2]$ inner product.
Therefore, in order to accelerate the iterative linear solve with $\Riesz_f$, we require a preconditioner.
As discussed in \cref{sec:Preconditioner}, $\Riesz_f$ is roughly $(-\Laplacian_f)^s = (-\Laplacian_f)^{-1}(-\Laplacian_f)^{2-s}(-\Laplacian_f)^{-1}$  (see~\cref{eq:FractionalLaplacianDomSpace}). 
So we construct our preconditioner accordingly:
As it is well-known, the finite-element discretization of  $(-\Laplacian_f)$ is a sparse matrix, the so-called \emph{cotan-Laplacian}.
Thus, we can apply $(-\Laplacian_f)^{-1}$ efficiently by sparse Cholesky factorization followed by forward and backward substitution. 
What remains is $(-\Laplacian_f)^{2-s}$.
Since $1 < s < 2$, we have $0 < 2 - s < 1$.
Analogous to the approach in \cref{sec:Preconditioner}, we replace it by the Gagliardo form
$
	\int_\Domain \! \int_\Domain
		\inner{ \DeltaOp[2-s][f] u, \DeltaOp[2-s][f] v } 
	\dd \Measure[f] 
$. 
Like $\Riesz_f$, the action of the underlying operator is implemented as hierarchical matrix and provided by \emph{Repulsor}.
For more details on this approach see \cite{RepulsiveS}, where this preconditioner was introduced.

\subsection{Implementation of the Gauss--Newton type algorithm} \label{sec:implementation_algorithm}

Here we describe how to implement the adapted version of the iteratively regularized Gauss--Newton method proposed in \cref{sec:IRGNM}.

We use a GMRES solver to solve the discretized version of \cref{eq:integral}.
This is needed to evaluate 
(i)~the far field $F(f_k)$, 
(ii)~the derivative  
$DF(f_k)' \, J_{\Lebesgue[2]} \, (g^\delta - F(f_k))$ of the fitting term, as well as 
(iii)~the action of the Gauss--Newton term 
$DF(f_k)' \, J_{\Lebesgue[2]} \, DF(f_k)$.
GMRES provides the freedom to choose the error tolerance in the stopping criterion.
In our experiments it turned out that (i) requires some decent accuracy (because the objective depends on it), while we may sacrifice quite some accuracy in (ii) and (iii) in favor of speed. 
After all, (ii) and (iii) influence only the search direction (see \cref{sec:Optimization}), and it is well-known that descending algorithms are quite robust against some mild errors in the search direction. 
In practice, we used error tolerances for (ii) and (iii) that are greater by 1--2 orders of magnitude than those for (i).

We now describe the computation of the update direction. 
For this we need to solve a linear equation roughly of the following form (see \cref{eq:update_step2} for details):
\begin{equation}
	\label{eq:GradientEquation}
	\pars[\big]{
		DF(f)' \, J_{L^2}  \, DF(f) + \alpha \, \Riesz_{f}
	} \, v
	= 
	-
	\pars[\big]{
		DF(f)' \, J_{\Lebesgue[2]} \, (F(f) - g^\delta) + \alpha \,  \overline{D\Energy}(f)
	}
	.
\end{equation}
As we have shown in the proof of \cref{lem:UpdateDirectionIsLipschitzContinuous}, the linear operator on the left-hand side is a compact perturbation of $\alpha \, \Riesz_f$.
Hence, the preconditioner from \cref{sec:DiscTP} is effective in preconditioning this operator, too.
Ideally, one would like to use a CG-solver for this purpose. 
Alas, we determine the action of $DF(f)$  and $DF(f)'$ only approximately, so our implementation of
$DF(f)' \, J_{L^2}  \, DF(f)$ is not exactly self-adjoint.
This indeed leads to some issues with the CG algorithm, so we utilize GMRES also here.
Again, with the same reasoning as above, we use a quite high error tolerance of about $1 \cdot 10^{-2}$ for the GMRES stopping criterion.

The next step in calculating the update is to find the step size for \cref{eq:update}.
We do so as detailed out in \cref{rem:LineSearch2} by continuous collision detection (accelerated by bounding volume hierarchies, also provided by \emph{Repulsor}), followed by a simple backtracking line search.
This also prevents self-intersections of the updated surface.

After calculating the surface update, we occasionally perform some remeshing (edge collapse and  edge splits) to prevent the triangles from degenerating; after all, we experience quite some substantial deformation.
Afterwards we perform Delaunay flips until we obtain a Delaunay triangulation and run a few rounds of simple Laplacian smoothing to repair most of the roughness introduced by our remeshing method.
The \emph{Repulsor} library also provides facilities for this.


\begin{table}
	\begin{center}
	\begin{tabular}[h]{l|l|l}
	\emph{shape} & \emph{$\#$triangles true} & \emph{$\#$triangles init.}\\
	\hline
	\emph{Stanford bunny} & $86632$ & $40560$ \\
	\emph{Triceratops} & $90560$ & $81920$\\
	Cow \emph{Spot}& $93696$ & $81920$\\
	Deformed torus & $149340$ & $38400$ \\
	Rubber duck \emph{Bob}& $171008$ & $38400$\\
	Fish \emph{Blub} & $227328$ & $81920$\\
	\end{tabular}
	\hspace{2em}
	\begin{tabular}{l}
	\includegraphics[trim = 110 25 160 65, 
		clip = true,width=0.22\textwidth]{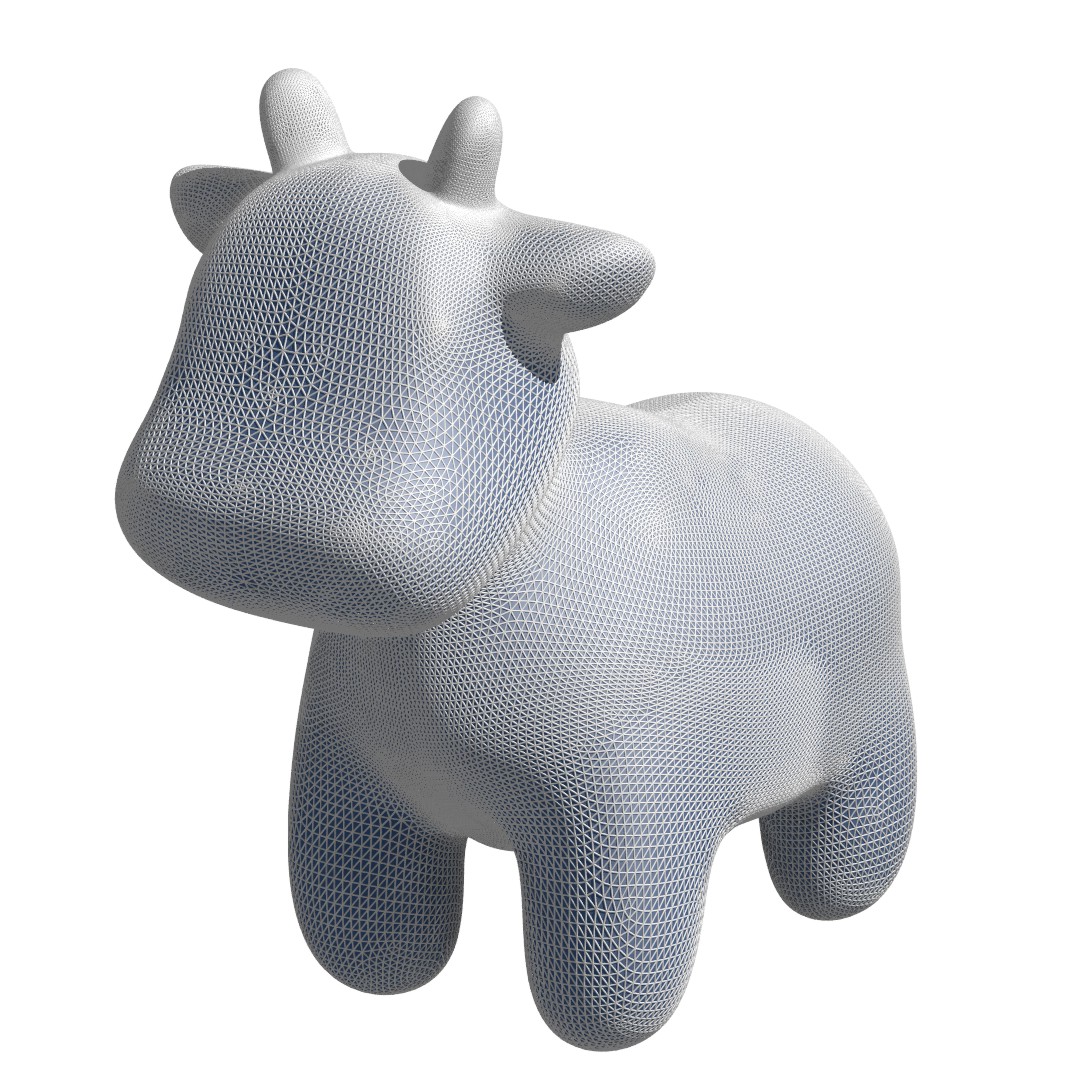}
	\end{tabular}
	\caption{Number of triangles of the meshes representing the true obstacles (used to compute the noisy far field data) and the initial guesses, respectively. 
	On the right-hand side we show the mesh \emph{Spot} to give a visual impression of the employed mesh resolutions.}
	\label{tab:meshes}
	\end{center}
	\end{table}
	
	\begin{table}[t]
	\begin{center}
	\begin{tabular}[h]{l|r|r|r|r|r|r|r}
	\emph{shape}& \emph{\#triangles} & $\emph{wavenumbers}$  & $\NumWaves$ & $\tau$ & \emph{$\#$iter.} & \emph{time}
	& \emph{Hsd.~d.}\\
	\hline
	\emph{Stanford bunny} & $40560$ & $\braces{\uppi,\, 2\,\uppi}$ & $32$ & $3$ & $28$ &$52\,\mathrm{min}$ & $0.0926$\\
	\emph{Stanford bunny} & $40560$ & $\braces{4\,\uppi}$ & $16$ & $2$ & $25$ & $33\,\mathrm{min}$ & $0.0441$\\
	\emph{Triceratops} & $81920$ &  $\braces{\uppi,\, 2\,\uppi}$ &$32$& $4$ & $24$ & $82\,\mathrm{min}$ & $0.0709$\\
	\emph{Triceratops} & $81920$ &  $\braces{4\,\uppi,\, 6\,\uppi}$ & $32$ & $4$ &$26$ & $277\,\mathrm{min}$ & $0.0335$\\
	\emph{Triceratops} & $81920$ &  $\braces{8\,\uppi}$ & $16$ &$3$ & $17$ & $99\,\mathrm{min}$ & $0.0183$\\
	\end{tabular}
	\caption{%
		Quantitative information for the reconstructions of \emph{Stanford bunny} and the \emph{Triceratops}.
		The number $\NumWaves$ refers to the total number of incident waves.
		The number of incident directions is the number of waves divided by the number of wavelengths.
		Again, $\tau$ denotes the discrepancy parameter for the stopping criterion.
		The column \emph{$\#$iter.} lists the number of Gauss--Newton iterations needed to reach the stopping criterion and ``time'' lists the total computation times.
		The last column displays the Hausdorff distance between the reconstruction and the true obstacle divided by the diameter of the true obstacle.}
	\label{tab:computation_times}
	\end{center}
\end{table}

\newpage 

\section{Numerical results}
\label{sec:num}
In this section we demonstrate the versatility of our proposed method for solving inverse obstacle scattering  problems for surfaces in $\R^3$.

If not stated otherwise, we used simulated data perturbed by $1\%$ noise (independent and identically distributed Gaussian white noise), $2562$ roughly evenly distributed far field evaluation points on the $\S^2$, and a regularization parameter $\alpha_k=\alpha_0\cdot\rho^k$ in iteration $k\geq 0$ with $\rho = 0.8$ for the experiments in this section.
This amount of evaluation points was chosen to make use of the closed formula for $DF'$ from \cite[Theorem~3]{zbMATH05251341}, which requires knowledge of the entire far field.

To terminate the regularization, we use the discrepancy principle as pointed out in \cref{sec:IRGNM}.
To avoid running into local minima, the general procedure for the reconstructions was to start the reconstruction of the object with wavelength $\lambda = 2\, \uppi / \kappa \sim \mathrm{object\,size}$ and proceed to use the reconstruction as a warm start for reconstructions with higher wavenumbers (see, e.g., \cref{fig:Triceratops}).
For visualization purposes, we included scales that visualize the wavelength(s) $\lambda = 2 \, \uppi / \kappa$.


\begin{figure}[t]%
	\begin{center}%
	\newcommand{\inca}[2]{%
	\begin{tikzpicture}%
		\node[inner sep=0pt] (fig) at (0,0) {\includegraphics[	trim = 30 30 30 30, 
		clip = true,  
		angle = 0,
		width = 0.25\textwidth]{#1}};%
		\node[above right= 0ex] at (fig.south west) {\begin{footnotesize}(#2)\end{footnotesize}};%
	\end{tikzpicture}%
	}%
    \newcommand{\incb}[2]{%
	\begin{tikzpicture}%
		\node[inner sep=0pt] (fig) at (0,0) {\includegraphics[	trim = 0 20 0 15, 
		clip = true,  
		angle = 0,
		width = 0.50\textwidth]{#1}};%
		\node[above right= 0ex] at (fig.south west) {\begin{footnotesize}(#2)\end{footnotesize}};%
	\end{tikzpicture}%
	}%
	\inca{Figure_Triceratops/triceratops_BAEMM_16inc_TPM0_8pi_remeshed_new_tau_3_rho_0.8/True}{a}%
    \,
    \incb{Figure_Triceratops/triceratops_BAEMM_16inc_TPM0_pi2pi_remeshed/Errors}{b}%
    \\
	\inca{Figure_Triceratops/triceratops_BAEMM_16inc_TPM0_pi2pi_remeshed/Reconstructed}{c}%
    \inca{Figure_Triceratops/triceratops_BAEMM_16inc_TPM0_4pi6pi_remeshed/Reconstructed}{d}%
    \inca{Figure_Triceratops/triceratops_BAEMM_16inc_TPM0_8pi_remeshed_new_tau_3_rho_0.8/Reconstructed}{e}%
	\end{center}
	\caption{%
		(a)~The mesh ``Triceratops'' as an obstacle.
		(b)~The errors of the reconstruction during the optimization on the way from the initial guess (a round sphere) to (c). 
		The Hausdorff distance is measured relatively to the diameter of the true obstacle.
		(c)~Coarse reconstruction for deliberately large wavelengths.
		(d)~Reconstruction for intermediate wavelengths.
		(e)~Final reconstruction; starting from (d) 
		it took $13$ Gauss--Newton steps to get to (e).
	}%
	\label{fig:Triceratops}
\end{figure}
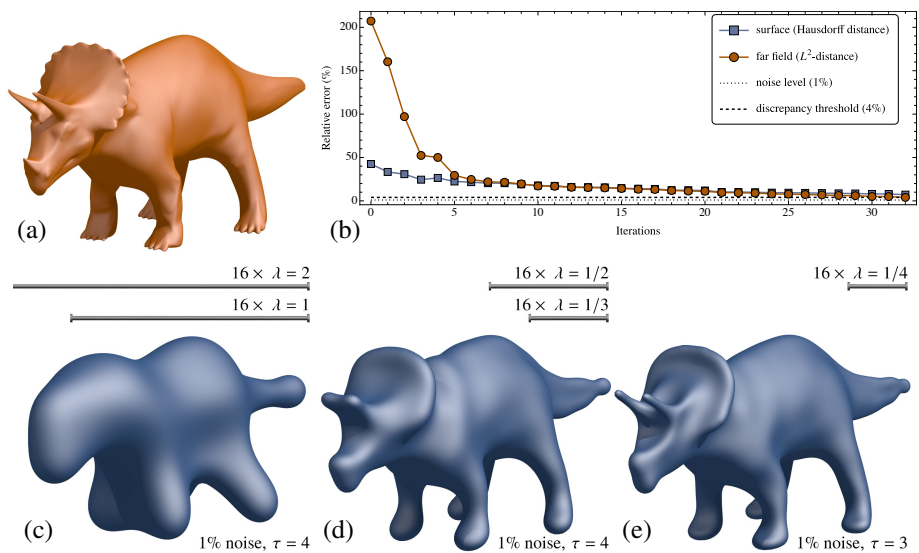

\subsection{Spherical examples}
We start this section with some examples of spherical but non-star shaped objects.
\cref{fig:Bunny} showed already the reconstruction of the \emph{Stanford bunny}.
\cref{fig:Triceratops} displays the reconstruction of the detailed triangulated model of \emph{Triceratops}.
The mesh sizes are displayed in \cref{tab:meshes}.
Further, \cref{tab:computation_times} shows the number of Gauss--Newton iterations needed, the overall computation time for these reconstructions, and the Hausdorff distance of the reconstructed surface to the original surface.
The main part of the computation time is the calculation of the update step.
We provide a rough estimate for the requisite number of matrix-matrix products for the reconstruction of the triceratops with wavenumbers $\braces{4\,\uppi,6\,\uppi}$.
The matrix-matrix products we are counting, which are calculated on the GPU, are of size $(81920\times 81920)$ times $(81920\times 32)$.
Taking some sample steps of the reconstruction one sees that the GMRES, to solve \cref{eq:GradientEquation}, takes between $\sim 20$ and $\sim 90$ iterations each of which needs to evaluate $DF$ and $DF'$ once.
Observing that each of the operator evaluations takes $\sim 10$ iterations itself, one obtains a rough estimate of $400$ to $1800$ matrix-matrix products needed to calculate the update direction.
Compared to this, the time spent by \emph{Repulsor} to compute the tangent-point energy and to apply the operator $\Riesz_f$ was negligible.

We proceed with a reconstruction of the mesh \emph{Spot} with a higher noise level of $10\%$ (see \cref{fig:Spot}) and compare the results of the usage of different combinations of wavelengths.

The examples in \cref{fig:Blub} show that choosing wavelengths with a high least common integer multiple can help to prevent some artifacts. In particular, this can be seen at the tail base and in the ``wake'' of the lateral fin in \cref{fig:Blub}~(b) and \cref{fig:Blub}~(c).
The overall better reconstruction in \cref{fig:Blub}~(d) can be explained by the lower wavelength.


\begin{figure}[t]
    \centering
    \newcommand{\inca}[2]{%
    \begin{tikzpicture}%
       \node[inner sep=0pt] (fig) at (0,0) {\includegraphics[	
        trim = 80 10 110 10, 
        clip = true,  
        angle = 0,
        width = 0.24\textwidth]{#1}};%
        \node[above right= -.5ex] at (fig.south west) {\begin{footnotesize}(#2)\end{footnotesize}};%
    \end{tikzpicture}%
    }%
    \begin{minipage}[c]{0.95\textwidth}%
        \inca{Figure_Spot/spot_BAEMM_16inc_TPM0_5pi7pi_remeshed_10perc_noise_tau_1_2/True_Colored}{a}%
        \hfill
        \inca{Figure_Spot/spot_BAEMM_16inc_TPM0_6pi_remeshed_10perc_noise_tau_1_2/DistanceField}{b}%
        \hfill
        \inca{Figure_Spot/spot_BAEMM_16inc_TPM0_5pi7pi_remeshed_10perc_noise_tau_1_2/DistanceField}{c}%
        \hfill 
        \inca{Figure_Spot/spot_BAEMM_16inc_TPM0_5pi7pi_remeshed_10perc_noise_tau_1_2/True_DistanceField}{d}%
    \end{minipage}%
    \hfill
    \begin{minipage}[c]{0.043\textwidth}%
        \includegraphics[
            trim = 0 0 0 0, 
            clip = true,  
            angle = 0,
            width = \textwidth
        ]{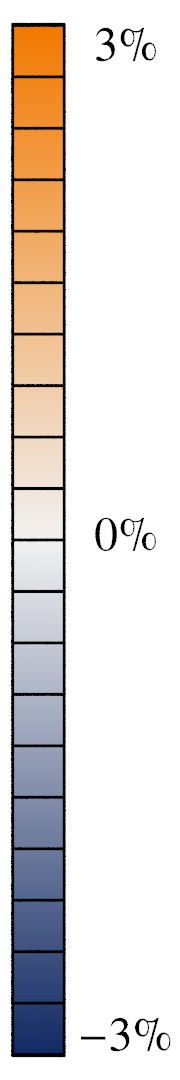}%
    \end{minipage}%
    \caption{%
        (a)~The mesh \emph{Spot} as the obstacle.
        In this example, the synthetic far field data was perturbed by $10\%$ white noise.
        (b), (c) The signed distance fields of the true obstacle, plotted over the reconstructions.
        (d) The signed distance fields of (c), plotted over the true obstacle.
    }%
    \label{fig:Spot}
\end{figure}
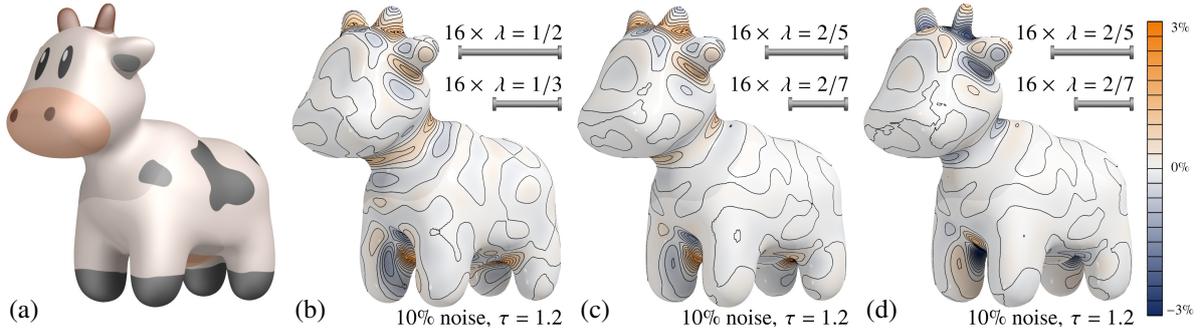


\begin{figure}[t]
    \centering
	\newcommand{\inca}[2]{%
		\begin{tikzpicture}%
			\node[inner sep=0pt] (fig) at (0,0) {\includegraphics[%
                trim = 30 10 20 20,%
                clip = true,%
                angle = 0,%
                width = 0.25\textwidth%
            ]{#1}};%
			\node[below right= 0.5ex] at (fig.north west) {\begin{footnotesize}#2\end{footnotesize}};%
		\end{tikzpicture}%
	}%
    \newcommand{\incb}[2]{%
        \begin{tikzpicture}%
            \node[inner sep=0pt] (fig) at (0,0) {\includegraphics[%
                trim = 30 40 20 90,%
                clip = true,%
                angle = 0,%
                width = 0.25\textwidth%
            ]{#1}};%
            \node[below right= 0.5ex] at (fig.north west) {\begin{footnotesize}#2\end{footnotesize}};%
        \end{tikzpicture}%
    }%
    \begin{minipage}[c]{\textwidth}%
        \begin{minipage}[c]{0.95\textwidth}%
            \inca{Figure_Blub/blub_BAEMM_16inc_TPM0_5pi_7pi_remeshed_tau_3_saved/Side_True_Colored}{(a)}%
            \hfill
            \inca{Figure_Blub/blub_BAEMM_16inc_TPM0_3pi5pi_remeshed_tau_3_saved/Side_DistanceField}{(b)}%
            \hfill
            \inca{Figure_Blub/blub_BAEMM_16inc_TPM0_4pi_6pi_remeshed_tau_3_saved/Side_DistanceField}{(c)}%
            \hfill
            \inca{Figure_Blub/blub_BAEMM_16inc_TPM0_5pi_7pi_remeshed_tau_3_saved/Side_DistanceField}{(d)}%
            \\
            \incb{Figure_Blub/blub_BAEMM_16inc_TPM0_5pi_7pi_remeshed_tau_3_saved/Top_True_Colored}{\phantom{(a)}}%
            \hfill
            \incb{Figure_Blub/blub_BAEMM_16inc_TPM0_3pi5pi_remeshed_tau_3_saved/Top_DistanceField}{\phantom{(b)}}%
            \hfill
            \incb{Figure_Blub/blub_BAEMM_16inc_TPM0_4pi_6pi_remeshed_tau_3_saved/Top_DistanceField}{\phantom{(c)}}%
            \hfill
            \incb{Figure_Blub/blub_BAEMM_16inc_TPM0_5pi_7pi_remeshed_tau_3_saved/Top_DistanceField}{\phantom{(d)}}%
        \end{minipage}%
        \hfill
        \begin{minipage}[c]{0.0485\textwidth}%
            \includegraphics[%
                trim = 0 0 0 0,%
                clip = true,%
                angle = 0,%
                width = \textwidth%
            ]{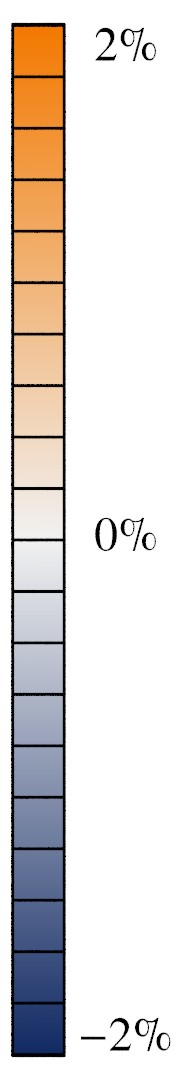}%
        \end{minipage}%
    \end{minipage}%
    \\
    \begin{minipage}[c]{\textwidth}
        \begin{minipage}[c]{0.95\textwidth}%
            \inca{Figure_Blub/blub_BAEMM_16inc_TPM0_5pi_7pi_remeshed_tau_3_saved/Side_True}{\phantom{(a)}}%
            \hfill
            \inca{Figure_Blub/blub_BAEMM_16inc_TPM0_3pi5pi_remeshed_tau_3_saved/Side_Reconstructed}{\phantom{(b)}}%
            \hfill
            \inca{Figure_Blub/blub_BAEMM_16inc_TPM0_4pi_6pi_remeshed_tau_3_saved/Side_Reconstructed}{\phantom{(c)}}%
            \hfill
            \inca{Figure_Blub/blub_BAEMM_16inc_TPM0_5pi_7pi_remeshed_tau_3_saved/Side_Reconstructed}{\phantom{(d)}}%
        \end{minipage}%
        \hfill 
        \phantom{.}
    \end{minipage}%
    \caption{%
    (a) The mesh \emph{Blub} as the obstacle.
    The columns (b), (c), (d) show the reconstructions with the signed distance field of the true obstacle for various combinations of wavelengths.
    Bottom row: 
    Dull surface renders highlight some artifacts associated to the surface normals, in particular, behind the lateral fins.
    }
    
    \label{fig:Blub}
\end{figure}
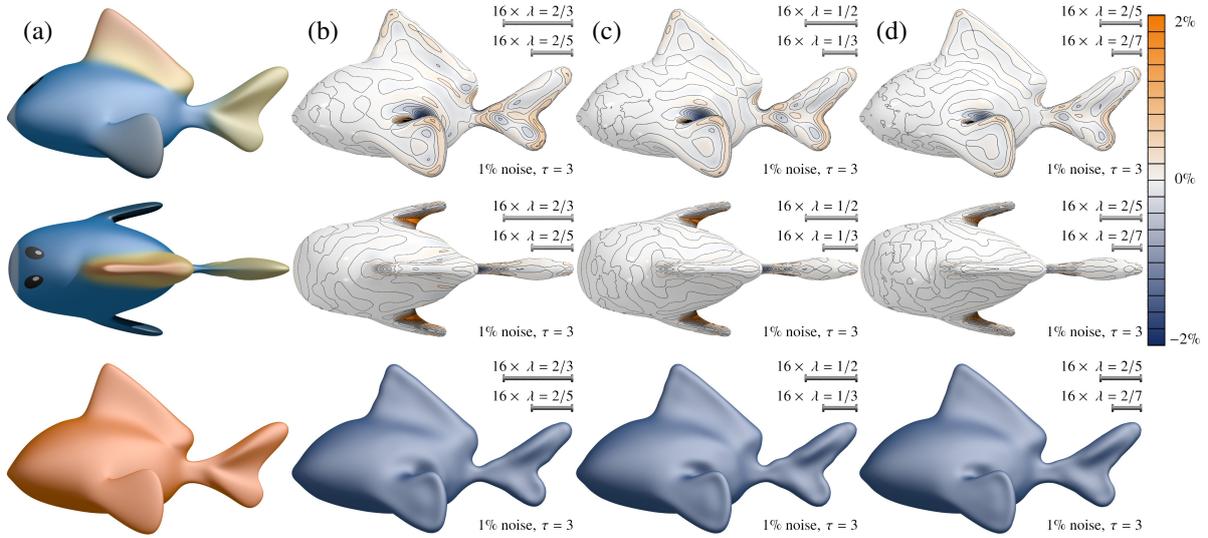

\subsection{Toroidal examples}


\begin{figure}[t]
	\centering
	\newcommand{\myheight}{0.23\textwidth}
	\newcommand{\inca}[2]{%
		\begin{tikzpicture}%
			\node[inner sep=0pt] (fig) at (0,0) {\includegraphics[	trim = 30 10 30 10, 
			clip = true,  
			angle = 0,
			height = \myheight]{#1}};%
			\node[above right= -0.5ex] at (fig.south west) {\begin{footnotesize}(#2)\end{footnotesize}};%
		\end{tikzpicture}%
	}%
	\inca{Figure_SnakeTorus/snake_torus_BAEMM_32inc_TPM0_pi_remeshed/Init_Compared}{a}%
	\hfill
	\inca{Figure_SnakeTorus/snake_torus_BAEMM_32inc_TPM0_pi_remeshed/DistanceField}{b}%
	\hfill
	\inca{Figure_SnakeTorus/snake_torus_BAEMM_32inc_TPM0_2pi_4pi_remeshed_tau_2/DistanceField}{c}%
	\hfill
	\inca{Figure_SnakeTorus/snake_torus_BAEMM_32inc_TPM0_2pi_4pi_remeshed_tau_2/True_DistanceField}{d}%
	\hfill
	\includegraphics[trim = 0 0 0 0, 
	clip = true,  
	angle = 0,
	height = \myheight]{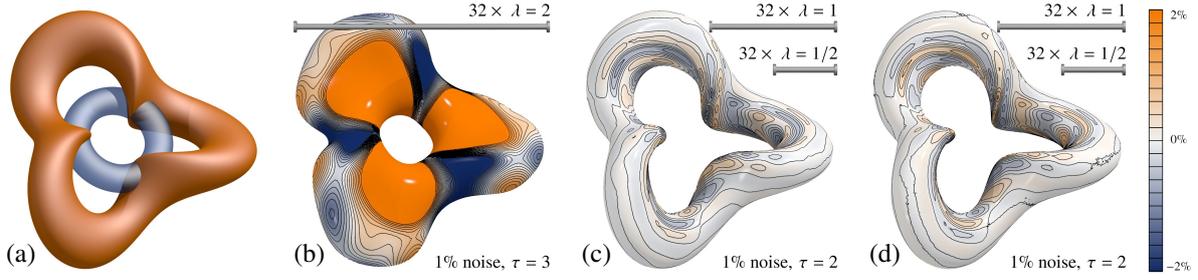}%
	\caption{%
		(a)~The true obstacle, a deformed torus (orange), together with the initial guess.
		(b)~A~deliberately rough reconstruction for large wavelengths plotted with the signed distance field of the true obstacle.
		(c)~The final reconstruction for smaller wavelengths, obtained from (b) after $22$ Gauss--Newton steps. 
		(d) The true obstacle with the signed distance field of the reconstruction (c).
	}%
	\label{fig:SnakeTorus}
\end{figure}

Notice that the algorithm proposed in \cref{sec:IRGNM} does not depend on any assumptions on the topology of the reference manifold -- except for the choice of an initial guess.
Therefore, we also present its behavior on surfaces with non-trivial topology.
The example presented in \cref{fig:SnakeTorus} is a standard torus, which we deformed periodically in radial and longitudinal direction and remeshed it.
One further adjustment we made for this experiment is that in \cref{eq:farfield} and \cref{eq:integral} we set $\eta=1$ instead of $\eta=\kappa$ for the evaluation of the synthetic data.
This was done to exclude the occurrence of an inverse crime, namely a strong dependence of the solution on the setup of the forward solver.
A second example is given in \cref{fig:Bob}, where we present the reconstruction of a synthetically created toroidal rubber duck with a torus of revolution as the initial guess.


\begin{figure}%
    \centering
    \newcommand{\inca}[2]{%
    \begin{tikzpicture}%
       \node[inner sep=0pt] (fig) at (0,0) {\includegraphics[	
            trim = 60 220 20 10, 
            clip = true,  
            angle = 0,
            width = 0.24\textwidth
        ]{#1}};%
        \node[above right= -.5ex] at (fig.south west) {\begin{footnotesize}(#2)\end{footnotesize}};%
    \end{tikzpicture}%
    }%
    \begin{minipage}[c]{0.96\textwidth}%
        \inca{Figure_Bob/bob_BAEMM_16inc_TPM0_4pi_remeshed_tau_1_5/True_Colored}{a}%
        \hfill
        \inca{Figure_Bob/bob_BAEMM_16inc_TPM0_4pi_remeshed_10perc_noise_tau_1_2/DistanceField}{b}%
        \hfill
        \inca{Figure_Bob/bob_BAEMM_16inc_TPM0_4pi_remeshed_tau_1_5/DistanceField}{c}%
        \hfill 
        \inca{Figure_Bob/bob_BAEMM_16inc_TPM0_4pi_remeshed_tau_1_5/True_DistanceField}{d}%
    \end{minipage}%
    \hfill
    \begin{minipage}[c]{0.0385\textwidth}%
        \includegraphics[
            trim = 0 0 0 0, 
            clip = true,  
            angle = 0,
            width = \textwidth
        ]{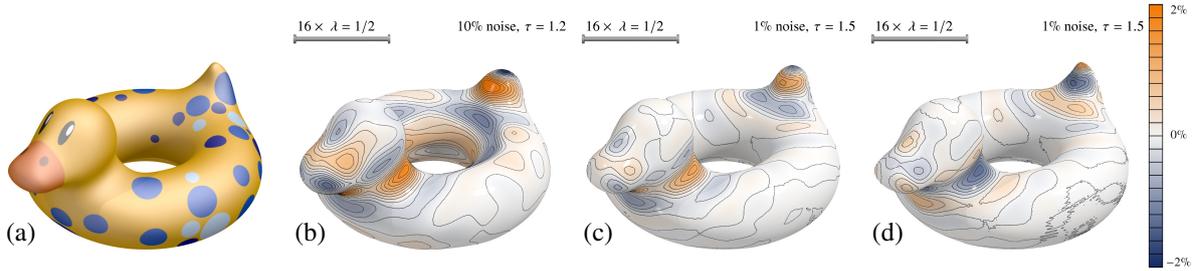}%
    \end{minipage}%
    \caption{%
    (a) The mesh \emph{Bob} as the obstacle.
    Our reconstructions with its signed distance field of the true obstacle: (b)~for $10\%$ noise white noise in the far field data and (c)~for $1\%$ noise.
    (d)
    The true obstacle with the signed distance field of the reconstruction~(c).
    }%
    \label{fig:Bob}
\end{figure}


\begin{figure}%
    \centering
    \newcommand{\myheight}{0.195\textwidth}
    \newcommand{\inca}[2]{%
    \begin{tikzpicture}%
       \node[inner sep=0pt] (fig) at (0,0) {\includegraphics[	
        trim = 60 220 20 10, 
        clip = true,  
        angle = 0,
        height = \myheight]{#1}};%
        \node[above right= -.5ex] at (fig.south west) {\begin{footnotesize}(#2)\end{footnotesize}};%
    \end{tikzpicture}%
    }%
    \newcommand{\incb}[2]{%
    \begin{tikzpicture}%
       \node[inner sep=0pt] (fig) at (0,0) {\includegraphics[	
        trim = 60 220 190 10, 
        clip = true,  
        angle = 0,
        height = \myheight]{#1}};%
        \node[above right= -.5ex] at (fig.south west) {\begin{footnotesize}(#2)\end{footnotesize}};%
    \end{tikzpicture}%
    }%
    \inca{Figure_Bob/bob_sphere_BAEMM_16inc_TPM0_3pi_4pi_remeshed_tau_2/DistanceField}{a}%
    \hfill 
    \inca{Figure_Bob/bob_sphere_BAEMM_16inc_TPM0_3pi_4pi_remeshed_tau_2/DistanceField_Sliced}{b}%
    \hfill
    \includegraphics[
        trim = 0 0 0 0, 
        clip = true,  
        angle = 0,
        height = \myheight
    ]{Figure_Bob/bob_BAEMM_16inc_TPM0_4pi_remeshed_tau_1_5/ColorBar}%
    \hfill
    \,\,\,\,
    \hfill
    \includegraphics[
        trim = 0 0 0 0, 
        clip = true,  
        angle = 0,
        height = \myheight
    ]{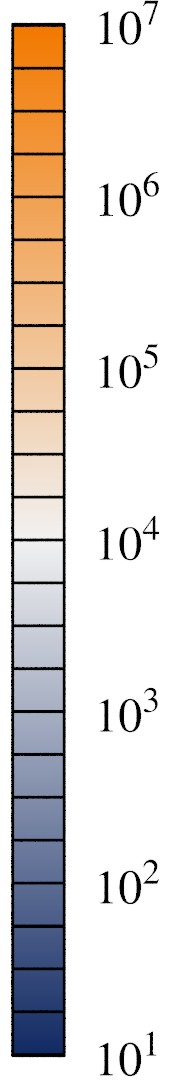}%
    \hfill 
    \inca{Figure_Bob/bob_sphere_BAEMM_16inc_TPM0_3pi_4pi_remeshed_tau_2/Reconstructed_TP_Density}{c}%
    \incb{Figure_Bob/bob_sphere_BAEMM_16inc_TPM0_3pi_4pi_remeshed_tau_2/Reconstructed_TP_Density_Sliced}{d}%
    \caption{%
    (a), (b)
    Our reconstruction of \emph{Bob} when using a sphere as initial guess.
    Even when we start from an incorrect topology, we obtain an interpretable result. 
    The thin double layer spanning the hole in the torus can easily be detected, e.g., by analyzing the \emph{tangent-point density} $\xi \mapsto \int_\varSigma r_\TP(\xi,\eta)^{-p} \dd \HsdM^{2}(\eta)$ as visualized in the logarithmic plots in (c) and (d).
    }%
    \label{fig:BobSphere}
\end{figure}

The tangent-point energy and our optimization pipeline do not allow for changing the isotopy type of the surface during optimization. 
So in principle, we have to know the topology a~priori.
The requisite guess could be performed, for instance, by using sampling methods as proposed in \cite{zbMATH02221668,sampling}. 
Nonetheless, it may still happen that a wrong initial guess for the topology is chosen.
An example for this is shown in \cref{fig:BobSphere}, where we started from a (topologically incorrect) sphere as initial guess.
Obviously, we cannot expect full convergence of the far field in this case. So the stopping rule by the discrepancy principle may never become active. 
Hence, we might have to stop prematurely.
Nonetheless, the reconstruction returns an interpretable result.

\subsection{Reduced data availability}

So far the directions of the incoming waves were distributed quite uniformly over the unit sphere. 
In practice, it might only be possible to use waves from a restricted set of directions.
In \cref{fig:BunnyQuadrant} we present an experiment with such reduced data input:
We chose $8$ incident wave directions condensed in the same quadrant to simulate a more realistic measuring situation.
Moreover, we chose a relatively high noise-level of $10\%$ and evaluated the far field at $162$ points on~$\S^2$.

To take this even further, in the example shown in \cref{fig:BunnyQuadrant_2} we increased the noise-level to $40\%$ and reduced the number of evaluation points to $64$.
This time, we use the combination of wavenumbers $\kappa = \braces{ 4\,\uppi,6\,\uppi}$.
Despite the highly corrupted data, our algorithm still yields a reasonable reconstruction, therefore demonstrating its robustness.


\begin{figure}[!ht]%
    \centering%
    \begin{tikzpicture}%
        \node[inner sep=0pt] (fig) at (0,0) {\includegraphics[	
            trim = 0 0 0 0, 
            clip = true,  
            angle = 0,
            width = 0.24\textwidth
        ]{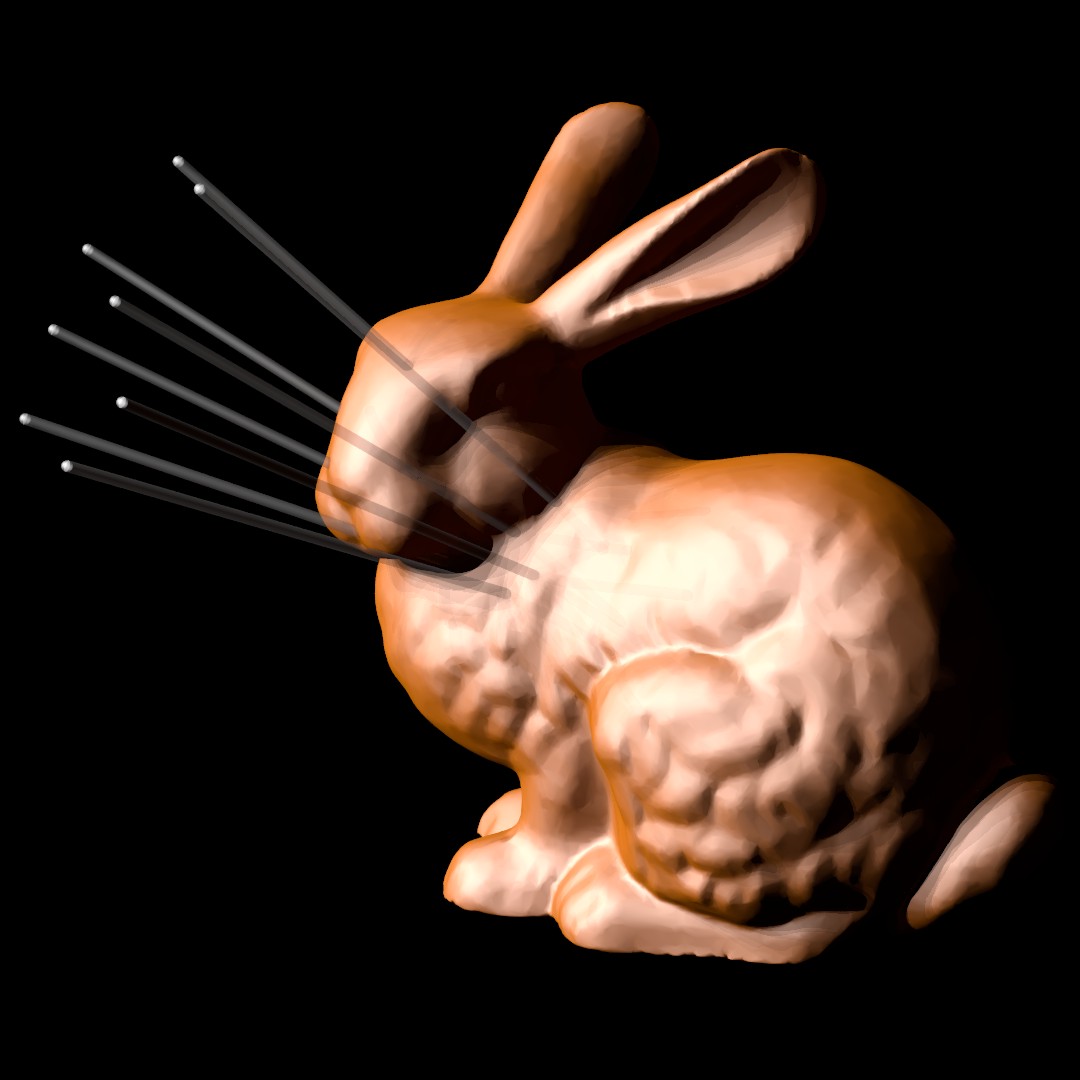}};%
        \node[above right = 0ex] at (fig.south west) {\begin{footnotesize}{\color{White}(a)}\end{footnotesize}};%
    \end{tikzpicture}%
    \begin{tikzpicture}%
        \node[inner sep=0pt] (fig) at (0,0) {\includegraphics[	
            trim = 0 0 0 0, 
            clip = true,  
            angle = 0,
            width = 0.24\textwidth
        ]{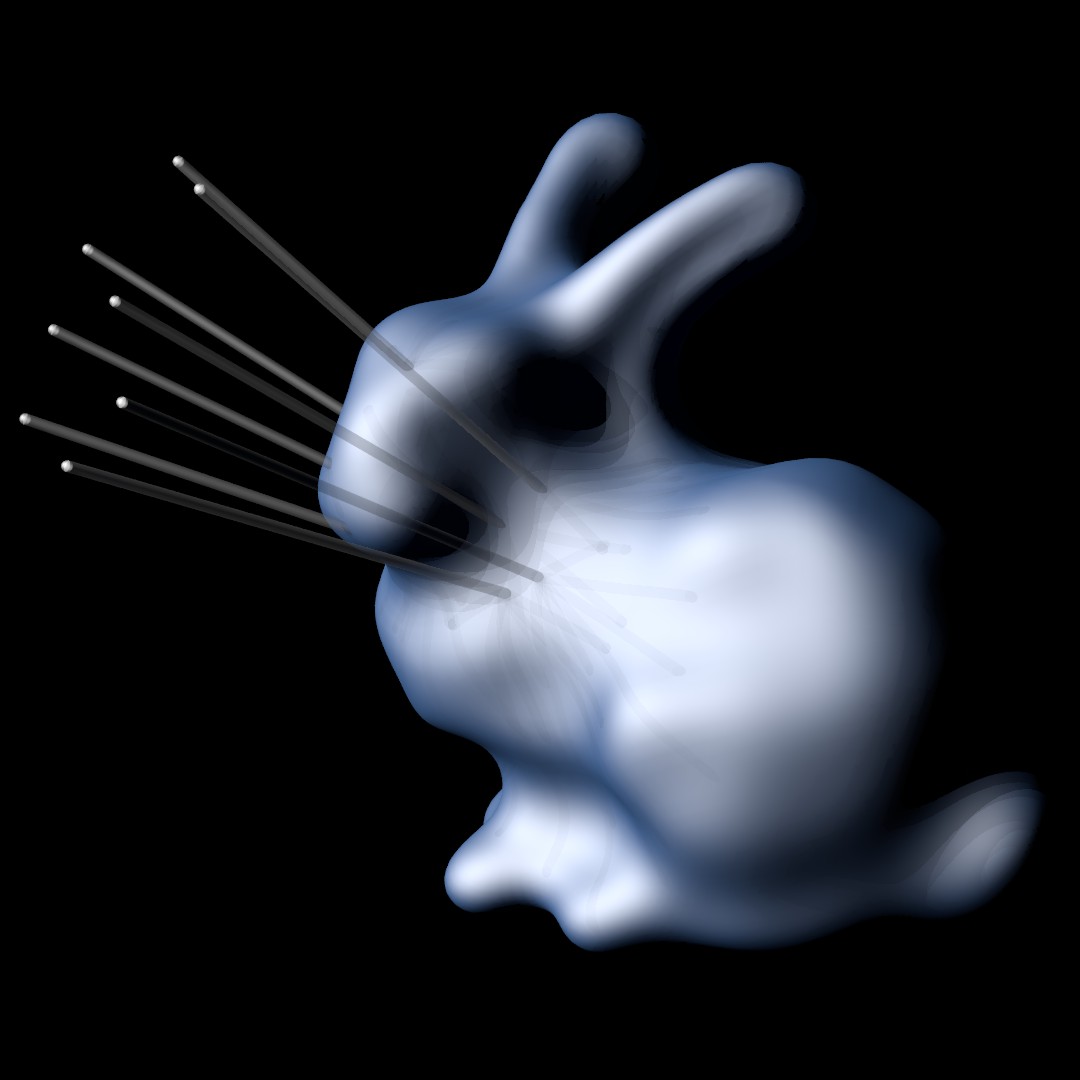}};%
        \node[above right = 0ex] at (fig.south west) {\begin{footnotesize}{\color{White}(b)}\end{footnotesize}};%
    \end{tikzpicture}%
     \begin{tikzpicture}%
        \node[inner sep=0pt] (fig) at (0,0) {\includegraphics[	
            trim = 60 40 90 120, 
            clip = true,  
            angle = 0,
            width = 0.24\textwidth
        ]{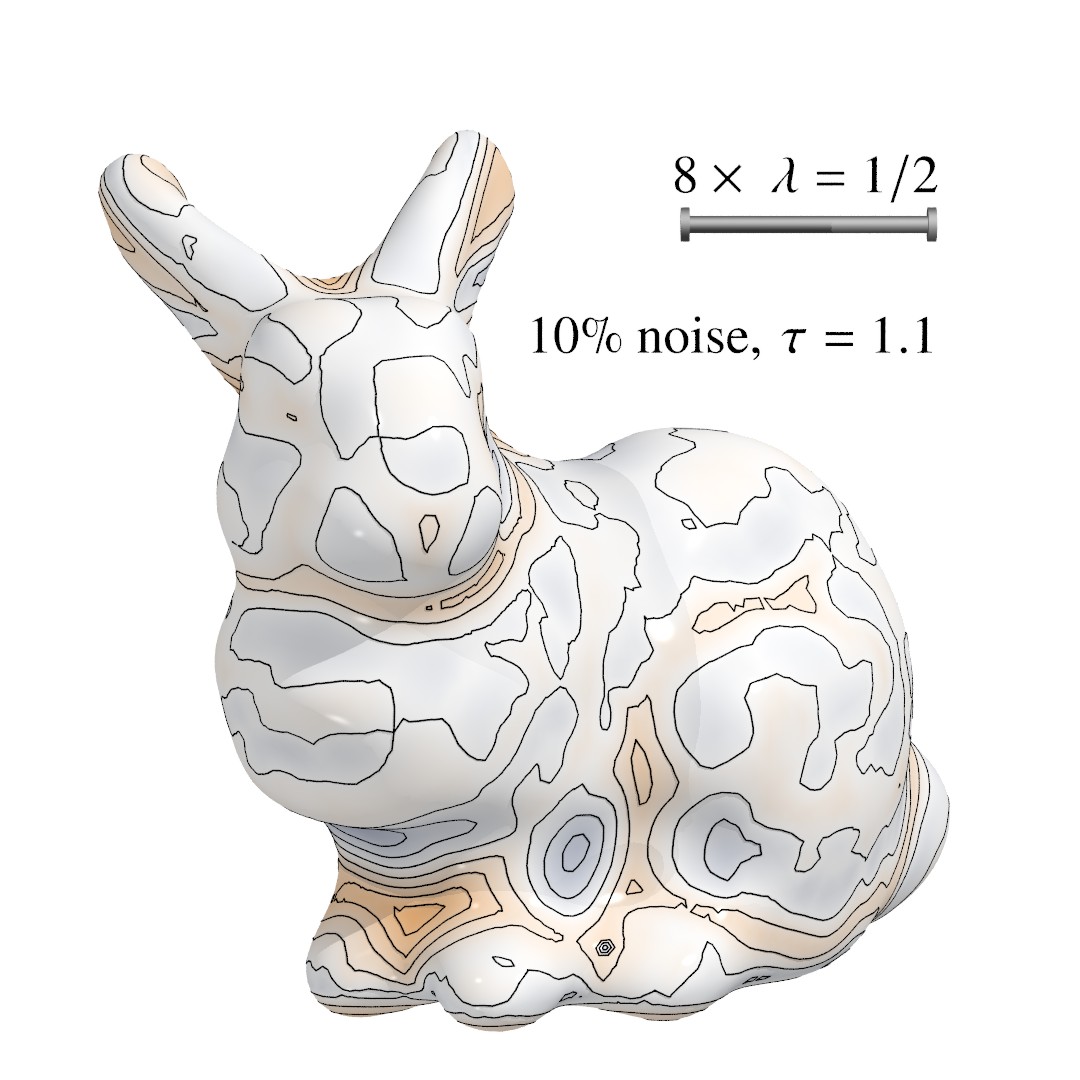}};%
        \node[above right= 0ex] at (fig.south west) {\begin{footnotesize}(c)\end{footnotesize}};%
    \end{tikzpicture}%
    \begin{tikzpicture}%
        \node[inner sep=0pt] (fig) at (0,0) {\includegraphics[	
            trim = 90 40 60 120, 
            clip = true,  
            angle = 0,
            width = 0.24\textwidth
        ]{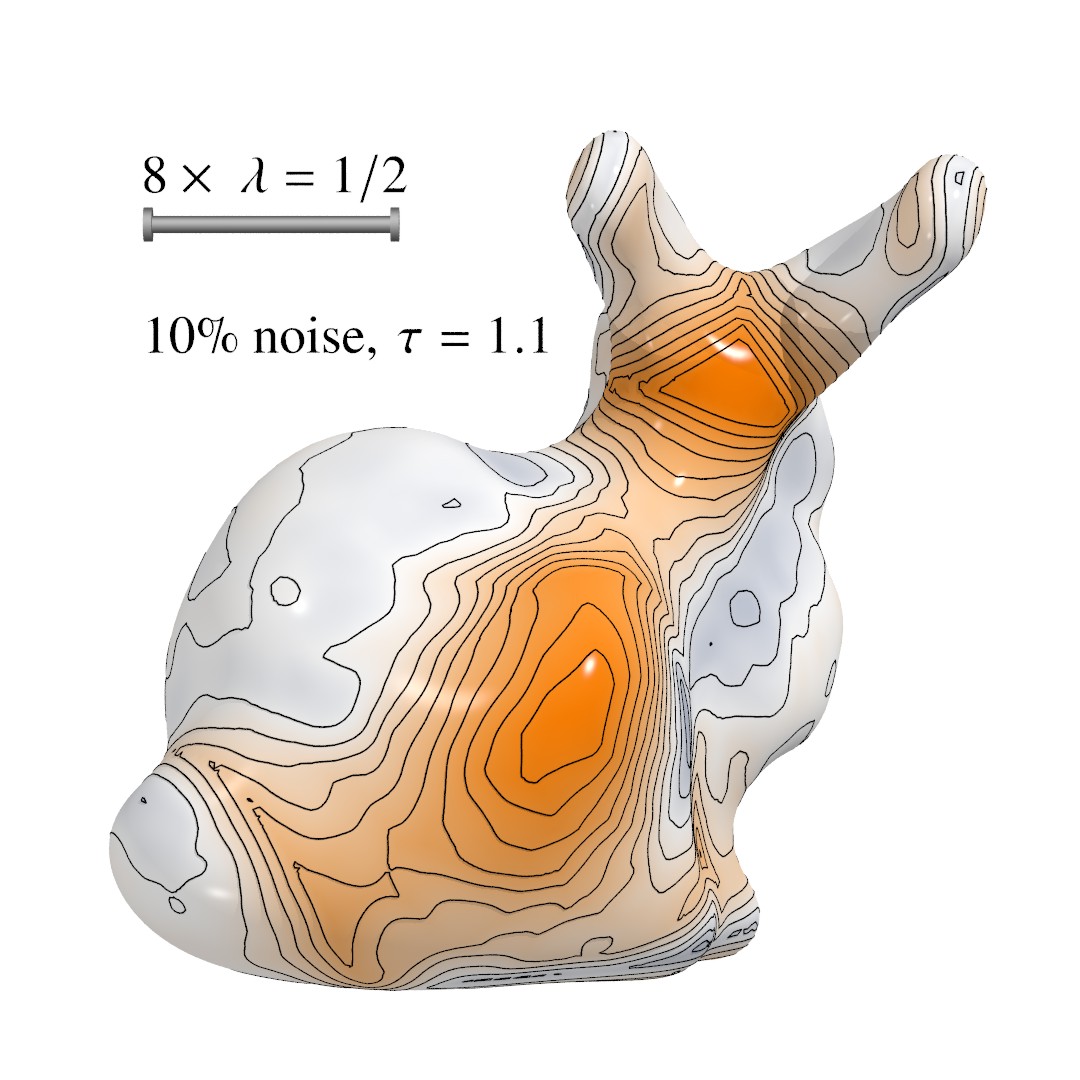}};%
        \node[above left = 0ex] at (fig.south east) {\begin{footnotesize}(d)\end{footnotesize}};%
    \end{tikzpicture}%
    \includegraphics[	
        trim = 0 0 0 0, 
        clip = true,  
        angle = 0,
        width = 0.040\textwidth
    ]{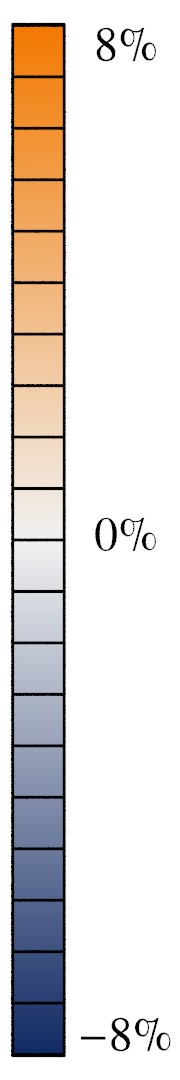}%
    \caption{%
        (a) The Standford bunny from \cref{fig:Bunny} as the obstacle for a focused bundle of $8$ incident wave directions.
        (b) Our reconstruction for far field-data evaluated at $162$ points and perturbed by $10\%$ white noise.
        (c), (d) The signed distance field of the true obstacle, plotted over the reconstructed surface.
    }%
    \label{fig:BunnyQuadrant}%
\end{figure}
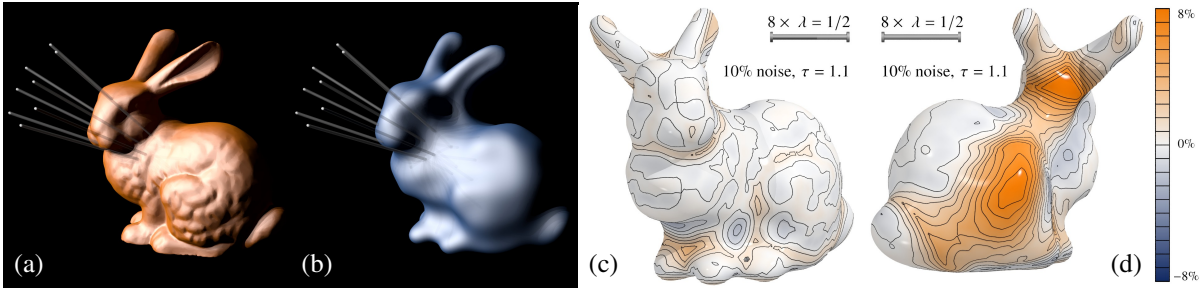


\begin{figure}[!ht]%
    \centering%
    \begin{tikzpicture}%
        \node[inner sep=0pt] (fig) at (0,0) {\includegraphics[	
            trim = 60 0 40 0, 
            clip = true,  
            angle = 0,
            width = 0.228\textwidth
        ]{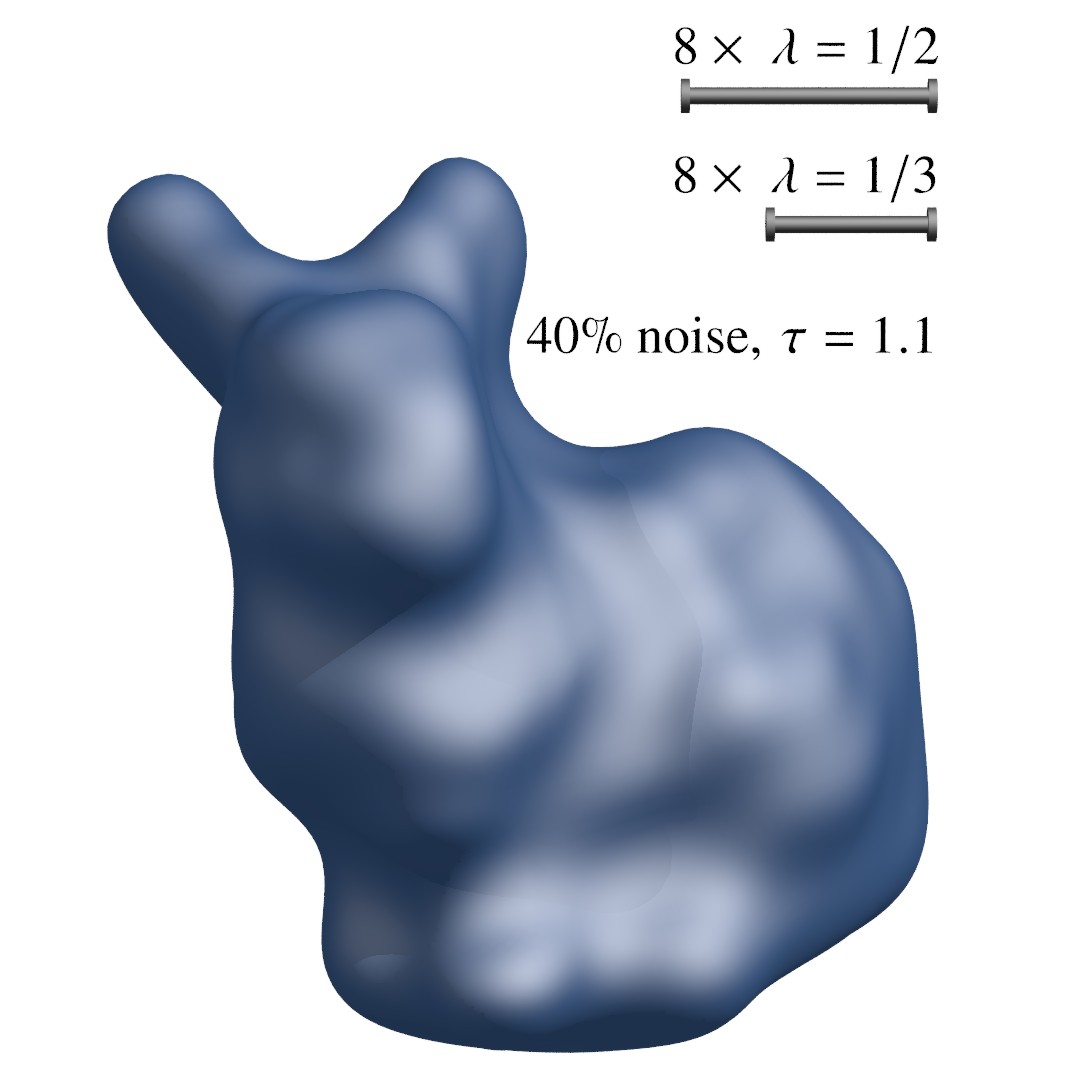}};%
        \node[above right = 0ex] at (fig.south west) {\begin{footnotesize}{(a)}\end{footnotesize}};%
    \end{tikzpicture}%
    \begin{tikzpicture}%
        \node[inner sep=0pt] (fig) at (0,0) {\includegraphics[	
            trim = 60 0 40 0, 
            clip = true,  
            angle = 0,
            width = 0.228\textwidth
        ]{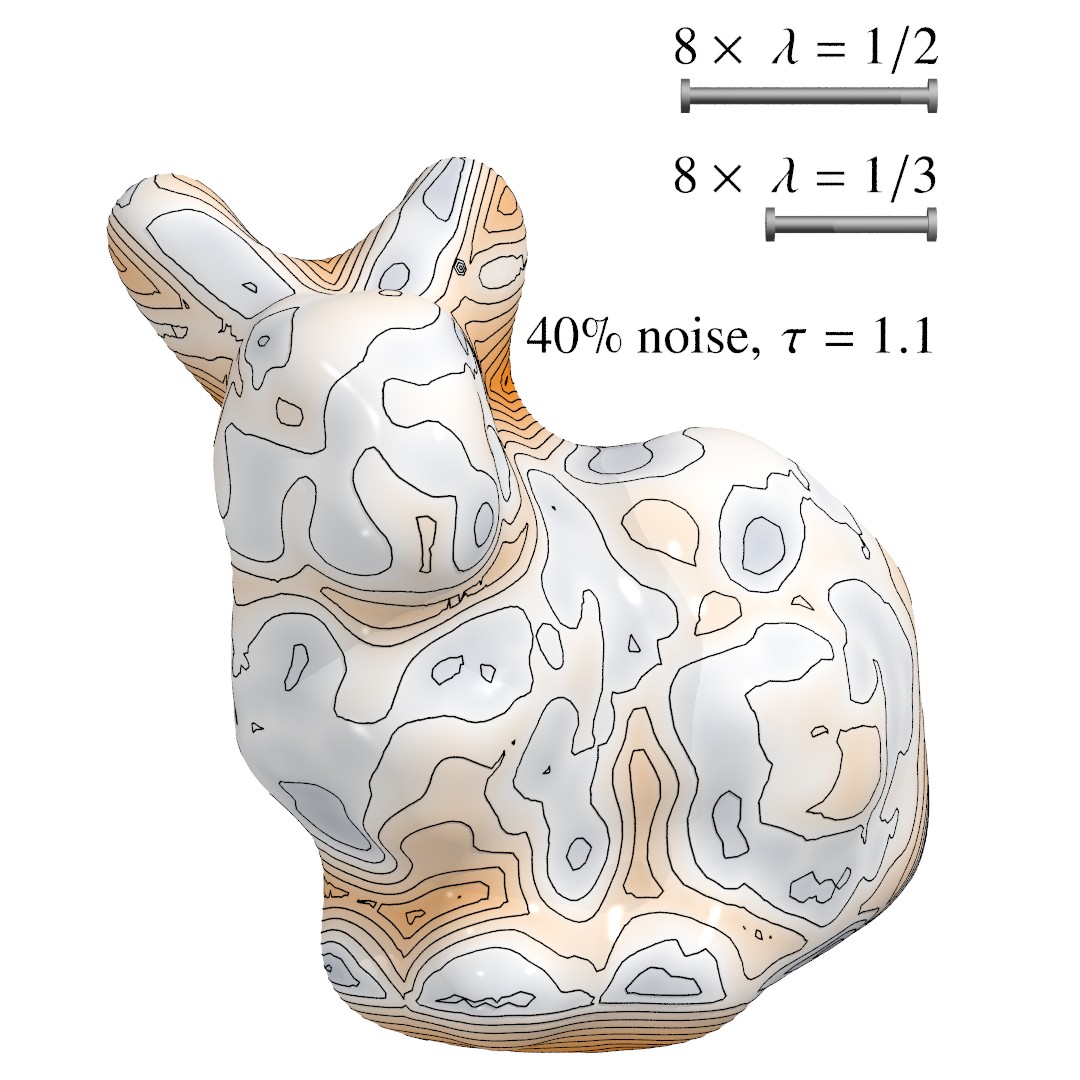}};%
        \node[above right = 0ex] at (fig.south west) {\begin{footnotesize}{(b)}\end{footnotesize}};%
    \end{tikzpicture}%
     \begin{tikzpicture}%
        \node[inner sep=0pt] (fig) at (0,0) {\includegraphics[	
            trim = 40 0 60 0, 
            clip = true,  
            angle = 0,
            width = 0.228\textwidth
        ]{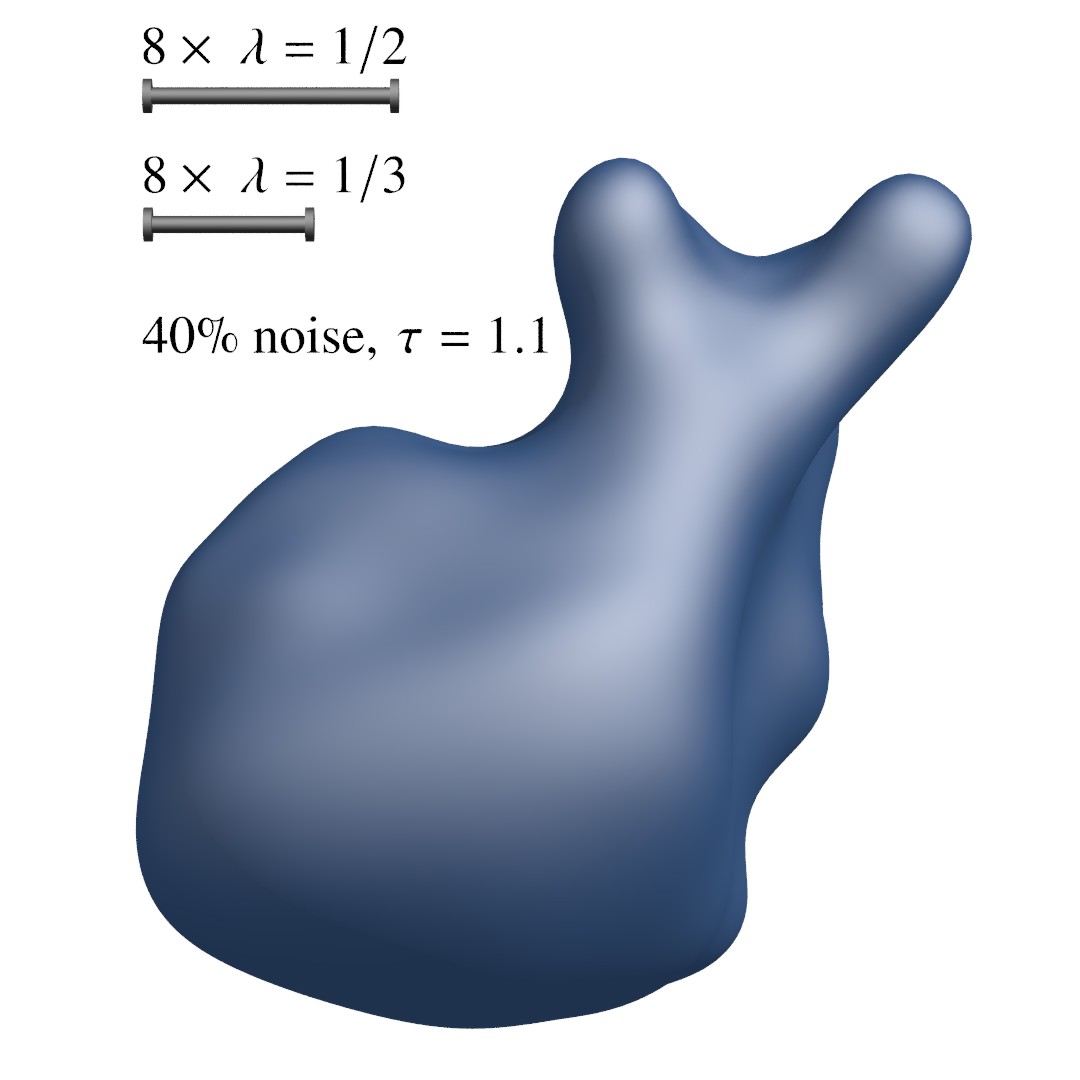}};%
        \node[above right= 0ex] at (fig.south east) {\begin{footnotesize}(c)\end{footnotesize}};%
    \end{tikzpicture}%
    \begin{tikzpicture}%
        \node[inner sep=0pt] (fig) at (0,0) {\includegraphics[	
            trim = 40 0 60 0, 
            clip = true,  
            angle = 0,
            width = 0.228\textwidth
        ]{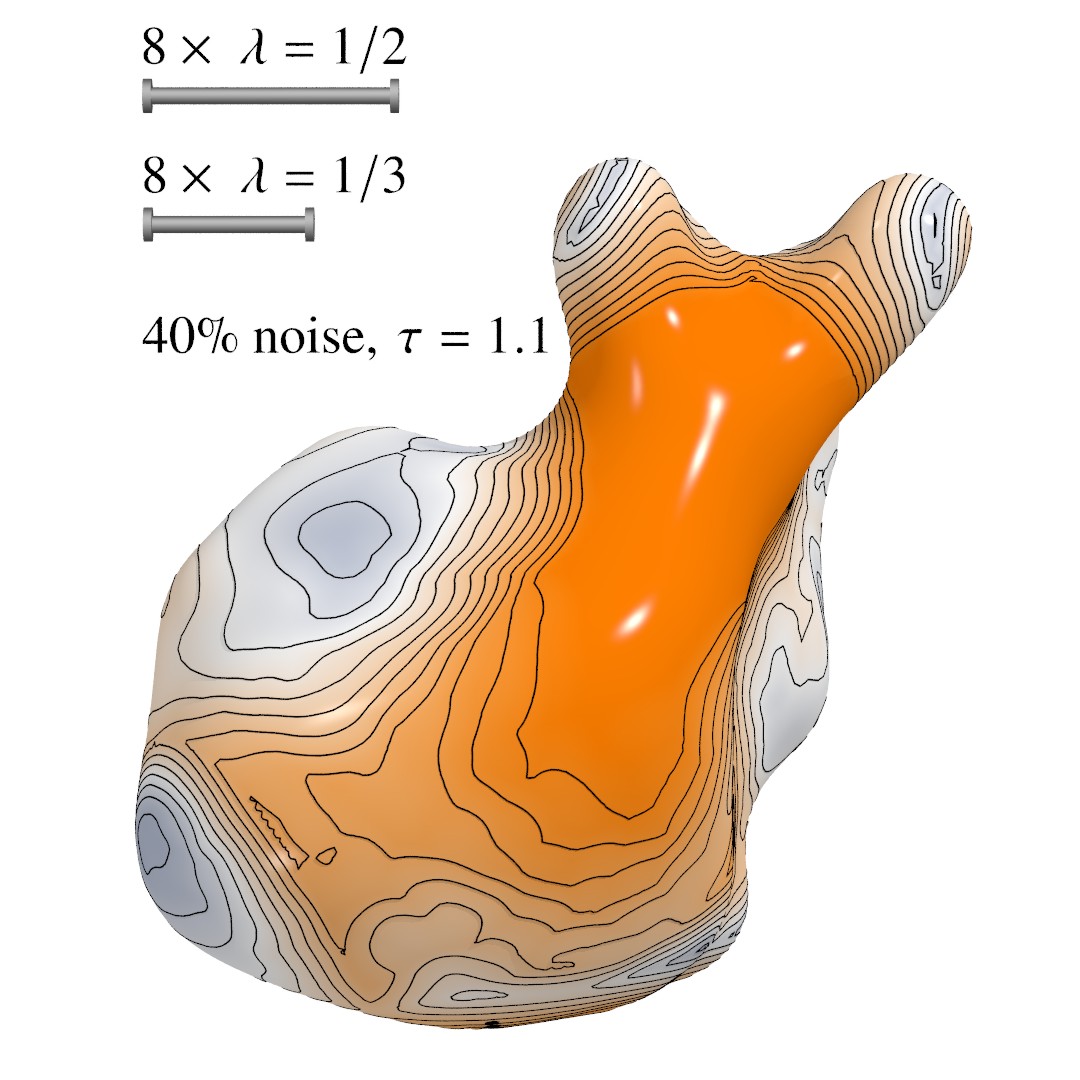}};%
        \node[above left = 0ex] at (fig.south east) {\begin{footnotesize}(d)\end{footnotesize}};%
    \end{tikzpicture}%
    \includegraphics[	
        trim = 0 0 0 0, 
        clip = true,  
        angle = 0,
        width = 0.040\textwidth
    ]{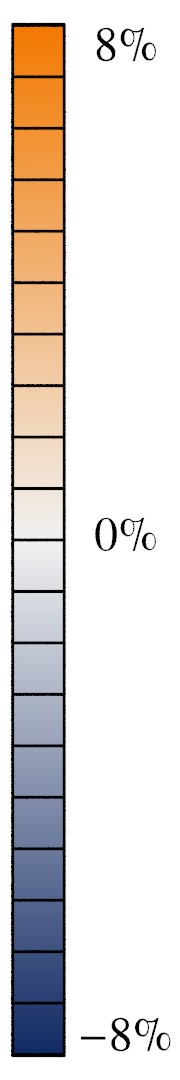}%
    \caption{%
    Consider the incoming directions from \cref{fig:BunnyQuadrant}.
        (a) The reconstruction with far field-data evaluated on $64$ points perturbed by $40\%$ white noise shown from the front.
        (b) The reconstruction from (a) with the signed distance field of the true obstacle, plotted over the reconstructed surface.
        (c), (d) The reconstruction from (a) and (b) shown from behind.
        Remarkably, despite the relatively high noise level and the incoming waves being focussed on the front, 
        our regularizer fills in the blind spot on the back in a plausible manner.
    }%
    \label{fig:BunnyQuadrant_2}%
\end{figure}

\appendix 
\crefalias{section}{appendix}

\appendix
\section{A product rule}\label{sec:ProductRule}

Product rules for \SoboSlobo spaces are well studied as special cases of product rules for Triebel--Lizorkin spaces and Besov spaces (see, e.g., \cite{zbMATH03594093}, \cite[Chapters 4 and 5]{zbMATH00929821}, or \cite{zbMATH07447116}).
Alas, we did not find the case required in the proof of \cref{lem:UpdateDirectionIsLipschitzContinuous}. 
Because it involves a Hölder space $\Holder[0,\sigma][][\Domain][\R] = \Sobo[\sigma,\infty][][\Domain][\R]$, it escapes some typical approaches from interpolation theory or for functions on the full space $\DomSpace$. We give a very brief and elementary proof.

\begin{lemma}\label{lem:ProductRuleWspWtinfty}
    Let $\alpha, \tau \in \intervaloo{0,1}$, $\sigma \in \intervaloo{\tau,1}$, $p \in \intervalcc{1,\infty}$, let 
    $\Domain$ be a closed, $n$-dimensional, smooth manifold,
    and let $f \in \Holder[1,\alpha][][\Domain][\AmbSpace]$ be an embedding.
    Then there is a $C_f >0$ such that 
    \[
        \norm{u \cdot v}_{\Sobo[\tau,p](f)}
        \leq 
        C_f \norm{u}_{\Sobo[\sigma,\infty](f)} \, \norm{v}_{\Sobo[\tau,p](f)}
        \quad 
        \text{for all $u \in \Sobo[\sigma,\infty][][\Domain][\R]$ and $v \in \Sobo[\tau,p][][\Domain][\R]$}
        .
    \]
\end{lemma}
\begin{proof}
    By Hölder's inequality we have 
    $
        \norm{u \cdot v}_{\Lebesgue[p](f)}
        \leq 
        \norm{u}_{\Lebesgue[\infty](f)} \norm{u}_{\Lebesgue[p](f)}
        .
    $
    Using Hölder's inequality once more, we obtain
    \begin{align*}
        \seminorm{u \cdot v}_{\Sobo[s,p](f)}
        &=
        \norm*{
            (x,y) \mapsto 
            u(y) \, \sdfrac{v(y)- v(x)}{\abs{f(y)-f(x)}^\tau}
            +
            \sdfrac{u(y)- u(x)}{\abs{f(y)-f(x)}^\sigma}
            \, 
            \sdfrac{v(x)}{\abs{f(y)-f(x)}^{\tau-\sigma}}
        }_{\Lebesgue[p](\mu_f)}
        \\
        &\leq 
        \norm{u}_{\Lebesgue[\infty]}
        \seminorm{v}_{\Sobo[\tau,p](f)}
        +
        \seminorm{u}_{\Sobo[\sigma,\infty](f)}
        \, 
        I(v)
        ,
    \end{align*}
    where 
    \[
        I(v) \ceq \norm[\big]{
            (x,y) \mapsto  v(x) \abs{f(y)-f(x)}^{\sigma-\tau}
        }_{\Lebesgue[p](\mu_f)} 
        .
    \]
    If $p = \infty$, then we may exploit that $\sigma-\tau > 0$ and that $\Domain$ is bounded
    to obtain
    $
    I(v) \leq 
        \diam(f(\Domain))^{\sigma-\tau}
        \norm{v}_{\Lebesgue[p](f)}
    $.
    If $p < \infty$, then we have 
    \[
        I(v)
        =
        \pars[\Big]{
            \textstyle
            \int_{\Domain} 
                \abs{v(x)}
                \,
                \varPsi(x)
            \dvol_f(x)
        }^{1/p}
        \quad 
        \text{with}
        \quad
        \varPsi(x)
        \ceq 
        \textstyle
        \int_{\Domain}
            \abs{f(y)-f(x)}^{p (\sigma-\tau) - \DomDim}
        \dvol_f(y)
        .
    \]
    Finally, we have $\norm{\varPsi}_{\Lebesgue[\infty]} < \infty$  because $\Domain$ is bounded, because $\dim(\Domain) = \DomDim$, and because $p (\sigma-\tau) - \DomDim > - \DomDim$.
    (See also \cref{lem:RfsIsSmoothLowOrder}, where we used a similar argument.)
    Hence, we get
    $
        I(v) \leq \norm{\varPsi}_{\Lebesgue[\infty](f)} \norm{v}_{\Lebesgue[p](f)},
    $
    which completes the proof.
\end{proof}

\section*{Acknowledgments}

We would like to thank Keenan Crane for putting the meshes \emph{Spot}, \emph{Bob}, and \emph{Blub} into the Public Domain and for making them available in \cite{Crane_Meshes}.
Moreover, we would like to express our gratitude towards Philipp Reiter who granted us access to computational resources, in particular to the NVIDIA~A40 GPUs of the Faculty of Mathematics at Chemnitz University of Technology.

\section*{Funding}

Jannik Rönsch has been supported by the Research Training Group 2088 \emph{Discovering structure in complex data: Statistics meets Optimization and Inverse Problems}, funded by the Deutsche Forschungsgemeinschaft (DFG, German Research Foundation).
Henrik Schumacher has been partially supported by the Research Training Group 2326 \emph{Energy, Entropy, and Dissipative Dynamics}, funded by the Deutsche Forschungsgemeinschaft (DFG, German Research Foun\-da\-tion).

\bibliographystyle{abbrvhref}

{\footnotesize
\setlength{\parskip}{0.001ex}
\bibliography{references}
}

\end{document}